%% file: bayeseiv.tex
\newif\ifrenderbboxes\renderbboxestrue\newif\ifrendertodo\rendertodotrue%
\documentclass{article}

\let\so\o%

\usepackage{StefansCommandos}
\begin{document}

\title{Bernstein-Von Mises for semiparametric mixtures}
\author{Stefan Franssen \and Jeanne Nguyen \and Aad van der Vaart}
\maketitle








\input{./Introduction/introduction.tex}
\input{./Introduction/Mixtures.tex}

\input{./Introduction/Efficiency.tex}
\input{./Introduction/EarlierWork.tex}
\input{./Introduction/results.tex}

\input{./Consistency/Consistency.tex}
\input{./BernsteinVonMises/SemiparametricBvM.tex}

\input{./BernsteinVonMises/LikelihoodExpansion.tex}

\input{./BernsteinVonMises/SemiparametricBvMDP.tex}

\input{Priors/priors.tex}


%
%


\input{Frailty/Frailty.tex}

\input{ErrorsInVars/ErrorsInVars.tex}

\input{Discussion/Conclusion.tex}
\let\o\so%

\bibliographystyle{alpha}
\bibliography{StefansBib}

\appendix

\input{Frailty/appendix.tex}

\input{ErrorsInVars/appendix.tex}

\end{document}

%% file: Introduction/introduction.tex
\section{Introduction}

Statistical modelling often involves striking a delicate balance between flexibility and simplicity. 
While parametric models offer a straightforward representation of data, they may fail to capture the intricate complexities present in real-world phenomena. 
On the other hand, nonparametric models provide remarkable flexibility, but their complexity can also be a drawback.

Semiparametric modelling is a compromise that tries to combine the strengths of parametric and nonparametric models.
(Strict) semiparametric models incorporate both a flexible nonparametric component and a simpler parametric component.
Typically, the parameter of interest is part of the parametric component, and the nuisance parameters are part of the nonparametric components.
The advantage of semiparametric models over nonparametric models is that it typically is easier to give an interpretation to the parametric part of the model.
Moreover, they typically are able to achieve at $n^{-1/2}$ rate, which is a parametric rate of learning, which is faster than most nonparametric models.
They do this while being more flexible and robust that parametric models, by allowing the nuisance part of the parameter to be modelled fully nonparametrically.
This gives the robustness against misspecification of nonparametric models while learning as fast as parametric models.

There are several distinct ways of creating estimators for a quantity of interest.
Ideally, we use the most accurate estimator possible, and have a method of quantifying its uncertainty.
Examples of estimators which are known to be asymptotically the most accurate are the one-step estimator and the maximum likelihood estimator. 
However, these methods do not provide uncertainty quantification in the method, one needs to use other methods like the bootstrap.
Bayesian methods do provide automatic uncertainty quantification in the form of credible sets.
However, these credible sets may fail to provide valid frequentist coverage, and the Bayesian posterior might not converge at an optimal rate.
One way of proving that the Bayesian posterior is optimal and has optimal and valid frequentist coverage is the Bernstein-von Mises theorem (BvM).
In this paper, we will prove such a Bernstein-von Mises theorem for semiparametric mixtures.
This means that practitioners can use this result to show that the Bayesian method has not only Bayesian but also frequentist interpretability.

%% file: Introduction/Mixtures.tex
\subsection{Mixture models}
Suppose we observe a random sample from a mixture density of the form
\begin{equation}
  x\mapsto p_{\q,F}(x):=\int p_\q(x\given z)\, dF(z).\label{eq:mixingdensity}
\end{equation}
The kernel $x\mapsto p_\q(x\given z)$ is the density of a random variable $X$ given 
a latent random variable $Z$, which has marginal distribution $F$. 
The kernel is indexed by a parameter  $\q\in\RR^d$ that we  wish to  estimate.

Equivalently, we can consider this as a latent variable model. 
We have some latent variable $Z \sim F$ that we do not observe, and the distribution of $X \given Z \sim p_\theta(x \given Z)$.
Because we do not observe $Z$, we might lose information for estimating $\theta$.
Because latent variables have widespread usage in statistics, semiparametric mixture models have a widespread application throughout statistics.
In this paper we will be considering two examples: exponential frailty and the errors-in-variables models.

Exponential frailty models are a particular case of frailty models.
We observe pairs of exponential distributed random variables $(X_i, Y_i)$ where $X_i |Z_i$ and $Y_i | Z_i$ are independent exponential random variables with rates $Z_i$ and $\theta Z_i$ respectively.
Here $Z_i$ is a  latent random variable which follows an unknown distribution $F$.
They are mixtures given by
\[
    p_{\theta, F}(x,y) = \int z^ 2 \theta e^{-z(x + \theta y)} \dd F(z)
\]
The parameter of interest is $\theta$, so this naturally forms a semiparametric mixture model.
We will consider this model in \Cref{sec:frailty}.
The frailty models are used in biomedical research.
Semiparametric Bayesian frailty models have been considered in~\cite{walkerHierarchicalGeneralizedLinear1997,pennellBayesianSemiparametricDynamic2006,tallaritaBayesianAutoregressiveFrailty2019}.
A frequentist analysis of the semiparametric exponential frailty model with nonparametric maximum likelihood estimator has been done in~\cite{vandervaartEfficientMaximumLikelihood1996}.
See for examples of general frailty models~\cite{balanTutorialFrailtyModels2020}.

The errors-in-variables model is the other example we will consider. 
In this model we want to run a regression of $Y$ on $Z$, however, we only observe $Y$ and a noisy version of $Z$.
This means that $Z$ is a latent random variable which we do not observe, only a noisy version of this.
The mixture is given by
\[
    p_{\alpha, \beta, \sigma, \tau, F}(x,y) = \int \phi_\sigma(x - z) \phi_\tau(y - \alpha - \beta z) \dd F(z),
\]
where $\phi$ is the density of a standard normal Gaussian distribution.
We are interested in estimating $\alpha, \beta$, while we consider $F$ as a nuisance parameter.
Thus, Errors-in-Variables is another semiparametric mixture model.
This model is more of a fundamental nature, being able to do linear regression in the situation where the covariates are observed with noise.
For more details on the model see \Cref{sec:ErrorsInVars}.
Bayesian errors-in-variables have been considered before, see for example~\cite{dellaportasBayesianAnalysisErrorsVariables1995,mallickSemiparametricErrorsvariablesModels1996,mullerBayesianSemiparametricModel1997}.

Now we want to construct estimators which are efficient.
One possible estimator is the full maximum likelihood estimator, which given a sample $X_1,\ldots, X_n$ from the distribution of $X$, maximizes the likelihood
\[
  (\q,F)\mapsto \prodin \int p_\q(X_i\given z)\,dF(z). 
\]
These estimators were shown to be consistent in some generality in~\cite{kieferConsistencyMaximumLikelihood1956}.
For a few special models the $\q$-component of the maximum likelihood estimator is known to be asymptotically normal and efficient from~\cite{vandervaartEstimatingRealParameter1988,murphyLikelihoodInferenceErrorsVariables1996,vandervaartEfficientMaximumLikelihood1996}. 

A Bayesian alternative is to equip $F$ with a \emph{species sampling prior} (See e.g.~\cite[Chapter 14]{ghosalFundamentalsNonparametricBayesian2017}), and, independently $\theta$ with a smooth prior.
Two common choices for a mixture prior would be the Dirichlet process prior or a finite discrete prior with a random number of components.
These two choices are the priors we will study in our examples.

The Dirichlet process is a stochastic process characterised by two parameters, the prior precision $M$ and the centre measure $G$.
Then a random variable $F$ is said to be Dirichlet process distributed if for any partition $A_1, \dots A_k$
\[
    (F(A_1),\dots, F(A_k)) \sim \textrm{Dir}\left(k; MG(A_1),\dots, MG(A_k)\right).
\]

As an example of a finitely discrete prior with a random number of components, pick $\lambda, \sigma > 0$ and $G$ a probability distribution.
Then we draw $K \sim \textrm{Geom}(\lambda)$, and given $K$, draw $Z_1, \dots, Z_K \overset{\textrm{i.i.d.}}\sim G$, $W \sim \textrm{Dir}(K; \sigma,\dots, \sigma)$.
Then 
\[
    F = \sum_{k = 1}^K W_k \delta_{Z_k}
\]
We will give more details on these priors in \cref{sec:priors}. 
In that section, we will also recall the facts that are needed for the purpose of this paper.

More generally, a species sampling process can be defined using the exchangeable partition probability function (EPPF) and a centre measure $G$.
There are many more examples of species sampling process priors, for example the Pitman-Yor processes, Gibbs processes or Poisson-Kingman processes.
For more details, see~\cite[Chapter 14]{ghosalFundamentalsNonparametricBayesian2017}.

%% file: Introduction/Efficiency.tex
\subsection{Semiparametric efficiency}

The goal of the statistician is to provide a reliable estimate for the parameter of inference $\theta$, and specify its uncertainty.
We want to learn $\theta$ as accurately as possible.
The first notion we need to guarantee is consistency, that is, the ability to find the true parameter at all.
After having established consistency, we want to understand if our method is asymptotically optimal.
\emph{Efficiency} will be how we will measure asymptotic optimality. 
Very roughly, this means that the posterior is as accurate as possible: there is no estimator which can be more accurate with the same amount of data.
To be a bit more precise, no other estimator can have lower squared estimation error, asymptotically.

In the study of optimal statistical estimation, one key notion is the Fisher information.
In semiparametric statistics, we often deal with a \emph{loss of information} and have to consider the efficient Fisher information.
We call a model differentiable in quadratic mean at $\theta_0$ if
\[
    \int \left[ \sqrt{p_\theta(x)} - \sqrt{p_{\theta_0}(x)} - \frac{1}{2} \dot\ell_{\theta_0}(x) \sqrt{p_{\theta_0}(x)} \right]^2 \dd \mu(x) = o(\| \theta - \theta_0 \|),
\]
as $\theta \rightarrow \theta_0$. 
The Fisher information then is given by $I_{\theta_0} = P_{\theta_0} \dot\ell_{\theta_0}\dot\ell_{\theta_0}^T$.
Our goal is to prove that certain estimators are optimal.
To do so, we must first establish a lower bound on how accurate estimators can be.
For this we use two ingredients,  Anderson's lemma~\cite[Lemma 8.5]{vandervaartAsymptoticStatistics1998} and The Convolution theorem~\cite[Lemma 8.8]{vandervaartAsymptoticStatistics1998}.
Let $l$ be a bowl-shaped loss-function (that is, $\{ x : l(x) \leq c \}$ are convex and symmetric around the origin, $l$ is any function to $[0, \infty)$), and $T_n$ a sequence of regular estimators.
By the convolution theorem
\[
    \sqrt{n}\left(T_n - \theta_0 \right) \rightsquigarrow N(0, I_{\theta_0}^{-1} ) \ast M_{\theta_0}.
\]
Then by Anderson's lemma
\[
    \int l \dd N(0, I_{\theta_0}) \ast M_{\theta_0} \geq \int l \dd N(0, I_{\theta_0}).
\]
This implies that for bowl-shaped loss functions, no regular estimator can be more accurate than those which achieve a mean zero Gaussian with the inverse Fisher information as variance.
We call estimators which achieve this, efficient.

Under regularity conditions, the Maximum Likelihood estimator is an efficient estimator.
For parametric Bayesian methods, we have a similar result called the (parametric) Bernstein-von Mises theorem, see for example~\cite[Lemma 10.1 and 10.2]{vandervaartAsymptoticStatistics1998}.
This result describes the asymptotic shape of the posterior.
Under regularity conditions, this shape will be a Gaussian distribution centred at an efficient estimator and with the inverse Fisher information as the variance.

The Bernstein-von Mises theorem is one tool that we can use to study the frequentist properties of Bayesian procedures.
By establishing the asymptotic shape, we can prove properties about the posterior, such as consistency, contraction rates and the validity of certain credible sets as asymptotic confidence sets.
When the asymptotic shape is a Gaussian with as mean an efficient estimator and as variance the inverse Fisher information, the posterior will be asymptotically as accurate as possible.
See for example~\cite[lemma 6.7 and 8.7]{ghosalFundamentalsNonparametricBayesian2017} and~\cite[Section 1.4.3.1]{franssenWhenSubjectiveObjective} using properties of the posterior distribution to give frequentist properties.

To study optimality of estimators in semiparametric statistics we want to use similar strategies.
In strict semiparametric models, models are index by two parameters, $\theta$ the parameter of interest and $F$ the nuisance parameter.
Because the nuisance parameter might be infinite dimensional, we can't simply take derivatives as in the Euclidean setting.
One strategy is to take finite dimensional submodels that are differentiable in quadratic mean, compute the Fisher information and find the infimum.
A submodel for which this infimum is taken is called a \emph{least favourable submodel}.
The information for the whole model is bounded from above by the information in the least favourable submodel.
The score for the least favourable submodel is taken is called the \emph{efficient score function}, and its covariance matrix is called the \emph{efficient Fisher information matrix}.

Just as in the parametric case we can prove theorems that provide lower bounds for the accuracy of semiparametric estimators.
In the same way as before, the semiparametric convolution theorem~\cite[Theorem 25.20]{vandervaartAsymptoticStatistics1998} together with Anderson's lemma reveal the best asymptotic distribution: mean zero Gaussian with the inverse efficient Fisher information as variance.
We call estimators which achieve this lower bound \emph{efficient estimators}.

%% file: Introduction/EarlierWork.tex
\subsection{Earlier work}
The maximum likelihood estimator is known to be consistent in fair generality due to the work by Kiefer and Wolfowitz~\cite{kieferConsistencyMaximumLikelihood1956}.
However, efficiency of the maximum likelihood estimator is only known in special models. 
Van der Vaart proved~\cite{vandervaartEfficientMaximumLikelihood1996} for special models that the maximum likelihood estimator is efficient.
An extension of the Errors-in-Variables model was studied by Murphy and Van der Vaart~\cite{murphyLikelihoodInferenceErrorsVariables1996}.
They also studied profile likelihood estimators~\cite{murphyProfileLikelihood2000}.

On the Bayesian side, Castillo~\cite{castilloSemiparametricBernsteinMises2012} derived conditions for the Bernstein-von Mises theorem to hold for specific semiparametric models and Gaussian process priors.
This work has later been extended by~\cite{bickelSemiparametricBernsteinMises2012,castilloBernsteinMisesPhenomenonNonparametric2014,castilloBernsteinMisesTheorem2015,raySemiparametricBayesianCausal2020,chaeSemiParametricBernsteinMisesTheorem2019,chaeSemiparametricBernsteinMisesTheorem2015}.
In these works, the authors prove the Bernstein-von Mises result under two assumptions.
The first of these assumptions is a locally asymptotically normal (LAN) expansion of the likelihood in a perturbation.
This condition closely resembles the assumptions needed to make the MLE efficient.
The second of these assumptions is a change of measure condition.
This condition roughly states that the prior and the perturbation interact nicely.

We will bring extra attention to the works by Chae~\cite{chaeSemiParametricBernsteinMisesTheorem2019,chaeSemiparametricBernsteinMisesTheorem2015}.
In these works, they prove a Bernstein-von Mises theorem for special semiparametric mixture models.
These are symmetric location mixtures, which are special since the least favourable submodel does not involve the nuisance parameter.
This means that the change of measure condition is a lot easier to analyse compared to our situation.

%% file: Introduction/results.tex
\subsection{Our Results}
Our first result concerns the consistency of posterior distributions.
We call a posterior distribution consistent if for every neighbourhood $U$ of $\theta$, the posterior assigns asymptotically zero mass to the complement of $U$.
For Bayesian semiparametric mixture models, consistency, the ability to concentrate all mass near the true parameter was unknown.
In \cref{sec:consistency} we provide two tools to show that the posterior distribution is consistent in mixture models. 
The first is based on the seminal works by Kiefer and Wolfowitz~\cite{kieferConsistencyMaximumLikelihood1956} and Schwartz~\cite{schwartzBayesProcedures1965} (see also~\cite[Lemma 6.16 and 6.17]{ghosalFundamentalsNonparametricBayesian2017}).
The second technique uses the Glivenko-Cantelli theorem and Schwartz's theorem again.

To show the asymptotic optimality of the Bayesian methodology, we prove a Bernstein-von Mises theorem for $\theta$.
Our work can be seen as a variant of the proof by Castillo and Rousseau~\cite{castilloBernsteinMisesTheorem2015}.

Consistency allows us to localise our assumptions to a neighbourhood of the truth.
This allows us to tackle the LAN expansion and the change of measure condition locally.
In \cref{sec:semiparBvM} we formulate a precise theorem including the assumption on the LAN expansion and the change of measure.

To verify the LAN expansion we provide a lemma in terms of Donsker conditions in \cref{sec:LAN}.
There are advantages and disadvantages to this approach.
Donsker conditions are technical conditions that can require work to verify.
However, since the asymptotic efficiency for the MLE also relies on Donsker conditions, this work has often already been done in earlier papers.

Before this work, the main issue lay with verifying the second condition.
We had no tools to verify this for mixture models.
This was a fundamental gap in our understanding of the behaviour of Bayesian methods in these models.
The main difficulty lies in understanding how the least-favourable submodels interact with the prior on the nuisance parameter.
We impose some structural assumption on the least-favourable submodel, which allows us to analyse this interaction when the prior is a species sampling process prior.
We provide the tools to do so and verify the change of measure condition in \cref{sec:changeofmeasure}.

Finally, we verify these assumptions for two semiparametric mixture models.
In \cref{sec:frailty}, we study the frailty model, while in \cref{sec:ErrorsInVars}, we study the errors in variables model.

%% file: Consistency/Consistency.tex
\section{Posterior consistency}\label{sec:consistency}
In order to apply our theorems we first need to show that the consistency of the posterior distribution in fair generality, in order to handle parameters $\gamma = (\theta, F)$.
The next results can handle any form of parameter, not necessarily of the form $(\theta, F)$.
We only require weak assumptions (See \Cref{assumption_compactness,assumption_density,assumption_continuity,assumption_identifiable,assumption_measurable_sup}).
We provide two tools for that.
The first tool is based on combining two seminal results. 
The first result on which our work is based is Schwartz theorem~\cite{schwartzBayesProcedures1965} (see also the proof of~\cite[Lemma 6.16, 6.17]{ghosalFundamentalsNonparametricBayesian2017}).
This theorem is the standard result for showing that the posterior distribution is consistent.
The second result on which our work is based is the consistency theorem by Kiefer and Wolfowitz~\cite{kieferConsistencyMaximumLikelihood1956}.
This shows that the supremum of the likelihood ratio outside some small ball is exponentially small.
We will follow the discussion for the assumptions in the style of Kiefer-Wolfowitz. 
Let $X_1, \dots, X_n \given \gamma\iid p_{\gamma}$, and $\Pi$ be a prior on $\gamma$.
In the case of semiparametric mixtures, $\gamma = (\theta, F)$.

\begin{assumption}\label{assumption_compactness}
  $\Gamma$ is a pre-compact metric space with metric $d$.
\end{assumption}

\begin{assumption}\label{assumption_density}
  For all $\gamma \in \Gamma$, $p_\gamma$ is a density with respect to a $\sigma$-finite measure $\mu$.  
\end{assumption}

\begin{assumption}\label{assumption_continuity}
  The map $\gamma \mapsto p_\gamma(x)$ is continuous for all $x$ and can be continuously extended to the compactification of $\Gamma$.
\end{assumption}

\begin{assumption}\label{assumption_identifiable}
  For all $\gamma_0, \gamma_1$ there exists $y$ such that 
  \[
    \int_{-\infty}^y p_{\gamma_1}(x) \dd \mu 
    \neq 
    \int_{-\infty}^y p_{\gamma_0}(x) \dd \mu 
  \]
\end{assumption}

\begin{assumption}\label{assumption_measurable_sup}
  The function
  \[
    w_{\gamma, \rho}(x) = \sup_{\gamma'|  d(\gamma, \gamma') < \rho } p_{\gamma'}(x)
  \]
  is measurable.
\end{assumption}

If $\gamma_0 \in \Gamma$, we denote the distribution with density $p_\gamma$ by $P_0 := P_{\gamma_0}$.

\begin{assumption}\label{assumption_integrability}
  For any $\gamma \in \Gamma$, we have
  \[
    \lim_{\rho \downarrow 0}
    \EE_{\gamma_0}\left[
      \log
      \frac{
        w_{\gamma, \rho}
      }{
        p_{\gamma_0}
      }(X)
    \right]
    < \infty
  \]
\end{assumption}

\begin{remark}\label{remark_LR_bound}
  Kiefer and Wolfowitz show in~\cite{kieferConsistencyMaximumLikelihood1956} that under~\Cref{assumption_compactness,assumption_density,assumption_continuity,assumption_identifiable,assumption_measurable_sup,assumption_integrability} for every $\rho > 0$ there exists $h(\rho)$ such that $0< h(\rho) < 1$ such that eventually $[P_{\gamma_0}^\infty]$-almost surely
  \[
    \sup_{ \gamma | d(\gamma, \gamma_0)> \rho } \prod_{i = 1}^n \frac{p_\gamma}{p_{\gamma_0}} (X_i) < h(\rho)^n
  \]
\end{remark}

\begin{remark}\label{remark_Unormalized_posterior_mass}
  Following the proof of Schwartz theorem~\cite[theorem 6.17]{ghosalFundamentalsNonparametricBayesian2017} we can prove that for any $\gamma_0$ that belongs to the Kullback-Leibler support of $\Pi$, for all $c' > 0$, eventually $[P_{\gamma_0}^\infty]$-a.s.:
  \[
    \int \prod_{i = 1}^n \frac{p_\gamma}{p_{\gamma_0}}(X_i) \dd \Pi(\gamma) \gtrsim e^{-c'n}
  \]
\end{remark}

\begin{remark}\label{remark:ConsistencyPerturbedPosterior}
    We can replace the almost sure argument from~\cite[theorem 6.17]{ghosalFundamentalsNonparametricBayesian2017} by theorem 6.26 from the same reference. 
    This can be used to prove consistency for the perturbed models if this would be needed.
\end{remark}

Using the proof of Kiefer and Wolfowitz and Schwartz theorem, we can give a new consistency proof.
\begin{lem}\label{lem:posteriorconsistency}
  Suppose that $\gamma_0$ belongs to the Kullback-Leibler support of $\Pi$ and that \Cref{assumption_compactness,assumption_density,assumption_continuity,assumption_identifiable,assumption_measurable_sup,assumption_integrability} are satisfied. Then the posterior distribution is consistent at $\gamma_0$.
\end{lem}

\begin{proof}
Let $\mathcal{U}$ be any open neighbourhood of $\gamma_0$. 
Using \cref{remark_LR_bound} and that there exists a $\rho$ such that the ball of radius $\rho$ centred at $\gamma_0$ is contained in $\mathcal{U}$, they show that there exists an $h = h(\rho)$ such that $0 < h< 1$ and eventually $[P_{\gamma_0}^\infty]$-almost surely:
  \begin{equation}
    \sup_{\gamma \in \mathcal{U}^c} \prod_{i = 1}^n \frac{p_{\gamma}}{p_{\gamma_0}}(X_i) < h^n. \label{step_LR_bound}
  \end{equation}
  Furthermore, by \Cref{remark_Unormalized_posterior_mass}, for all $0< c < -\log(h)$ such that eventually a.s. $[P_{\gamma_0}^\infty]$:
  \begin{equation}
    \int \prod_{i = 1}^n  \frac{p_\gamma}{p_{\gamma_0}}(X_i) \dd \Pi(\gamma) \gtrsim e^{-cn}. \label{step_unormalized_posterior_mass_bound}
  \end{equation}
  By using \Cref{step_LR_bound,step_unormalized_posterior_mass_bound},
  \[
      \Pi(\mathcal{U}^c | X^{(n)})
    = 
      \frac{
        \int_{\mathcal{U}^c} \prod_{i = 1}^n \frac{p_\gamma}{p_{\gamma_0}}(X_i) \dd \Pi(\gamma)
      }{
        \int                 \prod_{i = 1}^n \frac{p_\gamma}{p_{\gamma_0}}(X_i) \dd \Pi(\gamma)
      }
      \lesssim
     e^{(\log(h) + c) n},
     \]
      then, because $\log(h) + c < 0$, it follows that $[P_{\gamma_0}^\infty]$-almost surely, 
      \[
     \lim_{n \rightarrow \infty} \Pi(\mathcal{U}^c | X^{(n)}) = 0.
     \]
\end{proof}

Another strategy for proving consistency uses Glivenko-Cantelli classes instead.
This gives a second strategy, which can be useful since in the frequentist semiparametric literature, the log-likelihood is often shown to be a Glivenko-Cantelli or Donsker class.
This means to prove consistency we only need to verify a prior mass condition.

\begin{lem}\label{lem:ConsistencyByGC}
  The posterior distribution $\Pi_n\left( \cdot \middle| X^{(n)} \right)$ in the model $X_1, \dots, X_n | \gamma \sim p_\gamma$ and $\gamma\sim \Pi$ is strongly consistent at $p_{\gamma_0}$ if all the following hold:
  \begin{itemize}
    \item for any open $U$, $\sup_{\gamma \in U^c} \text{KL}\left( p_{\gamma} | p_{\gamma} \right) > 0$;
    \item $\{ x\mapsto \log p_\gamma(x) | \gamma \in \Gamma \}$ is $P_{\gamma_0}$-Glivenko-Cantelli;
    \item $\gamma_0$ belongs to the Kullback-Leibler support of $\Pi$.
  \end{itemize}
\end{lem}

\begin{proof}
  Let $U$ be any open neighbourhood of $\theta_0$. 
  Denote $c = \sup_{\theta \in U^c} \text{KL}\left( p_{\theta} | p_{\theta_0} \right) > 0$. 
  Then for any $0< c' <c$, eventually almost surely $\sup_{\theta \in U^c} \left(\ell_n(\theta) - \ell_n(\theta_0)\right) \leq c'$, and therefore $\int_{U^c} \prod_{i = 1}^n \frac{p_{\theta}}{p_{\theta_0}} (X_i) \dd \Pi(\theta) \leq \int_{U^c} e^{-c'n} \dd \Pi(\theta) \leq e^{-c'n}$. 
  Now apply \Cref{remark_Unormalized_posterior_mass} and finish the proof as usual.
\end{proof}

These two results show that the posterior distribution in fair generality is consistent in semiparametric mixture models.
However, we wish to show more than mere consistency, we also want to study the asymptotic optimality of Bayesian methods.

%% file: BernsteinVonMises/Semiparametricbvm.tex
\section{Semiparametric Bernstein-von Mises Theorem}\label{sec:semiparBvM}
Let $X_1,\ldots, X_n$ be an i.i.d.\ sample from a density $p_{\q,F}$ relative to a $\s$-finite measure on 
a measurable space $(\X,\A)$, for parameters $\q\in\Theta\subset\RR^d$ and $F$ ranging over a set $\F$, equipped
with a $\s$-field. Let $\ell_n(\q,F)=\log \prodin p_{\q,F}(X_i)$ be the log likelihood.
Given a prior $\Pi$ on $\Theta\times\F$, a posterior distribution of $(\q,F)$ given $X_1,\ldots, X_n$ is defined,
by Bayes's formula, as 
\[
\Pi_n\bigl((\q,F)\in B\given X^{(n)} bigr)=\frac{\int_B e^{\ell_n(\q,F)}\,d\Pi(\q,F)}{\int e^{\ell_n(\q,F)}\,d\Pi(\q,F)}.
\]
We are interested in proving a Bernstein-von Mises type theorem for
the marginal posterior distribution of $\q$.
This theorem gives a normal approximation with centring and covariance determined by the efficient
influence function for estimating $\q$. However, the following theorem is stated in abstract terms. We interpret
the objects involved and its conditions after the proof. The theorem can be considered a version of
Theorem~2.1 in~\cite{castilloBernsteinMisesTheorem2015} for estimating the parameter of a semiparametric model.
It measures the discrepancy between the posterior distribution and its normal approximation through the
weak topology rather than the much stronger total variation distance, as in~\cite{castilloSemiparametricBernsteinMises2012}, but
under weaker conditions on the prior distribution.

Let $\tilde\ell_0:\X\to \RR^d$ be a measurable map such that $P_0\tilde\ell_0=0$ and such that the matrix
$\tilde I_0:=\Cov_0\bigl(\tilde\ell_0(X_1)\bigr)$ exists and is nonsingular. The subscript $0$ will later be
taken to refer to the true parameter $(\q_0, F_0)$, but does not have a special meaning in the abstract theorem.
Let $\PP_n$ denote the empirical measure of  $X_1,\ldots, X_n$
and $\ep_n=\sqrt n(\PP_n-P_0)$ the empirical process with centring $P_0$.

For given $(\q,F)\in \Theta\times\F$, assume that there exists a map $t\mapsto F_t(\q,F)$ from a given neighbourhood $\mathcal{T}$ of
$0\in\RR^d$ to $\F$ such that, for given measurable subsets $\Theta_n\subset \Theta$ and $\F_n\subset\F$,
\begin{align}\label{eqLANExpansion}
&\ell_n\Bigl(\q+\frac t{\sqrt n},F_{\frac{t}{\sqrt n}}(\q,F)\Bigr)-\ell_n(\q,F)\\
&\qquad\qquad=t^T\GG_n\tilde\ell_0-t^T\bigl(\tilde I_0+R_{n,1}(\q,F)\bigr)\sqrt n(\q-\q_0)
 -\thalf t^T\tilde I_0 t+R_{n,2}(\q,F),\nonumber
\end{align}
for a matrix-valued process $R_{n,1}$ and scalar process $R_{n,2}$ such that 
\begin{align}
\label{EqRemainders}
&\sup_{\q\in\Theta_n,F\in\F_n}\bigl\| R_{n,1}(\q,F)\bigr\|+|R_{n,2}(\q,F)|\stackrel {P_0^n}\longrightarrow 0,
\end{align}
and such that
\begin{align}
\label{EqChangeOfMeasure}
\frac{\int_{\Theta_n\times \F_n} e^{\ell_n\bigl(\q-t/\sqrt n,F_{-t/\sqrt n}(\q,F)\bigr)}\,d\Pi(\q,F)}
{\int_{\Theta_n\times \F_n}  e^{\ell_n(\q,F)}\,d\Pi(\q,F)} \stackrel {P_0^n}\longrightarrow 1.
\end{align}

Let $\mathrm{BL}_1$ be the set of all functions $h: \RR^d\ra [0,1]$ such that $|h(x)-h(y)|\le \|x-y\|$, for every $x,y\in\RR^d$.

\begin{thm}\label{thm:semiparametricBvM}
If $\Pi_n(\q\in \Theta_n, F\in\F_n\given X^{(n)})\ra 1$, in $P_0^n$-probability, and~\eqref{EqRemainders}-\eqref{EqChangeOfMeasure}
hold for these choices of $\Theta_n, \mathcal{F}_n$, then
\[
\sup_{h\in \mathrm{BL}_1} \Bigl| E \left(h\bigl(\tilde I_0\sqrt n(\q-\q_0)-\ep_n\tilde\ell_0\bigr) \given X^{(n)} \right)-\int h\,dN(0,\tilde I_0)\Bigr|
\stackrel {P_0^n}\longrightarrow 0.
\]
\end{thm}

\begin{proof}
Let $\Pi_n^*\left( \cdot \given X^{(n)}\right)$ be the renormalised distribution of the posterior distribution conditioned to the set $\Theta_n \times \mathcal{F}_n$,
 that is $\Pi_n^*\left( A \given X^{(n)}\right) = \frac{ \Pi_n\left( A \cap \Theta_n \times \mathcal{F}_n \given X^{(n)} \right)}{ \Pi_n\left( \Theta_n \times \mathcal{F}_n \given X^{(n)} \right)}$.
The total variation distance between posterior $\Pi_n\left( \cdot \given X^{(n)} \right)$ and the renormalised posterior $\Pi_n^*\left( \cdot \given X^{(n)}\right)$ is bounded by 
$\Pi_n\bigl((\q,F)\notin \Theta_n\times\F_n\given X^{(n)}\bigr)$ (see e.g.\ Lemma~K.10 in~\cite{ghosalFundamentalsNonparametricBayesian2017}), and hence tends to zero in probability, by assumption.
Therefore, up to renormalisation, without loss of generality, we assume that $\Pi(\Theta_n\times\F_n)=1$.

Let $V_n(\theta, F) = \left( \tilde{I_0} + R_{n,1}( \theta, F) \right) \sqrt{n} (\theta - \theta_0) - \mathbb{G}_n\tilde\ell_0$.
We consider the Laplace transform of $V_n(\q,F)$ relative to the posterior distribution of $(\q,F)$ given $X^{(n)}$: for $t: \frac{-t}{\sqrt{n}}\in \mathcal{T}$,
\begin{align*}
\E \Bigl(e^{t^T V_n(\q,F)}\given X^{(n)}\Bigr) &=\frac{\int e^{t^T V_n(\q,F)}e^{\ell_n(\q,F)}\,d\Pi(\q,F)}{\int e^{\ell_n(\q,F)}\,d\Pi(\q,F)}\\
&=e^{\frac{1}{2}t^T\tilde I_0t}\,\frac{\int e^{-R_{n,2}(\q,F)} e^{\ell_n(\q-\frac{t}{\sqrt n},F_{\frac{-t}{\sqrt n}}(\q,F))}\,d\Pi(\q,F)}{\int e^{\ell_n(\q,F)}\,d\Pi(\q,F)},
\end{align*}
by the definitions of $R_{n,1}$ and $R_{n,2}$. 
Since $R_{n,2}(\q,F)\ra 0$, uniformly in $(\q,F)$, by~\eqref{EqRemainders}, the factor $e^{-R_{n,2}(\q,F)}$ can be taken
out of the integral at the cost of a multiplicative $e^{o_P(1)}$-term. The remaining quotient of integrals 
tends in probability to 1 by~\eqref{EqChangeOfMeasure}. Thus, 
the right side of the display tends in probability to $e^{\frac{1}{2}t^T\tilde I_0t}$, for every fixed $t$ in a neighbourhood of 0. The pointwise
convergence of Laplace transforms implies that $V_n(\q,F)\given X^{(n)} \weak N(0,\tilde I_0)$
(e.g. \cite[Chapter 1.13]{wellnerWeakConvergenceEmpirical2023}), i.e.
\[
\sup_{h\in \mathrm{BL}_1}\Bigl|\E \left(h\bigl(V_n(\q,F)\bigr)\given X^{(n)} \right)-\int h\,dN(0,\tilde I_0)\Bigr| \stackrel {P_0^n}\longrightarrow 0.
\]
On the event
$A_n=\{\sup_{\q,F}\|\tilde I_0\bigr(\tilde I_0+R_{n,1}(\q,F)\bigr)^{-1}-I\|<\e\}$, we have
\[
\bigl\|\tilde I_0\sqrt n(\q-\q_0)-\bigr(\tilde I_0+R_{n,1}(\q,F)\bigr)\sqrt n(\q-\q_0)\bigr\|
\le\e \bigl\|\bigr(\tilde I_0+R_{n,1}(\q,F)\bigr)\sqrt n(\q-\q_0)\bigr\|.
\]
Hence for every $M$ and Lipschitz $h: \RR^d\ra[0,1]$,
\begin{align*}
&\Bigl|\E 1_{A_n}\Bigl(h\bigl(\tilde I_0\sqrt n(\q-\q_0)-\ep_n\tilde\ell_0\bigr)-h\bigl(V_n(\q,F)\bigr)\given X^{(n)}\Bigr)\Bigr|\\
&\qquad\qquad\le \e M+\Pr\Bigl(\bigl\|V_n(\q,F)+\ep_n\tilde\ell_0\bigr\|>M\given X^{(n)}\Bigr).
\end{align*}
Because there exists a Lipschitz function with $1_{\|v\|>M}\le h(v)\le 1_{\|v\|>M-1}$, the bounded Lipschitz metric is
invariant under location shifts, 
the second term is bounded above by $N\bigl(\ep_n\tilde\ell_0,\tilde I_0\bigr)\{v: \|v\|>M-1\}+o_P(1)$. It follows that
\begin{align*}
&\Bigl|\E \Bigl(h\bigl(\tilde I_0\sqrt n(\q-\q_0)-\ep_n\tilde\ell_0\bigr)\given X^{(n)} \Bigr)-\int h\,dN(0,\tilde I_0)\Bigr| \\
&\qquad\le \Pr(A_n^c)+\e M+N\bigl(\ep_n\tilde\ell_0,\tilde I_0\bigr)\{v: \|v\|>M-1\}+o_P(1).
\end{align*}
The third and second terms on the right side can be made arbitrarily small by first choosing sufficiently large $M$, 
and next sufficiently small $\e$. For fixed $M$ and $\e$ the other terms tend to zero in probability. Hence,
the left side tends to zero in probability, uniformly in $h$.
\end{proof}

In a typical application we construct the submodels $t\mapsto F_t(\q,F)$ to be least favourable at $(\q,F)$ for
estimating $\q$, i.e.\ 
the score function of the induced model $t\mapsto p_{\q+t,F_t(\q,F)}$ is the efficient score function at $(\q,F)$  for
estimating $\q$:
\[
  \frac{\partial}{\partial t}_{|t=0}\log p_{\q+t,F_t(\q,F)}(x)=\tilde\ell_{\q,F}.
\]
Only the efficient score function at the true value $(\q_0,F_0)$ of the parameter enters the theorem.
This suggests some flexibility in the definition of the submodels, as has also been employed in 
proofs of semiparametric maximum likelihood (e.g.\ \cite{murphyProfileLikelihood2000}). The submodels need satisfy
the preceding display only approximately.

We can compare the condition~\eqref{EqRemainders} to the conditions appearing in earlier literature.
In earlier literature, for example~\cite{castilloBernsteinMisesTheorem2015}, required an analysis of the likelihood as well.
They do not use perturbations in least-favourable submodels, and instead analyse the full likelihood.
Our method has the advantage that if the least-favourable submodel is known, you do no longer have to analyse the full likelihood.
Instead, it suffices to see how the likelihood changes when compared to the likelihood shifted in the least-favourable submodel.
This makes our condition more suitable to study mixture models, where one parameter is a probability measure.

Condition~\eqref{EqChangeOfMeasure} then requires that the prior does not change much 
``in the direction of the least favourable direction''. It is essential that the submodel extends in every
direction $t$ in a neighbourhood of 0, as this gives the asymptotic tightness of the posterior distribution,
by way of the Laplace transform employed in the proof of the theorem.

%% file: BernsteinVonMises/LikelihoodExpansion.tex
\subsection{Verifying \texorpdfstring{\Cref{EqRemainders}}{Likelihood expansion}}\label{sec:LAN}

Condition~\eqref{EqRemainders} can be split in a random and a deterministic part: 
\begin{align}
&\ep_n\Bigr[\sqrt n\log \frac{p_{\q+n^{-1/2}t ,F_{n^{-1/2}t}(\q,F)}}{p_{\q,F}}-t^T\tilde\ell_0\Bigr]
=o_P(1),\label{EqRandomPart}\\
&n P_0\log \frac{p_{\q+n^{-1/2}t ,F_{n^{-1/2}t}(\q,F)}}{p_{\q,F}}=-t^T\bigl(\tilde I_0+o_P(1)\bigr) \sqrt n (\q-\q_0)
-\thalf t^T\tilde I_0t+o_P(1).\label{EqDeterministicPart}
\end{align}
The $o_P(1)$ terms should tend to zero in $P_{\q_0,F_0}^n$-probability, uniformly in $(\q,F)\in \Theta_n\times\F_n$.
The uniformity in~\eqref{EqRandomPart} can be proved by showing that the relevant set of functions are contained in a Donsker class, or through entropy conditions. 
The following lemma gives sufficient conditions in terms of a Donsker class.

\begin{lem}\label{Lem:LANByDonsker}
Suppose that the map $t\mapsto \ell(t;\q,F)(x):=\log p_{\q+n^{-1/2}t ,F_{n^{-1/2}t}(\q,F)}(x)$ is twice continuously differentiable in
a neighbourhood of zero, for every $(\q,F)\in \Theta_n\times\F_n$ and $x\in\X$.
\begin{itemize}
\item If the classes of functions $\{\dot\ell(t/\sqrt n;\q,F): \|t\|<1, (\q,F)\in\Theta_n\times\F_n\}$ are
contained in a given $P_0$-Donsker class and $P_0\|\dot\ell(t_n/\sqrt n;\q_n,F_n)- \tilde\ell_0\|^2\ra 0$,
for every $\|t_n\|<1$ and $(\q_n,F_n)\in\Theta_n\times\F_n$, then~\eqref{EqRandomPart} is valid. 
\item If $\|P_0\ddot\ell(t_n/\sqrt n;\q_n,F_n)- \tilde  I_0\|\ra 0$, 
for every $\|t_n\|\le 1$ and $(\q_n,F_n)\in\Theta_n\times\F_n$, then~\eqref{EqDeterministicPart} is satisfied if also
\[
  \sup_{\q\in\Theta_n,F\in\F_n} \frac{\|P_0\dot\ell(0;\q,F)+\tilde I_0(\q-\q_0)\|}
  {\|\q-\q_0\|+n^{-1/2}}\stackrel {P_0^n}\longrightarrow 0.
\]
\end{itemize}
\end{lem}

\begin{proof}
The left side of~\eqref{EqRandomPart} is equal to 
\[
\sqrt n\ep_n\Bigl[\ell\Bigl(\frac t{\sqrt n};\q,F\Bigr)-\ell(0;\q,F)\Bigr]-t^T\ep_n\tilde\ell_0
=t^T\ep_n\int_0^1 \Bigl[\dot\ell\Bigl(\frac{st}{\sqrt n}; \q,F\Bigr)-\tilde\ell_0\Bigr]\,ds.
\] 
If the classes of functions $\{\dot\ell(t/\sqrt n;\q,F): \|t\|<1, (\q,F)\in\Theta_n\times\F_n\}$ are contained in
a fixed Donsker class, then the convex hull relative to $t$ of these classes are contained in the
convex hull of the latter Donsker class, which is also Donsker. Thus, the 
classes of functions $\int_0^1 (\dot\ell(st/\sqrt n; \q,F)-\tilde\ell_0)\,ds$ are contained in
a Donsker class. Their empirical process tends to zero, as 
$\sup_{|t\|<1}P_0\|\dot\ell(t/\sqrt n;\q_n,F_n)- \tilde\ell_0\|^2\ra 0$, by assumption.

By a second-order Taylor expansion we can write the left side of~\eqref{EqDeterministicPart} as
\begin{align*}
nP_0\Bigl[\ell\Bigl(\frac t{\sqrt n};\q,F\Bigr)-\ell(0;\q,F)\Bigr]&=\sqrt n\, t^T P_0\dot\ell(0;\q,F)\\
&+t^T\!\int_0^1 P_0\ddot\ell\Bigl(\frac{st}{\sqrt n}; \q,F\Bigr)(1-s)\,ds\, t.
\end{align*}
The first term on the right is $-t^T\bigl(\tilde I_0+o_P(1)\bigr)\sqrt n(\q-\q_0)+o_P(1)$, by assumption.
The second term is $t^T\tilde I_0 t/2+o_P(1)$, by the assumption that
$\sup_{\|t\|\le 1}\|P_0\ddot\ell(t/\sqrt n;\q_n,F_n)- \tilde  I_0\|\ra 0$.
\end{proof}

For an exact least favourable submodel $t\mapsto F_t(\q,F)$, the derivative $\dot\ell(0;\q,F)$ is equal
to the efficient score function $\tilde \ell_{\q,F}$. Then the displayed condition in the lemma is satisfied if
$\Theta_n$ shrinks to $\q_0$ and, as uniformly in $F\in\F_n$, as $\q\ra\q_0$,
\[
P_{\q_0,F_0}\tilde\ell_{\q,F}=-\tilde I_0 (\q-\q_0)+o_P\bigl(\|\q-\q_0\|+n^{-1/2}\bigr).
\]
There are two ways in which we can verify this, both starting from the fact that
$P_{\q,F}\tilde\ell_{\q,F}=0$, for every $(\q,F)$. The first is to write the left side as
\[
P_{\q_0,F_0}(\tilde\ell_{\q,F}-\tilde\ell_{\q_0,F})+P_{\q_0,F_0}\tilde\ell_{\q_0,F}
\approx P_{\q_0,F_0}\dot{\tilde\ell}_{\q_0,F_0}(\q-\q_0)+P_{\q_0,F_0}\tilde\ell_{\q_0,F}.\]
Here $P_{\q_0,F_0}\dot{\tilde\ell}_{\q_0,F_0}=-\tilde I_0$, as follows from differentiating 
across the identity $P_{\q,F}\tilde\ell_{\q,F}=0$, and the fact that $P_{\q,F}\tilde\ell_{q,F}\dot\ell_{\q,F}^T=\tilde I_{\q,F}$,
for $\dot\ell_{\q,F}$ the ordinary score function for $\q$.
The second way is to write the left side as
\[
(P_{\q_0,F_0}-P_{\q,F_0})\tilde\ell_{\q,F}+P_{\q,F_0}\tilde\ell_{\q,F}
\approx -P_{\q_0,F_0}\tilde\ell_{\q_0,F_0}\dot\ell_{\q_0,F_0}^T(\q-\q_0)+P_{\q,F_0}\tilde\ell_{\q,F}.
\]
In both ways the term on the far right will be required to be $o(n^{-1/2})$, which is version of the ``no-bias condition'', which also popped up in the analysis of the maximum likelihood estimator.
This may require a rate of convergence of $F\in\F_n$ to $F_0$ (typically $o(n^{-1/4})$, because the first derivative
with respect to $F$ should vanish in view of the orthogonality of $\tilde \ell_{\q,F}$ to score functions
for $F$), but structural properties of the model determine may help. 
In the mixture models we consider here, the term is identically zero. Indeed,
$\tilde\ell_{\q,F}$ is orthogonal to all score functions for $F$, which include the scores $p_{\q,G}/p_{\q,F}-1$. Thus,
for every $F,G$,
\[
P_{\q,G}\tilde\ell_{\q,F}
=\int \Bigl(\frac{p_{\q,G}}{p_{\q,F}}-1\Bigr)\tilde\ell_{\q,F}\,dP_{\q,F}=0.
\]

Next, we will provide a lemma that is useful for the computations. To apply~\Cref{Lem:LANByDonsker} we need to compute the gradient and the hessian matrix. The following lemma provides a strategy for computing these. 

\begin{lem}\label{lem:ExpressionsForGradiantHessian}
    Suppose that $p_{\theta}(x \given z) = e^{h_\theta(x \given z)}$, and that the perturbation in the least favourable submodel is given by $p_{\theta + t, F_t}(x) = \int p_{\theta + t} (x \given \phi_{t,\theta}(z)) \dd F(z)$. Denote $D_{t, \theta} (x \given z) := \nabla_t h_{\theta + t}( x \given \phi_{t, \theta}(z))$. Then, whenever the functions within the integrals are dominated by an $F(z)$ integrable function:
    \[
        \nabla_t p_{\theta + t, F_t}(x) = \frac{\int e^{h_{\theta+t}(x \given \phi_{t,\theta}(z))} D_{t, \theta} (x \given z) \dd F(z)}
        {
            p_{\theta + t, F_t}(x)
        }
    \]
    and
    \begin{align*}
        H_t p_{\theta + t, F_t}(x) &= 
          \frac{ 
              \int e^{h_{\theta + t}(x \given \phi_{t, \theta}(z))} \left(H_t(h_{\theta + t}(x \given \phi_{t, \theta}(z))) + D_{t, \theta}(x \given z) D_{t, \theta}^T(x\given z) \right) \dd F(z)
              }{
              p_{\theta + t, F_t}(x)
          }\\
          &+ \left(\nabla_t p_{\theta + t, F_t}(x)\right) \left( \nabla_t p_{\theta + t, F_t}(x) \right)^T.
    \end{align*}
\end{lem}

\begin{proof}
   Use the chain rule to compute the gradient and the Hessian of $\log(f)$. 
   This gives $\frac{f'}{f}$ and $\frac{ f''}{f} - \left( \frac{f'}{f} \right) \left( \frac{f'}{f} \right)^T$. 
   Set $f(t) = \int e^{h(t,z)} \dd F(z)$, so that 
   \begin{align*}
       f' &= \nabla_t \int e^{h(t,z)} \dd F(z),\\
       f'' &= H_t \int e^{h(t,z)} \dd F(z) = \nabla_t^T f'.
   \end{align*}
   Next we use the dominated integrability to apply the Leibniz rule and move the differentiating inside the integration sign and apply the chain rule. This gives
   \begin{align*}
       f' &= \int e^{h(t,z)} \nabla_t h(t,z) \dd F(z), \\
       f'' &= \int e^{h(t,z)} \left( \nabla_t^T \nabla_t h(t,z) + \left(\nabla_t h(t,z)\right) \left(\nabla_t h(t,z)\right)^T \right) \dd F(z).
   \end{align*}
   Filling in the definitions in as in the assumptions gives the final result.
\end{proof}

To study the asymptotic behaviour of the derivatives of the log likelihood in $P_{\theta_0, F_0}$, we can use the following lemma:
\begin{lem}\label{lem:LimitsOfExpectationsLogLik}
    Let $p_{\theta}(x \given z) = e^{h_\theta(x \given z)}$ be the mixture kernel, and the perturbation in the least favourable submodel given by $p_{\theta + t, F_t} = \int p_{\theta + t}(x \given \phi_{t, \theta}(z)) \dd F(z)$.
    Suppose that $h_{\theta}(x \given z)$ is twice continuously differentiable in $(\theta, z)$,
    and that $\phi_{t,\theta}(z)$ is twice continuously differentiable in $t, \theta$.
    Furthermore, assume that the maps $z \mapsto \nabla_t p_{\theta + t}(x \given \phi_{t, \theta}(z)) \given_{t = t_n/\sqrt{n}}$ and $z \mapsto H_t p_{\theta + t}(x \given \phi_{t, \theta}(z))\given_{t = t_n/\sqrt{n}}$ are bounded for all $x$, $| t_n| \leq 1$ and $\theta \in \Theta_n$.
    Finally, assume that $\nabla_t \log p_{\theta + t, F_t}(x) \given_{t = t_n/\sqrt{n}}$ and $H_t \log p_{\theta + t, F_t}(x) \given_{t = t_n/\sqrt{n}}$ have envelope functions that are square integrable and integrable respectively.
    Then for all $|t_n| \leq 1$, $\theta_n \rightarrow \theta_0, F_n \rightsquigarrow F_0$:
    \begin{align*}
        P_{\theta_0, F_0} \| \dot\ell( t_n/\sqrt{n} ; \theta_n, F_n) - \ell_0 \|^2 &\rightarrow 0 \\
        \|P_{\theta_0, F_0} \ddot\ell(t_n/\sqrt{n} ; \theta_n, F_n) - \tilde{I}_0 \|&\rightarrow 0.
    \end{align*}
\end{lem}

\begin{proof}
    We use the explicit expressions found in \Cref{lem:ExpressionsForGradiantHessian}.
    By the chain rule, continuity, integrability and weak convergence we can conclude that pointwise
    \begin{align*}
        \dot\ell( t_n/\sqrt{n} ; \theta_n, F_n)(x) - \dot\ell (0; \theta_0, F_0)(x) &\rightarrow 0 \\
        \ddot\ell(t_n/\sqrt{n} ; \theta_n, F_n)(x) - \ddot\ell(0; \theta_0, F_0)(x) &\rightarrow 0.
    \end{align*}
    Next, we use the pointwise limits, the envelope functions together with the dominated convergence theorem
    \begin{align*}
        P_{\theta_0, F_0} \| \dot\ell( t_n/\sqrt{n} ; \theta_n, F_n) - \dot\ell(0; \theta_0, F_0) \| &\rightarrow 0\\
        P_{\theta_0, F_0} \| \ddot\ell(t_n/\sqrt{n} ; \theta_n, F_n) - \ddot\ell(0;\theta_0, F_0) \| &\rightarrow 0
    \end{align*}
    By using that $\ell_0 = \dot\ell(0; \theta_0, F_0)$ and $P_{\theta_0, F_0} \ddot\ell(0; \theta_0, F_0) = \tilde{I}_0$, we finally get the results as claimed:
    \begin{align*}
        P_{\theta_0, F_0} \| \dot\ell( t_n/\sqrt{n} ; \theta_n, F_n) - \ell_0 \|^2 &\rightarrow 0 \\
        \|P_{\theta_0, F_0}  \ddot\ell(t_n/\sqrt{n} ; \theta_n, F_n) - \tilde{I}_0 \|&\rightarrow 0.
    \end{align*}

\end{proof}

%% file: BernsteinVonMises/SemiparametricBvMDP.tex
\subsection{Verifying \texorpdfstring{\cref{EqChangeOfMeasure}}{the Change-of-Measure condition} for Species sampling processes}\label{sec:changeofmeasure}
The following lemma allows us to verify the BvM if $F_t$ is a coordinate transform of $F$, it is consistent under the coordinate changes, and it does not create too many clusters in the posterior.
We will formalise these assumptions as to be able to discuss them.

\begin{assumption}\label{ass:ChangeOfVariables}
If $F_t$ is the map as in \Cref{EqChangeOfMeasure}, we assume that $ p_{\theta +t, F_t}(X) = \int p_{\theta + t}(x \given \phi_{t, \theta}(z)) \dd F(z)$ \end{assumption}
Under this assumption $\phi_{t,\theta}$ defines a change of variables of the latent variables.
This is equivalent to applying this change of variables to the centre measure $G$ of the species sampling process. 
Denote the prior that we get after changing the coordinates by $\Pi_{t, \theta}$.

\begin{assumption}\label{ass:Smoothness}
    For $h$ the density of the prior $\pi$ and $G$ the centre measure of the species sampling process with density $g$, we assume that there exists an open set $U$ of $\theta$ and constants $C_t, C'_t > 0$ such that, for all $\theta \in U$ and all $z$
     \begin{align}
         \left\| \frac{g( \phi^{-1}_{\frac{t}{\sqrt{n}}, \theta - \frac{t}{\sqrt{n}}}(z)) \left\| \det \dot \phi^{-1}_{\frac{t}{\sqrt{n}}, \theta - \frac{t}{\sqrt{n}}}(z) \right\|}{g(z)} - 1  \right\| &\leq \frac{C_t}{\sqrt{n}};\label{ass:smoothnessinG}\\
        \left\| \frac{h( \theta + \frac{t}{\sqrt{n}})}{h(\theta)} - 1 \right\| & \leq \frac{C'_T}{\sqrt{n}}.\label{ass:smoothnessinh}
    \end{align}
\end{assumption}
The first condition is needed for a compatibility of the change of coordinates with the centre measure.
This is needed to ensure that the errors will vanish.
The second condition is fairly weak, it is implied by a local log Lipschitz condition on the prior density. See \Cref{lem:SmoothnessInG,lem:SmoothnessInH} for tools for verifying these assumptions.

For our argument to work, we need to show that there is a small number of distinct clusters in the posterior.
We can do this both for the unperturbed posterior $\pi \otimes \Pi$ and the perturbed posterior $\pi \otimes \Pi_{\frac{t}{\sqrt{n}}, \theta}$
We can rewrite the posterior using latent variables $Z_i$. 
Denote by $K_n$ the number of distinct values in $Z_i$.
We want this number to be much smaller than $\sqrt{n}$
\begin{assumption}\label{ass:SmallNumberClusters}
    We assume that for all $t$ in a neighbourhood of $0$, and for $U$ a neighbourhood of $\theta_0$, 
    \[
        \pi \otimes \Pi_{ \frac{t}{\sqrt{n}}, \theta} \left( \theta + \frac{t}{\sqrt{n}} \not\in U, K_n \geq k_n \given X^{(n)} \right) = o(1),
    \]
    for some sequence $k_n = o(\sqrt{n})$.
\end{assumption}
To verify~\Cref{ass:SmallNumberClusters} we can either use a direct argument, or use the remaining mass theorem~\cite[Theorem 8.20]{ghosalFundamentalsNonparametricBayesian2017}.
We provide a suitable tail bound on the distinct number of clusters under the Dirichlet process prior in \Cref{lem:tailbound_distinct_obs_DP}.
For a Mixture of finite mixtures, the number of clusters is given by $K$, so all that is needed is to show that $\PP(K > k_n)$ is small enough.

Finally, we need to assume the following consistency result
\begin{assumption}\label{ass:Consistency}
    Assume that for all $t$ the following limit converges to $1$.
    \[
\frac{
    \iint_{\Theta_n, \mathcal{F}_n} e^{\ell_n\left(\theta + \frac{t}{\sqrt{n}}, F_{\frac{t}{\sqrt{n}}}(\theta, F)\right)} \dd \pi(\theta) \dd \Pi(F)
    }{
        \iint e^{\ell_n\left(\theta + \frac{t}{\sqrt{n}}, F_{\frac{t}{\sqrt{n}}}(\theta, F)\right)} \dd \pi(\theta) \dd \Pi(F)
    }
    \rightsquigarrow 1.
    \] 
\end{assumption}

With these assumptions, we can prove the lemma that allows us to verify the chance of measures condition. 
\begin{lem}\label{lem:ChangeOfMeasure}
    Let $\Pi$ be a species sampling process with centre measure $G$, such that $G$ has density $g$.
    Let $\pi$ be a probability distribution on $\Theta$ with density $h$.  
    Suppose that for all $t \in T$, with $T$ an open neighbourhood of $0$ \Cref{ass:ChangeOfVariables,ass:Smoothness,ass:SmallNumberClusters} hold.
        Then
    \[
        \frac{
            \iint_{\Theta_n, \mathcal{F}_n} e^{\ell_n\left(\theta + \frac{t}{\sqrt{n}}, F_{\frac{t}{\sqrt{n}}}(\theta, F)\right)} \dd \pi(\theta) \dd \Pi(F)
        }{
            \iint_{\Theta_n, \mathcal{F}_n} e^{\ell_n\left(\theta, F\right)} \dd \pi(\theta) \dd \Pi(F)
        }
        \rightsquigarrow 1.  
    \]
\end{lem}

\begin{proof}
    First, by \Cref{ass:Consistency} it suffices to show
        \[
        \frac{
            \iint e^{\ell_n\left(\theta + \frac{t}{\sqrt{n}}, F_{\frac{t}{\sqrt{n}}}(\theta, F)\right)} \dd \pi(\theta) \dd \Pi(F)
        }{
            \iint e^{\ell_n\left(\theta, F\right)} \dd \pi(\theta) \dd \Pi(F)
        }
        \rightsquigarrow 1.  
    \]

    By \Cref{ass:SmallNumberClusters} we see that 
    \begin{align*}
        &\int_{U + \frac{t}{\sqrt{n}}} \int_{K \leq k_n} e^{\ell_n\left(\theta + \frac{t}{\sqrt{n}}, F_{t/\sqrt{n}}(\theta, F) \right)} \dd \Pi(F) \dd \pi(\theta) \\
        &\sim \iint e^{\ell_n\left(\theta + \frac{t}{\sqrt{n}}, F_{t/\sqrt{n}}(\theta, F) \right)} \dd \Pi(F) \dd \pi(\theta).
    \end{align*}
    By evaluating this when $t = t$ and $t = 0$, we find that
    \begin{align*}
      & \frac{
          \iint e^{\ell_n\left(\theta + \frac{t}{\sqrt{n}}, F_{t/\sqrt{n}}(\theta, F) \right)} \dd \Pi(F) \dd \pi(\theta)
      }{
        \iint e^{\ell_n\left(\theta, F \right)} \dd \Pi(F) \dd \pi(\theta)
      }\\
      &\sim \frac{
          \int_{U + \frac{t}{\sqrt{n}}} \int_{K \leq k_n} e^{\ell_n\left(\theta + \frac{t}{\sqrt{n}}, F_{t/\sqrt{n}}(\theta, F) \right)} \dd \Pi(F) \dd \pi(\theta) 
      }{
        \int_U \int_{K \leq k_n} e^{\ell_n\left(\theta, F \right)} \dd \Pi(F) \dd \pi(\theta)
    }.
    \end{align*}
    In the integral in the numerator we do a change of variables $\theta + \frac{t}{\sqrt{n}} \mapsto \theta$.
    \begin{align*}
       &\int_{U + \frac{t}{\sqrt{n}}} \int_{K \leq k_n} e^{\ell_n\left(\theta + \frac{t}{\sqrt{n}}, F_{t/\sqrt{n}}(\theta, F) \right)} \dd \Pi(F) \dd \pi(\theta) 
      \\
      & =\int_{U} \int_{K \leq k_n} e^{\ell_n\left(\theta, F_{t/\sqrt{n}}(\theta - \frac{t}{\sqrt{n}}, F) \right)} \dd \Pi(F) h(\theta - \frac{t}{\sqrt{n}}) \dd \theta 
  \end{align*}
  By \Cref{ass:smoothnessinh} in \Cref{ass:Smoothness} and \Cref{lem:MoveRatioBoundIntoSum} we find that 
  \[
      \left\|
      \frac{
          \int_{U} \int_{K \leq k_n} e^{\ell_n\left(\theta, F_{t/\sqrt{n}}(\theta - \frac{t}{\sqrt{n}}, F) \right)} \dd \Pi(F) h(\theta - \frac{t}{\sqrt{n}}) \dd \theta
      }{
          \int_U \int_{K \leq k_n} e^{\ell_n\left(\theta, F \right)} \dd \Pi(F) h(\theta) \dd \theta
      } - 1
      \right\| \leq \frac{c_t}{\sqrt{n}}
  \]
  By \Cref{ass:ChangeOfVariables} and \Cref{lem:ChineseRepSpeciesSampling}, it follows that 
    \begin{align*}
        &\int_{K \leq k_n} e^{\ell_n\left(\theta , F_{t/\sqrt{n}}(\theta - \frac{t}{\sqrt{n}}, F) \right)} \dd \Pi(F) \\
        &=  \sum_{\mathcal{S} \in S(n)} \II_{\# \mathcal{S} < k_n} \mathcal{P}(S) \prod_{i = 1}^{\#\mathcal{S}} \int \prod_{j \in S_i} p_{\theta}(X_j\given \phi_{\frac{t}{\sqrt{n}}, \theta - \frac{t}{\sqrt{n}}}(z)) \dd G(z),
    \end{align*}
    where $S(N)$ is the collection of partitions of $[n]$.
    By plugging this in the preceding expression we get
    \begin{align*}
        \frac{
            \sum_{\mathcal{S} \in S(n)} \II_{\# \mathcal{S} < k_n} \mathcal{P}(S) \int_{U} \prod_{i = 1}^{\#\mathcal{S}} \left[\int \prod_{j \in S_i} p_{\theta }(X_j\given \phi_{\frac{t}{\sqrt{n}}, \theta - \frac{t}{\sqrt{n}}}(z)) \dd G(z)\right] \dd \pi(\theta) 
      }{
          \sum_{\mathcal{S} \in S(n)} \II_{\# \mathcal{S} < k_n} \mathcal{P}(S) \int_U \prod_{i = 1}^{\#\mathcal{S}} \left[\int \prod_{j \in S_i} p_{\theta}(X_j\given z) \dd G(z) \right] \dd \pi(\theta)
    }.
    \end{align*}

    We will now apply the change of variables $\phi_{\frac{t}{\sqrt{n}}, \theta - \frac{t}{\sqrt{n}}}(z) \mapsto z$. 
    We will study the effect in steps, starting by analysing
    \begin{align*}
        &\frac{\int \prod_{j \in S_i} p_{\theta}(X_j\given \phi_{\frac{t}{\sqrt{n}}, \theta - \frac{t}{\sqrt{n}}}(z)) \dd G(z)}{\int \prod_{j \in S_i} p_{\theta}(X_j\given z) \dd G(z)} \\
        &=\frac{\int \prod_{j \in S_i} p_{\theta }(X_j\given z) g(\phi_{\frac{t}{\sqrt{n}}, \theta - \frac{t}{\sqrt{n}}}( z)) \|\det \dot \phi^{-1}_{t, \theta - \frac{t}{\sqrt{n}}} \| \dd z}{\int \prod_{j \in S_i} p_{\theta}(X_j\given z) g(z) \dd z} 
    \end{align*} 
    By using \Cref{ass:smoothnessinG} of \Cref{ass:Smoothness} and \Cref{lem:MoveRatioBoundIntoSum} we see that uniformly for all $\theta \in U$ and every partition $S \in S(n)$ and every $S_i \in S$
    \[
        \left\| \frac{\int \prod_{j \in S_i} p_{\theta }(X_j\given z) g(\phi_{\frac{t}{\sqrt{n}}, \theta - \frac{t}{\sqrt{n}}}( z)) \|\det \dot \phi^{-1}_{\frac{t}{\sqrt{n}}, \theta - \frac{t}{\sqrt{n}}} \| \dd z}{\int \prod_{j \in S_i} p_{\theta}(X_j\given z) g(z) \dd z} 
        - 1 \right\| \leq \frac{C_t}{\sqrt{n}}.
    \]
    Next we use that when $\| \frac{a_i}{b_i} - 1 \| \leq \frac{C_t}{\sqrt{n}}$, then $\| \log \left( \frac{a_i}{b_i} \right) \| = \| \log \left( 1 + \frac{ a_i - b_i}{b_i} \right) \| \leq \|\frac{a_i}{b_i} - 1\|$. Hence, $\| \log \prod_{i = 1}^{k_n} \left( \frac{a_i}{b_i} \right) \| \leq \sum_{i = 1}^{k_n} \| \log \frac{a_i}{b_i} \| \leq \frac{C_t k_n}{\sqrt{n}} \rightarrow 0$. By continuity of the exponential map, then also $\| \prod_{i = 1}^{k_n} \frac{a_i}{b_i} - 1 \| \rightarrow 0$ uniformly.
    This yields that uniformly for every partition $S \in S(n)$ with $\# S < k_n = o(\sqrt{n})$ and $\theta \in U$, 
    \[
        \left\| \frac{ \prod_{i = 1}^{\# S} \left[ \int \prod_{j \in S_i} p_{\theta }(X_j\given z) g(\phi_{\frac{t}{\sqrt{n}}, \theta - \frac{t}{\sqrt{n}}}( z)) \|\det \dot \phi^{-1}_{\frac{t}{\sqrt{n}}, \theta - \frac{t}{\sqrt{n}}} \| \dd z  \right]}{ \prod_{i = 1}^{\# S} \left[ \int \prod_{j \in S_i} p_{\theta}(X_j\given z) g(z) \dd z \right] } - 1 \right\| \rightarrow 0.
    \]
    Hence also when we sum over all such partitions in the numerator and denominator. 
    \[
        \left\|
        \frac{
            \int_{K \leq k_n} e^{\ell_{\theta, F_{\frac{t}{\sqrt{n}}}}\left( \theta, F_{\frac{t}{\sqrt{n}}}(\theta - \frac{t}{\sqrt{n}}, F) \right)} \dd \Pi(F)
        }{
            \int_{K \leq k_n} e^{\ell_{\theta, F}} \dd \Pi(F)
        } - 1
        \right\| \rightarrow 0,
    \]
    uniformly in $\theta \in U$ as $n \rightarrow \infty$.
    Finally plugging this in the previous result combined with \Cref{lem:MoveRatioBoundIntoSum} we get the final result.
\end{proof}

\textbf{Supporting Lemma}
Our first lemma is used to characterise expectations of products of functions of observations of a species sampling process.
In our paper, we plug in the likelihood and use a latent variable representation.
\begin{lem}\label{lem:ChineseRepSpeciesSampling}
    Let $\Pi$ be a species sampling prior with EPPF $\mathcal{P}$ and centre measure $G$. 
    If we sample $Z_1,\dots, Z_n \given P \sim P$ for $P \sim \pi$, we have
    \[
        \EE[ \prod_{i = 1}^n f_i(Z_i)] = \sum_{\mathcal{S} \textrm{Partition of }[n]} \mathcal{P}(S) \prod_{i = 1}^{\#\mathcal{S}} \int \prod_{j \in S_i} f_{j}(z) \dd G(z).
    \]
\end{lem}

\begin{proof}
    The probability that a random partition $\mathcal{S}$ takes value $S$ is given by the EPPF $\mathcal{P}$, $\mathcal{P}(S)$.
    Given a partition $\mathcal{S} = S$, we can compute the conditional probability by splitting the data into each of the $K_n$ partitions, and then integrating out using the centre measure $G$.
    This yields
    \begin{align*}
        \EE[ \prod_{i = 1}^n f_i (Z_i)] &= \EE[ \EE[ \prod_{i = 1}^n f_i (Z_i) | \mathcal{S} = S]] \\
            &= \EE[ \EE[ \prod_{i = 1}^{K_n} \prod_{j \in S_i} f_i (\theta_i)|\mathcal{S} = s ] ]\\
            &= \EE[ \EE[ \prod_{i = 1}^{K_n} \int \prod_{j \in S_i} f_i (\theta) \dd G(\theta) | \mathcal{S} = S]]\\
            &= \EE[ \prod_{i = 1}^{K_n} \int \prod_{j \in S_i} f_i (\theta) \dd G(\theta) ] \\
            &= \sum_{S \textrm{ Partition of }[n]} \mathcal{P}(S) \prod_{i = 1}^{K_n} \int \prod_{j \in S_i} f_i (\theta) \dd G(\theta).
    \end{align*}
\end{proof}

To show \Cref{ass:Smoothness} might require a specific argument for \Cref{ass:smoothnessinG}.
In our examples, we can use the following lemma:
\begin{lem}\label{lem:SmoothnessInG}
    Suppose that there exists a neighbourhood $U$ of $\theta_0$ such that there exists constants $L, L', C, C'> 0 $ such that
    \begin{align*}
        \left\|\det \dot\phi_{\frac{t}{\sqrt{n}}, \theta - \frac{t}{\sqrt{n}}}(z)^{-1} - 1 \right\| & \leq C \frac{\|t\|}{\sqrt{n}}\\
        \left\| \frac{g\left(\phi_{\frac{t}{\sqrt{n}},\theta - \frac{t}{\sqrt{n}}}(z)\right) - g(z)}{g(z)} \right\| &\leq L' \frac{\|t\|}{\sqrt{n}}
    \end{align*}
    for all $\theta \in U$, $z$ and $x$ small enough.
    Then \Cref{ass:smoothnessinG} is satisfied.
\end{lem}

The conditions of this lemma seem somewhat natural, the change of variables $\phi_t$ has to converge to the identity, cannot have unbounded derivatives and $g$ has to be locally Lipschitz in a suitable way.
\begin{proof}
    First we can use the triangle inequality to get
    \[
        \|\det \dot\phi_{\frac{t}{\sqrt{n}}, \theta - \frac{t}{\sqrt{n}}}(z)^{-1}\| \leq \|\det \dot\phi_{\frac{t}{\sqrt{n}}, \theta - \frac{t}{\sqrt{n}}}(z)^{-1} - 1\| + 1 \leq C \frac{\|t\|}{\sqrt{n}} + 1 \leq L, 
    \]
    by bounding for $n = 1$, setting $L = Ct + 1$.
    To show \Cref{ass:smoothnessinG} we will use the triangle inequality and the assumptions as follows
    \begin{align*}
        &\left\| \frac{g(\phi_{\frac{t}{\sqrt{n}}, \theta - \frac{t}{\sqrt{n}}}(z)) \det\dot\phi_{\frac{t}{\sqrt{n}}, \theta - \frac{t}{\sqrt{n}}}(z)^{-1}}{g(z)} -1 \right\| \\
        &=  \left\| \frac{g(\phi_{\frac{t}{\sqrt{n}}, \theta - \frac{t}{\sqrt{n}}}(z)) \det\dot\phi_{\frac{t}{\sqrt{n}}, \theta - \frac{t}{\sqrt{n}}}(z)^{-1} - g(z)}{g(z)} \right\|\\
        &=  \left\| \frac{g(\phi_{\frac{t}{\sqrt{n}}, \theta - \frac{t}{\sqrt{n}}}(z)) \det\dot\phi_{\frac{t}{\sqrt{n}}, \theta - \frac{t}{\sqrt{n}}}(z)^{-1} - g(z) \det\dot\phi_{\frac{t}{\sqrt{n}}, \theta - \frac{t}{\sqrt{n}}}(z)^{-1}}{g(z)}\right\|\\
        &+ \left\|\frac{g(z)\det\dot\phi_{\frac{t}{\sqrt{n}}, \theta - \frac{t}{\sqrt{n}}}(z)^{-1} - g(z)}{g(z)} \right\| \\
        &  \leq \left\|\det\dot \phi_{\frac{t}{\sqrt{n}}, \theta - \frac{t}{\sqrt{n}}}(z)^{-1} \right\| \left\| \frac{g(\phi_{\frac{t}{\sqrt{n}}, \theta - \frac{t}{\sqrt{n}}}(z)) - g(z)}{g(z)} \right\| \\
        &+ \left\|\det\dot\phi_{\frac{t}{\sqrt{n}}, \theta - \frac{t}{\sqrt{n}}}(z)^{-1} - 1 \right\| \\
    &\leq ( L L' + C) \frac{\|t\|}{\sqrt{n}}.
    \end{align*}
\end{proof}

The conditions of \Cref{lem:SmoothnessInG} are somewhat abstract.
The following lemma gives more explicit conditions that work for our examples.
It applies to most of the common densities with polynomial tails.
\begin{lem}\label{lem:SmoothnessInGSufficient}
    Suppose that $t \in \RR^d$ and $z \in \RR^k$, and that $\phi_{t, \theta}(z): \RR^k \rightarrow \RR^k$ is given by 
    \[
        \phi_{t, \theta}(z)_j = z_j + z^T \alpha_{\theta, j} t + (\beta t)_j,
    \]
    for $\alpha_{\theta} \in \mathcal{A}$, $\beta \in \mathcal{B}$. 
    Assume that $\mathcal{A}, \mathcal{B}$ are compact.
    Then the first line from the display in \Cref{lem:SmoothnessInG} is satisfied.
    If in addition, suppose that for $g$ the density from \Cref{lem:SmoothnessInG}, $\log g$ is continuously differentiable, and that, for $h = \nabla \log g$, the following hold:
    \begin{itemize}
        \item $h_i$ is uniformly bounded in $z$ for all $i$ or $\mathcal{B} = \{ 0 \}$;
        \item and the maps $z \mapsto z_i h(z)_j$ is uniformly bounded in $z$ for all $i,j$ or $\mathcal{A} = \{ 0 \}$,
    \end{itemize}
    then the second line from the display in \Cref{lem:SmoothnessInG} is satisfied.
\end{lem}
\begin{proof}
    First note that $\dot\phi_{t, \theta}(z)_{i,j} = \II_{i = j} + \sum_{l = 1}^d \alpha_{\theta, i, j, l} t_l$.
    Using a Taylor expansion of the determinant, we find that for $t$ small enough
    \begin{align*}
        \left\| \det\left( \dot\phi_{t, \theta}(z) \right)  - 1 \right\| &\leq 2 \left\| \sum_{j = 1}^k \sum_{l = 1}^d \alpha_{\theta, j, j, l} t_l \right\|\\
        &\leq 2 k \sup_{\alpha \in \mathcal{A}} \max_{i,j, l} \|\alpha_{i,j,l}\| \| t\|.
    \end{align*}
    Here $\sup_{\alpha \in \mathcal{A}} \max_{i,j, l} \alpha_{i,j,l}$ is bounded since $\mathcal{A}$ is compact.
    Therefore, the first line from the display in \Cref{lem:SmoothnessInG} is satisfied. 
    Because $\log g$ is continuously differentiable, we can define the map $f_{\theta, z}(t) = \log g(\phi_{t, \theta}(z))$. 
    Then $f_{\theta, z}(0) = \log g(z)$ and by the chain rule
    \[
      D f = \frac{(D \phi_{t, \theta}(z))^T \nabla g}{g}.
    \]
    We use that $D \phi_{t, \theta}(z) = \alpha_{\theta}(z) + \beta$, where $\alpha(z)$ is the matrix given by $\alpha(z)_{i,l} = \sum_{j}^k \alpha_{i,j,k} z_j$. 
    From the mean value theorem, 
    \[
        \| f(t) - f(0) \| \leq \|t \| \sup_{\alpha \in \mathcal{A}, \beta \in \mathcal{B}, z} \|\frac{(Z \alpha + \beta)^T \nabla g(z)}{g(z)}\|.
    \]
    This is finite since we can bound the supremum by
    \begin{align*}
        \sup_{x \in U, \alpha \in \mathcal{A}, \beta \in \mathcal{B}, z} & \left\|\frac{(Z \alpha + \beta)^T \nabla g(z)}{g(z)}\right\|\\
        & \leq \sup_{\alpha \in \mathcal{A}, z} \left\| \frac{\alpha(z) \nabla g(z)}{g(z)} \right\| + \quad + \sup_{x \in U, \beta \in \mathcal{B}, z} \left\| \frac{ \beta \nabla g(z)}{g(z)} \right\| \\
        &\leq \sup_{\alpha \in \mathcal{A}} \left\| \max_{i,j,l} \alpha_{i,j,l} \right\| \max_{i,j} \sup_z \left\| z_i \left(\nabla \log g(z)\right)_j \right\| \\
        &\quad + \sup_{\beta \in \mathcal{B}} \left\| \max_{i,j} \beta_{i,j} \right\| \max_{i} \sup_{z} \left\| \left(\nabla \log g(z)\right)_i \right\|.
    \end{align*}
    Because $\mathcal{A}, \mathcal{B}$ we assumed to be compact, and the assumptions on the gradient, these are all bounded.
    Hence, there exists a $C > 0$ such that 
    \[
        \sup_{\alpha \in \mathcal{A}, \beta \in \mathcal{B}, z} \|\frac{(Z \alpha + \beta)^T \nabla g(z)}{g(z)}\| \leq C.
    \]
    This means that 
    \[
        \sup_{z, t \in U, \theta \in V} \left\| \log \frac{ g( \phi_{t, \theta}(z))}{g(z)}\right\| \leq \|t\| C.
    \]
    Using the first-order Taylor-expansion for $e^x$ when $\| x \| < 1$ gives 
    \[
        \sup_{z, t \in U, \theta \in V} \left\|  \exp \left( \log \frac{ g( \phi_{t, \theta}(z))}{g(z)} \right) -1 \right\| \leq 2\|t \| C.
    \]
    Hence, expanding the definition gives
    \[
        \sup_{z, t \in U, \theta \in V} \left\|\frac{ g( \phi_{t, \theta}(z))}{g(z)} -1 \right\| \leq 2\|t \| C.
    \]
    Hence, taking $L' = 2C$ verifies the second equation in the display from \Cref{lem:SmoothnessInG}.
\end{proof}

To verify \Cref{ass:smoothnessinh} we can use the following result
\begin{lem}\label{lem:SmoothnessInH}
    Suppose that the prior density $h$ is continuously differentiable and non-zero in a neighbourhood of $\theta_0$. Then there exists a neighbourhood $U$ and $C_t$ such that, for all $n$ large enough
    \[
         \| \frac{h( \theta + \frac{t}{\sqrt{n}})}{h(\theta)} - 1 \| \leq \frac{C'_T}{\sqrt{n}}
    \]
    for all $\theta \in U$.
\end{lem}

\begin{proof}
    Because $h$ is continuously differentiable in a neighbourhood of $\theta_0$, it is locally Lipschitz in a neighbourhood of $\theta_0$. Call this constant $L$. Then
    \begin{align*}
        \| \frac{ h( \theta + x)}{h(\theta)} - 1 \| &= \| \frac{ h(\theta + x) - h(\theta)}{h(\theta)} \| \\
        & = \frac{ \| h(\theta + x) - h(\theta) \|}{ h(\theta)} \\
        & \leq \frac{ L}{\inf_{\theta \in U} h(\theta)} x
    \end{align*}
    provided that $x$ is close to $0$.
    For $n$ large enough $\frac{t}{\sqrt{n}}$ will be close to $0$.
\end{proof}

We use the following lemma in the proof of \Cref{lem:ChangeOfMeasure}.
\begin{lem}\label{lem:MoveRatioBoundIntoSum}
    If $\nu$ is a measure and suppose that two sequences of non-negative functions satisfy the relation $f_n, g_n$ $\sup_x \| \frac{g_n (x)}{f_n (x)} - 1 \| \rightarrow 0$, then
    \[
        \left\|
        \frac{\int g_n (x) \dd \nu(x)}{\int f_n (x) \dd \nu(x)} - 1
        \right\| \rightarrow 0.
    \]
\end{lem}

\begin{proof}
    \begin{align*}
        \left\|
        \frac{\int g_n(x) \dd \nu(x)}{\int f_n(x) \dd \nu(x)} - 1
        \right\| &=  \left\|
        \frac{\int g_n(x) - f_n(x) \dd \nu(x)}{\int f_n(x) \dd \nu(x)}
        \right\|\\
        &\leq 
        \frac{\int \| g_n(x) - f_n(x) \| \dd \nu(x)}{\int f_n(x) \dd \nu(x)}\\
        &=\frac{\int f_n(x) \| \frac{g_n(x) - f_n(x)}{f_n(x)} \| \dd \nu(x)}{\int f_n(x) \dd \nu(x)} \\
        &\leq \frac{\int f_n(x) \sup_y \left \|\frac{g_n(y)}{f_n(y)} - 1 \right\| \dd \nu(x)}{\int f_n(x) \dd \nu(x)}\\
        &=  \sup_y \left \|\frac{g_n(y)}{f_n(y)} - 1 \right\| \rightarrow 0.
    \end{align*}
\end{proof}

We have now all the tools to prove the Bernstein-von Mises theorem in the Frailty model and the Errors-in-Variables model.
We will first introduce two different priors for the mixing distribution $F$, and then study the two examples in depth.

%% file: Priors/priors.tex
\section{Priors}\label{sec:priors}
To provide examples for our theorems, we will need to specify priors. We will use Dirichlet process mixtures and mixtures of finite mixtures as priors.
Let $\Pi$ denote either a finite mixture or a Dirichlet process, and $\pi$ the prior on $\theta$.
Then we create the mixture prior by
\begin{align*}
  \theta &\sim \pi\\
  F \given \theta &\sim \Pi\\
  X \given \theta, F &\sim p_{\theta, F}.
\end{align*}
If $\Pi$ is the Dirichlet process we call the resulting distribution a Dirichlet process mixture and if $\Pi$ is a finite mixture, we call the result a mixture of finite mixtures.
\input{Priors/DP.tex}

\input{Priors/MFM.tex}

%% file: Priors/DP.tex
\subsection{The Dirichlet process}
One of the most common priors within Bayesian nonparametrics is the Dirichlet process.
This prior has various equivalent definitions.
For our purposes, the following definition is probably the most useful:
\begin{definition}\label{def:DirichletProcess}
  We say that a random measure $F$ on $\mathcal{X}$ is distributed as a Dirichlet process with centre measure $G$ and prior mass $M$, if for every finite partition $ (A_1, \dots, A_k)$ of $\mathcal{X}$
  \[
    \left(F(A_1),\dots, F(A_k)\right) \sim \textrm{Dir}\left(k; MG(A_1),\dots, MG(A_k)\right).
  \]
\end{definition}
It is not immediately clear from this definition, but $F$ is almost surely discrete.
This becomes clear by~\cite{sethuramanConstructiveDefinitionDirichlet}, see also~\cite[Chapter 4]{ghosalFundamentalsNonparametricBayesian2017}.

In order to apply the remaining mass to rule out that there are at most $k_n = o (\sqrt{n})$ distinct observations in the posterior we need to have a prior tail-bound. 
The following lemma gives a prior tail-bound for distinct observations under a Dirichlet process prior:
\begin{lem}\label{lem:tailbound_distinct_obs_DP}
    Let $K_n$ be the distinct observations in the model $X_1, \dots, X_n | F \sim F$ where $F \sim \textrm{DP} (MG)$. Let $c, \beta> 0$. Then there exists a $C > 0$, which only dependents on $M, \beta$ and $c$, such that
  \[
      \PP( K_n > c n^\beta )\leq  e^{- Cn^\beta \log(n)}.
  \]
\end{lem}

\begin{proof}
  Note that $K_n = \sum_{i= 1}^n D_i$ where $D_i \overset{\text{ind}}\sim \textrm{Ber} (\frac{1}{M + i - 1})$. Since the probability of $0$ of a $\textrm{Ber} (p)$ is $1-p \leq e^{-p}$, it follows that $\text{Ber} (p)$ is stochastically smaller than $\textrm{Poisson} (p)$. Therefore, $K_n$ is stochastically smaller than $\textrm{Poisson} (\sum_{i = 1}^n \frac{1}{M + i - 1})$. This in turn is stochastically smaller than $\textrm{Poisson} (C \log(n))$ for some constant $C> 0$. Let $X \sim \textrm{Poisson} (\mu)$. Then $\EE[ e^{\lambda X}] = e^{\mu (e^{\lambda} - 1)}$. Hence, by Chernoff inequality we find 
  \[
    \PP(X > x) \leq e^{-\lambda x } e^{\mu( e^\lambda - 1)}.
  \]
  Optimizing over $\lambda$ finds $e^\lambda = \frac{X}{\mu}$. Hence,
  \[
    \PP(X > x) \leq e^{-x \log \frac{\mu}{x} + x - \mu}
  \]
  By picking $\mu = c' \log (n)$, $x = c n^\beta$ and using the stochastic domination we find that there exists a $C> 0$, only depending on $M$ and $c$, such that
  \[
    \PP( K_n > c n^\beta) \leq  e^{- Cn^\beta \log(n)}.
  \]
\end{proof}

%% file: Priors/MFM.tex
\subsection{Finite mixtures}
Finite mixtures provide an alternative to the Dirichlet process.
To specify a finite mixture, we need to provide a centre measure $G$, a mass $M > 0$ and a mixing measure $H$.
\begin{definition}\label{def:FiniteMixture}
  We call a random measure $F$ on $\mathcal{X}$ a mixture of finite mixtures with centre measure $G$, mass $M > 0$ and mixing measure $H$, if it is distributed as the random variable $F$ in the following scheme:
\begin{align*}\label{scheme:FiniteMixture}
  K &\sim H,\\
 Z_1,\dots, Z_k \given K &\iid G,\\
 (w_1,\dots, w_K) \given K, Z_1, \dots, Z_k &\sim \textrm{Dir}(K; M, \dots, M),\\
 F &= \sum_{j = 1}^K w_j \delta_{z_j}.
\end{align*}
\end{definition}
This scheme provides a mixture over finite distributions.
The amount of support points is random but unbounded if $H$ has infinite support.
For the purposes of the remainder of this paper, we consider $H$ to be the geometric distribution with rate $\lambda > 0$.

%% file: Frailty/Frailty.tex
\section{Exponential Frailty}\label{sec:frailty}
\input{Frailty/model.tex}
\input{Frailty/BvM.tex}

%% file: Frailty/model.tex
\subsection{Model}
We observe a sample from the distribution of $ (X,Y)^T$, where given $Z$ the variables
$X$ and $Y$ are independent and exponentially distributed with intensities $Z$ and $\theta Z$.
Hence, $F$ is a distribution on $ (0,\infty)$ and $\theta>0$.

\begin{definition}\label{def:FrailtyModel}
  A frailty model is defined by a kernel which is given by
  \[
    p_{\theta}(x, y | z) = z e^{-zx} \theta z e^{- \theta z y} = z^2 \theta e^{-z(x + \theta y)}.
  \]
  The frailty model distribution is defined to be
  \[
    p_{\theta, F} (x,y) = \int  z^2 \theta e^{-z(x + \theta y)} \dd F(z).
  \]
\end{definition}

\begin{remark}\label{lem:FrailtyLeastFavSubmodel}
    The least favourable submodel for the frailty model is given by
    \begin{align*}
        p_{\theta + t, F_t}(x, y) &= \int p_{\theta + t}(x,y \given z) \dd F_t \\
        &= \int p_{\theta + t}(x,y \given \phi_{t, \theta}(z)) \dd F(z)
    \end{align*}
    where
    \begin{align*}
        F_{t, \theta}(B) &= F\left( B \left( \frac{2 \theta}{2 \theta + t} \right) \right).\\
        \phi_{t, \theta}(Z) &= (1 + \frac{t}{2 \theta})(Z).
    \end{align*}
    For details and proofs of these results, see~\cite{vandervaartEfficientMaximumLikelihood1996} and references therein. 
\end{remark}

The score for the parametric model given by the kernel is given by
\[
    \dot l_\theta(x,y \given z) = \frac{1}{\theta} - zy.
\]
The efficient score in the mixture model is then given by
\[
    \tilde l_{\theta, F} = 
    \frac{
        \int \frac{1}{2} ( x - \theta y) z^3 \exp( - z(x + \theta y)) \dd f(z)
    }{
        \int \theta z^2 \exp( - z(x + \theta y)) \dd f(z)
    }.
\]

\subsection{Parametric model and inefficiency}
Observe that $W = \frac{Y}{X}$ is a random variable with density
\[
    f_W(w) = \frac{\theta}{ \left( w + \theta \right)^2}.
\]

This gives a score given by
\[
    \frac{\partial}{\partial \theta} \ell_{\theta}(W) = \frac{1}{b} - \frac{2}{b + W},
\]
and a Fisher information 
\[
    I_{\theta} = \frac{1}{3 \theta^2}.
\]

This transformation allows us to use a sequence $ (X_i, Y_i)$ of i.i.d.\ observations into an i.i.d.\ sequence $W_i$. 

We can model this in the usual Bayesian way by putting a prior on $\theta$ and compute the posterior.
Since the likelihood is smooth, we can create tests and apply the parametric Bernstein von Mises theorem.
This tells us that if our prior had a continuous density, the posterior will be asymptotically normal, centred at the MLE and with inverse Fischer information.
This is the asymptotic distribution of the maximum likelihood estimator.
However, this estimator can be inefficient, as can be seen by the semiparametric estimator having smaller inverse Fisher information. 

%% file: Frailty/BvM.tex
\subsection{Bernstein-von Mises}

\begin{lem}\label{lem:BvMFrailty}
  Let $G$ be a distribution which admits a density $g$ such that $g (z) \gtrsim (z + 1)^{-\alpha}$, for some $\alpha > 0$. 
  Suppose that $\pi$ is a distribution with a density $h$ which is continuously differentiable.
  Let $\Pi$ be either a Dirichlet process or Mixture of finite mixtures with centre measure $G$ and prior mass $M > 0 $.
  Let $p_{\theta, F}$ denote the frailty model (c.f.\ \ref{def:FrailtyModel}).
  In the model that $ (\theta, F) \sim \pi \otimes \Pi$ and $ (X, Y) \sim p_{\theta, F}$.
  Then the conditions of \Cref{thm:semiparametricBvM} hold for every frailty model with $\theta > 0$ and $F_0$.
\end{lem}

\begin{proof}
  Consistency can be verified directly by using \Cref{lem:posteriorconsistency} and the prior mass bounds provided by \Cref{lem:FrailtyPriorMassBoundDP,lem:FrailtyPriorMassBoundMFM}.
  
  Next we verify the LAN expansion. This is proven in \Cref{lem:FrailtyModelLAN}.

  Finally, we need to verify the change of measure condition. This is proven in \Cref{lem:FrailtyChangeOfMeasure}.

\end{proof}

The references lemmas and all the prerequisites for them can be found in \Cref{sec:AppendixFrailty}.

%% file: ErrorsInVars/ErrorsInVars.tex
\input{ErrorsInVars/Model.tex}

\input{ErrorsInVars/BvM.tex}

%% file: ErrorsInVars/Model.tex
\section{Errors-in-variables}\label{sec:ErrorsInVars}
\subsection{Model}
Consider observing a sample from the distribution of a two-dimensional vector $X$ that
can be represented as $X=Y+\e$, for independent vectors $Y$ and $\e$,
where $Y$ has an unspecified distribution 
that is concentrated on an unspecified  line $L=\{a+\l b: \l\in \RR\}\subset \RR^2$, and
$\e$ is bivariate normally distributed  with mean zero and unknown covariance matrix $\Sigma$.
The parameter of interest is the slope of the line $L$, or, equivalently, a normalised version
of the vector $b$. We may further parameterize $Y=a+Z b$,  for a univariate variable $Z$ with 
unknown distribution. The distribution of $X$ given $Z$ is bivariate normal with mean $a+Z b$ and covariance matrix $\Sigma$.
The parameters are $ (a,b,\Sigma)$ and the distribution $F$ of $Z$. We are interested
in the slope $b_2/b_1$.

This is a version of the errors-in-variables model:  provided that $b_1\not=0$, we have $Y_2=a_2+b_2(Y_1-a_1)/b_1$ 
and
\begin{align*}
X_1&=Y_1+\e_1,\\
X_2&= a_2-\frac{a_1b_2}{b_1}+\frac{b_2}{b_1}Y_1+\e_2.
\end{align*}

So $Y_1$ is an independent variable that is observed with the error $\e_1$, and $Y_2$ is 
a linear regression on $Y_1$ with error $\e_2$. The errors are assumed bivariate normal.
The slope parameter $b_2/b_1$ was shown to be identifiable if
the distribution of $Z$ is not normal (and not degenerate) in~\cite{reiersolIdentifiabilityLinearRelation1950}.

\begin{definition}\label{def:ErrorsInVarsModel}
  The errors in variables model is defined by a kernel which is given by
  \[
    p_\theta(x,y \given z) = \phi_\sigma(x - z) \phi_\tau( y - \alpha - \beta z),
  \]
  where $\phi_s$ is the density of the mean zero normal distribution with variance $s^2$.
  The errors in variables model is defined to be
  \[
    p_{\theta, F}(x,y) = \int \phi_\sigma(x - z) \phi_\tau(y - \alpha - \beta z) \dd F(z).
  \]
\end{definition}

\begin{remark}\label{rem:LeastFavSubmodelErrorsInVars}
    The least favourable submodel for the Errors in variables model is given by
    \begin{align*}
        p_{\theta + t, F_t}(x) &= \int p_{\theta + t}(x \given z) \dd F_t \\
        &= \int p_{\theta + t}(x \given \phi_{t,\theta}(z)) \dd F(z),
    \end{align*}
    where
    \begin{align*}
        F_{t, \theta}(B) &= F( B\left( 1 - \frac{t_2 \beta}{1 + \beta^2} \right)^{-1} - \frac{t_1 \beta}{1 + \beta^2})\\
        \phi_{t, \theta}(Z) &= Z\left( 1 - \frac{t_2 \beta}{1 + \beta^2} \right) + \frac{t_1 \beta}{1 + \beta^2}.
    \end{align*}
    For details, see~\cite{murphyLikelihoodInferenceErrorsVariables1996}.
\end{remark}

Efficient estimators for the slope $b_2/b_1$ (equivalently $\q$) were constructed in~\cite{bickelEfficientEstimationErrors1987}  and~\cite{vandervaartEstimatingRealParameter1988,murphyLikelihoodInferenceErrorsVariables1996,vandervaartEfficientMaximumLikelihood1996}.

%% file: ErrorsInVars/BvM.tex
\subsection{Bernstein-von Mises}

\begin{lem}\label{lem:BvMErrorsInVars}
  Let $G$ be a distribution which admits a density $g$ such that $g$ satisfies $g(z) \gtrsim |z|^{-\delta}$ for some $\delta > 0$ and the assumptions posed in \Cref{lem:SmoothnessInGSufficient}.
  Suppose that $\pi$ is a distribution with a density $h$ which is continuously differentiable.
  Let $\Pi$ be either a Dirichlet process or Mixture of finite mixtures with centre measure $G$ and prior mass $M > 0 $.
  Let $p_{\theta, F}$ denote the errors-in-variables model (c.f.~\ref{def:ErrorsInVarsModel}).
  In the model that $ (\theta, F) \sim \pi \otimes \Pi$ and $ (X, Y) \sim p_{\theta, F}$.
  Then the conditions of \Cref{thm:semiparametricBvM} hold for every errors-in-variables model with $\theta$ and $F_0$ with $F_0(z^{7+ \delta})< \infty$ for some $\delta > 0$.
\end{lem}

\begin{proof}
  Consistency can be verified directly by using \Cref{lem:posteriorconsistency} and the prior mass bounds provided by \Cref{lem:ErrorsInVarsPriorMassEpsDP,lem:ErrorsInVarsPriorMassEpsMFM}.
  
  Next we verify the LAN expansion. This is proven in \Cref{lem:LanErrorsInVars}.

  Finally, we need to verify the change of measure condition. This is proven in \Cref{lem:ErrorsInVarsChangeOfMeasure}.

\end{proof}

The references lemmas and all the prerequisites for them can be found in \Cref{sec:AppendixErrorsInVars}.

%% file: Discussion/Conclusion.tex
\section{Conclusion}
In this article we have seen new ways of proving consistency and asymptotic optimality of the Bayesian posterior in semiparametric mixture models.
The consistency results hold in fair generality, they should be able to cover most of the models used in practise.
For the asymptotic optimality we have proven a Bernstein-von Mises theorem.
The Bernstein-von Mises theorem lets us conclude that under its conditions, the Bayesian posterior is asymptotically optimal and yields valid uncertainty quantification.
We then provide tools for verifying the assumptions.
This requires a special structure of the least-favourable submodel, which holds for the special models we have considered.
These special models are the Frailty model and, the Errors-in-Variables model.
This means that we can reliably use semiparametric mixture models with the right species sampling process priors in inference for these models.
This enables researchers to incorporate their prior knowledge into these studies without compromising the soundness of the inference.

The results on the Frailty model and Errors-in-Variables model can be extended to extensions of these models, however, there exists different models for which this assumption is not true.
For these models, new techniques need to be developed, which we aim to do in future research.

%% file: Frailty/appendix.tex
\section{Exponential frailty}\label{sec:AppendixFrailty}
In this appendix we have collected all the results for the exponential frailty model.
\input{Frailty/LANNew.tex}
\input{Frailty/ChangeOfMeasure.tex}
\input{Frailty/L1Approx.tex}

\input{Frailty/LRbounds.tex}

\input{Frailty/PriorMass.tex}

\input{Frailty/TechnicalLemmas.tex}

%% file: Frailty/LANNew.tex
\subsection{LAN expansion for the Frailty model}
To prove the LAN expansion, we will use \Cref{Lem:LANByDonsker}.
We use \Cref{lem:ExpressionsForGradiantHessian} to compute expressions for the gradient and the Hessian of the perturbed model.

The perturbed kernel is given by
\begin{align*}
    p_{\theta_t, F_t} &= \int z^2 ( \theta + t) e^{-z( x + (\theta + t) y)} \dd F_t(z) \\
       &= \int \left( z \left( 1 - \frac{t}{2 \theta} \right) \right)^2 ( \theta + t) e^{ - z \left( 1 - \frac{t}{2 \theta}\right) ( x + (\theta + t)y)} \dd F(z).
\end{align*}
Note that 
\begin{align*}
    h_\theta(x ,y \given z) &= 2 \log(z) + \log(\theta) -z(x + \theta y).\\
    \phi_{t,\theta}(z) &= z\frac{1 - t}{2 \theta}.
\end{align*}
Thus, $h_{\theta + t} (x , y\given \phi_{t,\theta} (z))$ has gradient 
\begin{align*}
  \frac{\partial}{\partial t} &\left(h_{\theta + t}(x , y \given \phi_{t, \theta}(z))\right)\\
   &= \frac{\partial}{\partial t} \left(2 \log \left(z \frac{1 - t}{2 \theta}\right) + \log(\theta + t) - z\frac{ 1 - t}{2\theta} \left(x + (\theta + t)y \right)\right)\\
   &= \frac{2}{1 - t} + \frac{1}{\theta + t} + z \frac{x + (\theta + 2t - 1)y}{2 \theta}.
\end{align*}
This means that the first derivative is given by
    \[
        f' = \dot p_{\theta_t, F_t} = \int (a_1 + a_2 xz + a_3 yz) z^2 e^{-z( b_1 x + b_2 y) } \dd F(z)
    \]
    where $a_1, a_2, a_3, b_1, b_2$ are given by:
    \begin{align*}
           a_1 &= \frac{\left(t - 2\theta\right)\left(6 t \theta\right)}{3 \theta^3},
        && a_2 = \frac{\left(t - 2 \theta\right)^2 (t + \theta)}{8 \theta^3},
        && a_3 = \frac{ \left( t - 2 \theta\right)^2\left( \theta - 2t \right)(t + \theta)}{8 \theta^3},\\
           b_1 &= \frac{(2\theta - t)}{2 \theta} = 1 - \frac{t}{2 \theta},
        && b_2= \frac{(2\theta - t)(t + \theta)}{2 \theta} = \left( 1 - \frac{t}{2 \theta} \right)\left( \theta + t \right)
    \end{align*}
    Similarly, for the second derivative, we find
    \[
        f'' = \ddot p_{\theta_t, F_t} = \int (c_1 + c_2 xz + c_3 yz + c_4 x^2z^2 + c_5 xy z^2 + c_6 y^2 z^2) z^2 e^{-z(b_1x + b_2 y)} \dd F(z)
    \]
    where $b_1, b_2$ are as before and $c_1,\dots, c_6$ are given by
    \begin{align*}
           c_1 &= \frac{24 \theta^2 \left(t - \theta\right)}{16\theta^4}, 
        && c_2 = \frac{12 \theta (t - 2 \theta) t}{16 \theta^4}, 
        && c_3 = \frac{4 \theta (t - 2 \theta)( 7 t^2 - 4 t \theta - 2 \theta^2)}{16 \theta^4},\\
           c_4 &= \frac{ \left(t - 2 \theta\right)^2 \left( t + \theta\right)}{16 \theta^4},
        && c_5 = \frac{ \left(t - 2 \theta\right)^2 \left( t - \theta\right) \left( 2t - \theta \right) }{16 \theta^4},
        && c_6 = \frac{ \left(t - 2 \theta\right)^2 \left( t - \theta\right) \left( 2t - \theta \right)^2}{ 16 \theta^4}\\
    \end{align*}
    By now plugging in these expressions we find the first and second derivative of $\log p_{\theta_t, F_t}$. The first derivative is given by 
    \begin{align*}
        \frac{f'}{f} = \frac{\partial}{\partial t} \ell(t; \theta, F) &= 
        \frac{
            \int  (a'_1 + a'_2 x z + a'_3 yz)z^2 e^{-z( b_1 x + b_2 y)} \dd F(z)
        }{
            \int z^2 e^{-z( b_1 x + b_2 y)} \dd F(z)
        }
    \end{align*}
    where $b_1$ and $b_2$ are given as before and $a'_1, a'_2, a'_3$ are given as $a'_i = \frac{4 \theta^2 a_i}{\left(\theta + t \right)\left(2 \theta - t \right)^2}$.

    While for the second derivative we first compute $\frac{f''}{f}$.
    \begin{align*}
        \frac{f''}{f}&= 
        \frac{
            \int  (c'_1 + c'_2 x z + c'_3 yz + c_4 x^2 z^2 + c_5 xyz^2 + c_6 y^2z^2)z^2 e^{-z( b_1 x + b_2 y)} \dd F(z)
        }{
            \int z^2 e^{-z( b_1 x + b_2 y)} \dd F(z)
        }
    \end{align*}
    where $b_1, b_2$ are given as before, and $c_i' = \frac{ 4 \theta^2 c_i}{ \left(\theta + t \right) \left(2 \theta - t \right)^2}$.
    We can then finally plug this in to get 
    \[
        \frac{\partial^2}{\partial t^2} \ell(t; \theta, F) = \frac{f''}{f} - \left( \frac{f'}{f} \right)^2.
    \]

\begin{lem}\label{lem:FrailtyModelLAN}
    Assume that $\theta_0 > 0$ and $F_0(z^2 + z^{-5}) < \infty$, then 
    \Cref{EqRemainders} holds for $P_{\theta_0, F_0}$.
\end{lem}

\begin{proof}
By consistency, we can restrict to small neighbourhoods $\theta_0 \in \Theta, F_0 \in \mathcal{F}$.
Next, by~\cite[Corolarry 4.1]{vandervaartEfficientMaximumLikelihood1996}, and consistency we see that

\[
    \{ \dot\ell(t/\sqrt{n}; \theta, F) : \theta \in \Theta, F \in \mathcal{F} \}
\]
is a Donsker class.
By~\cite[Corolarry 4.1]{vandervaartEfficientMaximumLikelihood1996} we also show that the first derivative has a square integrable envelope and the second derivative has an integrable envelope. 
Furthermore, in the notation of \Cref{lem:LimitsOfExpectationsLogLik}, $h_{\theta} (x \given z)$ is twice differentiable in $\theta$ and $z$, and $\phi_{t, \theta} (z)$ is twice differentiable in $\theta, t$. 
Moreover, the gradient and the hessian remain bounded.
Thus, we can apply \Cref{lem:LimitsOfExpectationsLogLik}.
Together with the Donsker condition and the discussion after \Cref{Lem:LANByDonsker} we can verify all the assumptions of \Cref{Lem:LANByDonsker} and conclude that \Cref{EqRemainders} holds.
\end{proof}

%% file: Frailty/ChangeOfMeasure.tex
\subsection{Verifying the change of measure condition for the exponential frailty model}

\begin{lem}\label{lem:FrailtyChangeOfMeasure}
  Under the conditions of \Cref{lem:BvMFrailty},~\eqref{EqChangeOfMeasure} holds. 
\end{lem}

\begin{proof}
  We want to verify the assumptions of \Cref{lem:ChangeOfMeasure}.

  The least favourable submodel is given in \Cref{lem:FrailtyLeastFavSubmodel}.
  Thus, \Cref{ass:ChangeOfVariables} is verified.
  We can verify posterior consistency in the perturbed model.
  By \Cref{lem:FrailtyPriorMassBoundDP} and \Cref{lem:FrailtyPriorMassBoundMFM} each neighbourhood of the truth has positive mass under the DP prior and the MFM prior respectively.
  Hence, by \Cref{lem:posteriorconsistency} and \Cref{remark:ConsistencyPerturbedPosterior} and the consistency proof in~\cite{vandervaartEfficientMaximumLikelihood1996} we satisfy \Cref{ass:Consistency}.

  Next, we verify \Cref{ass:Smoothness}.

  First we can verify the conditions on $h$. We have assumed that $h$ continuously differentiable in a neighbourhood of $\theta_0$. 
  By applying \Cref{lem:SmoothnessInH} we verify the conditions on $h$.
  
  Now we want to verify the conditions on $G$. 
  We want to verify the assumptions of \Cref{lem:SmoothnessInGSufficient}.
  The least favourable submodel satisfies the assumption.
  Moreover, we assumed the density of the centre measure satisfy the requirements of the lemma.
  Hence, by \Cref{lem:SmoothnessInGSufficient} gives us the conditions to apply \Cref{lem:SmoothnessInG}.
  This means that \Cref{ass:smoothnessinG} is satisfied.

  Finally, we need to verify \Cref{ass:SmallNumberClusters}. For both models we will verify this assumption using the remaining mass theorem.
  \textbf{Dirichlet process case}

  Let $k_n = n^{1/3} = o (\sqrt{n})$.
  Then by \Cref{lem:tailbound_distinct_obs_DP} we have the following upper bound on the prior mass:
  \[
    \PP\left( K_n > n^{1/3} \right) \leq e^{- C n^{1/3} \log n}
  \]
  Let $\epsilon_n = \frac{1}{\sqrt{n}}$.
  By \Cref{lem:FrailtyPriorMassBoundDP} we have
  \[
    \Pi_t\left( \textrm{KL}^2 + V_2 \left( p_{\theta_0, F_0}; p_{\theta, F} \right) < \epsilon_n \right) \gtrsim e^{c \log( \frac{1}{\epsilon_n})^3}.
  \]
  In order to apply the remaining mass theorem, we have to show that
  \[
    \frac{e^{- C n^{1/3} \log n}}{e^{c \log( \frac{1}{\epsilon_n})^3}} = o(e^{-2 n \epsilon_n^2}) = o(1).
  \]
  This is the same as showing 
  \[
    e^{ \log(n)^3 - n^{1/3}\log(n)} \rightarrow 0
  \]
  Which is true. 
  Thus, by~\cite[Theorem 8.20]{ghosalFundamentalsNonparametricBayesian2017} we have verified \Cref{ass:SmallNumberClusters} in case of the Dirichlet process.

  \textbf{Mixture of finite mixtures case}

  Let $k_n = n^{1/3} = o (\sqrt{n})$.
  Because we assumed that $H$ is the geometric distribution,
  \[
    \PP\left( K \geq \floor{k_n} \right) = \lambda(1 - \lambda)^{\floor{k_n}} \lesssim e^{-c n^{1/3}}.
  \]
  Let $\epsilon_n = \frac{1}{\sqrt{n}}$.
  By \Cref{lem:FrailtyPriorMassBoundMFM} we have
  \[
    \Pi_t( \textrm{KL}^2 + V_2 \left( p_{\theta_0, F_0}; p_{\theta, F} \right) < \epsilon) \gtrsim e^{c \log( \frac{1}{\epsilon})^3}.
  \]
  In order to apply the remaining mass theorem, we have to show that
  \[
    \frac{e^{- C n^{1/3}}}{e^{c \log( \frac{1}{\epsilon_n})^3}} = o(e^{-2 n \epsilon_n^2}) = o(1).
  \]
  This is the same as showing 
  \[
    e^{ \log(n)^3 - n^{1/3}} \rightarrow 0
  \]
  Which is true. 
  Thus, by~\cite[Theorem 8.20]{ghosalFundamentalsNonparametricBayesian2017}, we have verified \Cref{ass:SmallNumberClusters} in case oof the finite mixtures.

\end{proof}

%% file: Frailty/L1Approx.tex
\subsubsection{Approximation in Total variation distance}
In this section we produce suitable approximations in total variation distance to the true distribution.
First we control the $\theta$ part of the approximation.
\begin{lem}\label{lem:L1boundintheta}
    For all $\theta, \theta_0 >0 $ and distributions $F$ it holds that 
    \[
        \| p_{\theta, F} - p_{\theta_0, F} \|_1 \leq 2 \frac{ \| \theta - \theta_0 \|}{\theta_0}.
    \]
\end{lem}

\begin{proof}
    It is most convenient to use
    \[
        p_{\theta, F} = \int z e^{-z x} z \theta e^{- z \theta y} \dd F(z).
    \]
    Using Jensen's inequality 
    \begin{align*}
        \|p_{\theta, F} - p_{\theta_0, F} \|_1 &= \| \int z e^{-z x} z \theta e^{- z \theta y} \dd F(z) - \int z e^{-z x} z \theta_0 e^{- z \theta_0 y} \dd F(z) \\
        & \leq \int \|  z e^{-z x} z \theta e^{- z \theta y} -  z e^{-z x} z \theta_0 e^{- z \theta_0 y} \|_1 \dd F(z) \\
    \end{align*}
    Now we can use~\cite[Lemma B.8]{ghosalFundamentalsNonparametricBayesian2017} which states that the total variation distance between products of independent random variables is less than the sum of their individual total variation distances. 
    Since the first density is the same, this yields as distance of zero.
    Moreover, we can remove the $Z$ scaling by using the scale invariance of the total variation metric.
    This leads to the following upper bound: 
    \[
        \|p_{\theta, F} - p_{\theta_0, F} \|_1  \leq \| \theta e^{-\theta y} - \theta_0 e^{- \theta_0 y} \|_1.
    \]
    
    By \Cref{lem:TVExponentials} expansion we find
    \[
        \| \theta e^{-\theta y} - \theta_0 e^{- \theta_0 y} \| \leq 2 \frac{\| \theta - \theta_0\|}{\theta_0}.
    \]

\end{proof}

First, we can cut of the tails of any $F_0$ while incurring a small error.
\begin{lem}\label{lem:FrailtycompactApprox}
    Let $F_0$ be a probability distribution on $[0, \infty)$ such that there exists a $\gamma > 0$ such that $\int z^\gamma + z^{-\gamma} \dd F_0(z) < \infty$.
    Then for every $0< \epsilon < \frac{1}{4}$ there exists a probability distribution $F^*$ supported on $[ \frac{1}{\EE[Z^{-\gamma}]^{1/\gamma}} \epsilon^{-1/\gamma},  \EE[Z^\gamma]^{1/\gamma} \epsilon^{1/\gamma}]$ such that
    \[
        \| p_{\theta_0, F_0} - p_{\theta_0, F^*} \|_1 < 4 \epsilon.
    \]
\end{lem}
\begin{proof}
    Let $A_0 = [0, \frac{1}{\EE[Z^{-\gamma}]^{1/\gamma}} \epsilon^{-1/\gamma}) \cup (\EE[Z^\gamma]^{1/\gamma} \epsilon^{1/\gamma}, \infty)$
    Define $F^* = \frac{1}{1 - F_0(A_0)} F (\II_{A_0^c})$.
    Then by the triangle inequality we get
    \[
        \| p_{\theta_0, F_0} - p_{\theta_0, F^*} \|_1 \leq \iint \int_{A_0} p_{\theta_0}(x,y | z) \dd F_0(z) \dd x \dd y + ( \frac{1}{1 - F_0(A_0)} - 1 ) \| p_{\theta_0, F^*} \|_1.
    \]
    Since $F^*$ is a probability distribution, $p_{\theta_0, F^*}$ is a probability distribution and hence has $L^1$ norm $1$.
    Furthermore, note that $ \iint \int_{A_0} p_{\theta_0} (x,y | z) \dd F_0(z) \dd x \dd y = F_0(A_0)$.
    This shows that
    \[
        \| p_{\theta_0, F_0} - p_{\theta_0, F^*} \|_1 \leq F_0(A_0) + \left( \frac{1}{1 - F_0(A_0)} - 1 \right)
    \]
    Then $F_0(A_0) = F_0 ([0, \frac{1}{\EE[Z^{-\gamma}]^{1/\gamma}} \epsilon^{-1/\gamma})) + F_0 ((\EE[Z^\gamma]^{1/\gamma} \epsilon^{1/\gamma}, \infty))$. 
    By Markov's inequality, we can bound these quantities.
    \[
        \PP(Z <  \EE[Z^\gamma]^{1/\gamma} \epsilon^{1/\gamma} ) = \PP(Z^{-\gamma} > \left( \frac{1}{\EE[Z^{-\gamma}]^{1/\gamma}} \epsilon^{-1/\gamma}\right)^{-\gamma}) \leq \frac{\EE[ Z^{-\gamma}]}{\left( \frac{1}{\EE[Z^{-\gamma}]^{1/\gamma}} \epsilon^{-1/\gamma}\right)^{-\gamma}} =  \epsilon
    \]
    and
    \[
        \PP(Z >  \EE[Z^\gamma]^{1/\gamma} \epsilon^{1/\gamma} ) = \PP(Z^\gamma > \left( \EE[Z^\gamma]^{1/\gamma} \epsilon^{1/\gamma}\right)^\gamma) \leq \frac{ \EE[ Z^\gamma]}{\left( \EE[Z^\gamma]^{1/\gamma} \epsilon^{1/\gamma}\right)^\gamma} = \epsilon.
    \]
    Thus, $F_0(A_0) < 2 \epsilon$. 
    When $F_0(A_0) < \frac{1}{2}$
    \[
        F_0(A_0) + \left( \frac{1}{1 - F_0(A_0)} - 1 \right) \leq 4 F_0(A_0).
    \]
\end{proof}
We can approximate every mixture using only a finite number of components. 
\begin{lem}\label{lem:FrailtyfiniteApproxLin}
    Let $F$ be a finite measure on $[m, M]$ with $0 < m < M < \infty$. 
    For every $0 < \epsilon < \frac{1}{2}$ there exists a discrete measure $F_\epsilon$ with $|F| = |F_\epsilon|$, with fewer than $c \frac{M}{m} \log(\frac{1}{\epsilon})$ support points, for some universal constant $C> 0$, such that
    \[
        \| p_{\theta, F} - p_{\theta, F_0} \| < |F|\epsilon.
    \]
\end{lem}

\begin{proof}
    We start by using a Taylor expansion for $e^{-x}$. 
    \[
        e^{-x} = \sum_{i = 1}^{K-1} \frac{(-x)^i}{i!} + R(x),
    \]
    with $R (x) = \frac{e^{-\xi_x x} x^k}{k!}$ for some $0 <\xi_x <1$. 
    So when $\int z^2 z^i \dd F (z) - \int z^2 z^i \dd G (i)$ for $i = 0,\dots, K$ we get
    \[
        | \int_m^M z^2 e^{-x z} \dd (F - G)(z) | \leq \int_m^M \frac{ (xz)^K}{K!} \dd (F + G)(z) \leq \frac{2(xM)^K}{K!}.
    \]
    We can match $K$ moments using $K$ support points by~\cite[Lemma L.1 G+V]{ghosalFundamentalsNonparametricBayesian2017}.\\
    By using that $K!> \left(\frac{e}{K} \right)^K$ we find that $2 \frac{(xM)^K}{K!} \leq 2 \left(\frac{ e xM}{K} \right)^K$.
    We can bound the $L_1$ norm as follows:
    \[
        \iint | \int_m^M p_{\theta}(x,y | z) \dd (F - G)(z) | \dd x \dd y \leq |F|\left(\iint_{x + \theta y > L} m^2 \theta e^{-m(x + \theta y)} \dd x \dd y + 2 \left( \frac{ e LM}{K} \right)^K\right).
    \]
    Note that 
    \[
        \iint_{x + \theta y > L} m^2 \theta e^{-m(x + \theta y)} \dd x \dd y \leq e^{-mL/2}
    \]
    If we pick $L = \frac{2}{m} \log(\frac{1}{\epsilon})$ we can bound this by $\epsilon$.
    Next we choose $K$ so that $2 \left(\frac{eLM}{K} \right)^K < \epsilon$. 
    For example, it suffices that $K > 2 e L M = \frac{4eM}{m} \log(\frac{1}{\epsilon})$ and $\left(\frac{1}{2} \right)^K < \epsilon$. 
    Hence, $K \sim \frac{4 e M}{m} \log\left(\frac{1}{\epsilon}\right)$ suffices.

    This shows that
     \[
        \| p_{\theta, F} - p_{\theta, F_0} \| < |F|\epsilon.
    \]
\end{proof}
We can use this Lemma to improve our approximation.
\begin{lem}\label{lem:FrailtyfiniteApproxLog}
    Let $F$ be a finite measure on $[m, M]$ with $0 < m < M < \infty$. 
    For every $0 < \epsilon < \frac{1}{2}$ there exists a discrete measure $F_\epsilon$ with $|F| = |F_\epsilon|$, with fewer than $c \log(\frac{M}{m}) \log(\frac{1}{\epsilon})$ support points, for some universal constant $C> 0$, such that
    \[
        \| p_{\theta, F} - p_{\theta, F_0} \| < |F|\epsilon.
    \]
    Moreover, in every interval $I_i = [e^i m, e^{i + 1}m]$ there are atoms unless $F (I_i) = 0$.
\end{lem}

\begin{proof}
    We partition $[m, M]$ into $\log(\frac{M}{m})$ many intervals $I_i$, where $I_i = [a_i, a_{i+1}]$, with $a_1 = m$ and $a_{i+1} = e a_i$. We then apply \Cref{lem:FrailtyfiniteApproxLin} to each of these intervals to construct our approximation and then the triangle inequality gives the final result.
\end{proof}

The next step is building an approximation of $p_{\theta_0, F_0}$ by $p_{\theta_0, F_{\epsilon}}$ using a simple enough $F_\epsilon$.
This Lemma is used to create the approximation bounds for the Dirichlet process prior.
\begin{lem}\label{lem:L1boundinF}
    Let $ (0, \infty) = \cup_{j = 0}^\infty A_j$ be a partition and $F_N = \sum_{j = 1}^N w_j \delta_{z_j}$ a probability measure with $Z_j \in A_j$ for $j = 1, \dots, N$. 
    Then
    \[
        \| p_{\theta_0, F} - p_{\theta_0, F_N} \|_1 \leq 4 \max_{1 \leq j \leq N} \frac{ \textrm{diam} A_j}{\min(a: a \in A_j)} + \sum_{j = 1}^N | F(A_j) - w_j | + F(A_0).
    \]
\end{lem}

\begin{proof}
    We denote $t = x + \theta y$. 
    First step is to compute an expression for $p_{\theta_0, F_0} - p_{\theta_0, F_n}$:
    \begin{align*}
        \int \theta z^2 e^{-zt} \dd \left(F - F_n\right) (z) &= \sum_{j = 1}^N \int_{A_j} z^2 \theta e^{-zt} - z_j^2 \theta e^{-z_j t} \dd F(z) \\
                                               &+ \sum_{j = 1}^N \theta z_j^2 e^{-z t} \left( F(A_j) - w_j \right) \\
                                               &+ \int_{A_0} \theta z^2 e^{-zt} \dd F(z).
    \end{align*}
    We start with the last two terms.
    By Fubini, positivity, and $p_{\theta} (x,y | z)$ being a probability density it follows that
    \[
        \| \int_{A_0} z^2 \theta e^{-z(x + \theta y)} \dd F(z) \|_1 \leq \int_{A_0} \iint z^2 \theta e^{-z( x + \theta y)} \dd x \dd y \dd F(z) \leq F(A_0).
    \]
    Similarly, for 
    \[
        \| \theta z_j^2 e^{-z t} \left( F(a_j) - w_j \right)\|_1 \leq \left( F(a_j) - w_j \right).
    \]
    Thus, we only need to show that
    \[
        \| \sum_{j = 1}^N \int_{A_j} z^2 \theta e^{-zt} - z_j^2 \theta e^{-z_j t} \dd F(z) \|_1 \leq \max_{1 \leq j \leq N} \frac{\textrm{diam}(A_j)}{\min(a: a \in A_j)}.
    \]
    First we use the triangle inequality to get
    \[
        \| \sum_{j = 1}^N \int_{A_j} z^2 \theta e^{-zt} - z_j^2 \theta e^{-z_j t} \dd F(z) \|_1 \leq  \sum_{j = 1}^N \int_{A_j} \| z^2 \theta e^{-zt} - z_j^2 \theta e^{-z_j t} \|_1 \dd F(z). 
    \]
    By applying~\cite[Lemma B.8.i]{ghosalFundamentalsNonparametricBayesian2017}, and using the scale invariance of the total variation distance, we get that
    \[
        \| z^2 \theta e^{-zt} - z_j^2 \theta e^{-z_j t} \|_1 \leq 2 \| z e^{-z t} - z_j e^{-z_j t} \|_1.
    \]
    By \Cref{lem:TVExponentials} and we find
    \[
        \| z e^{-z t} - z_j e^{-z_j t} \|_1 \leq 2 \frac{ |z - z_j|}{z_j}.
    \]
    Combining this yields
    \[
        \| \sum_{j = 1}^N \int_{A_j} z^2 \theta e^{-zt} - z_j^2 \theta e^{-z_j t} \dd F(z) \|_1 \leq \sum_{j = 1}^N \int_{A_j} 4 \frac{ \textrm{Diam}(A_j)}{\min A_j} \dd F(z) \leq 4 \max_{1 \leq J \leq N} \frac{\textrm{Diam}(A_j)}{\min A_j }.
    \]
    Adding the final 3 bounds yields the claimed result.

\end{proof}

Next we prove a counterpart for this approximation result, but now so we can use it for mixtures of finite mixtures.

\begin{lem}\label{lem:L1boundinFMFM}
    Let $F_N = \sum_{j = 1}^N w_j \delta_{z_j}$ and $F'_{N'} = \sum_{j = 1}^{N'} w'_j \delta_{z'_j}$.
    Then for all $K \leq \min (N, N')$:
    \begin{align*}
        \| p_{\theta_0, F_N} - p_{\theta_0, F'_{N'}} \|_1 &\leq 4 \max_{1 \leq j \leq K} \frac{ | z_j - z'_j|}{z_j} \\
        & + \sum_{j = 1}^K |w_j - w'_j| + \sum_{j = K+ 1}^N w_j + \sum_{j = K + 1}^{N'} w'_j.
    \end{align*}
\end{lem}

\begin{proof}
    Write $t = x + \theta_0 y$.
    \begin{align*}
        p_{\theta_0, F_n} - P_{\theta_0, F'_{N'}} &= \sum_{j = 1}^K w_j \left(p_{\theta_0}((x,y) \given Z_j) - p_{\theta_0}( (x,y) \given z'_j)\right) \\
            &+ \sum_{j = 1}^{K} \left(w'_j - w_j \right) p_{\theta_0} ( (x,y) \given z'_j) \\
            & + \sum_{j = K+1}^N w_j p_{\theta_0}( (x,y) \given z_j) + \sum_{j = K+1}^{N'} w'_j p_{\theta_0} ( (x,y) \given z'_j).
    \end{align*}
    Plugging this into the total variation norm and using the triangle inequality gives
    \begin{align*}
        \| p_{\theta_0, F_n} - p_{\theta_0, F'_{N'}} \|_1 & \leq \|\sum_{j = 1}^K w_j \left(p_{\theta_0}((x,y) \given Z_j) - p_{\theta_0}( (x,y) \given z'_j)\right) \|_1\\
            &+ \|\sum_{j = 1}^{K} \left(w'_j - w_j \right) p_{\theta_0} ( (x,y) \given z'_j)\|_1 \\
            & + \|\sum_{j = K+1}^N w_j p_{\theta_0}( (x,y) \given z_j)\|_1 + \|\sum_{j = K+1}^{N'} w'_j p_{\theta_0} ( (x,y) \given z'_j)\|_1.  
    \end{align*}
    We start with the last three terms. By Fubini, positivity and $p_{\theta} ((x,y) \given z)$ being a probability density, it follows that
    \[
        \|\sum_{j = K+1}^N w_j p_{\theta_0}( (x,y) \given z_j)\|_1 \leq \sum_{j = K + 1}^N w_j \| p_{\theta_0}( (x,y) \given z_j) = \sum_{j = K+1}^N w_j
    \]
    Similarly for the fourth term.
    For the second term,
    \[
        \|\sum_{j = 1}^{K} \left(w'_j - w_j \right) p_{\theta_0} ( (x,y) \given z'_j)\|_1 \leq \sum_{j = 1}^K |w_j - w'_j| \| p_{\theta_0}( (x,y) \given z'_j) \|_1 = \sum_{j = 1}^K | w_j - w'_j|
    \]
    So we only need to deal with the first term.
    \[
        \|\sum_{j = 1}^K w_j \left(p_{\theta_0}((x,y) \given Z_j) - p_{\theta_0}( (x,y) \given z'_j)\right) \|_1 \leq \sum_{j = 1}^k w_j  \| p_{\theta_0}( (x,y) \given z_j) - p_{\theta_0}( (x,y) \given z'_j) \|_1
    \]
    Since the $w_j$ sum to something less than or equal to $1$, it suffices to bound
    \[
        \max_{1 \leq j \leq K}  \| p_{\theta_0}( (x,y) \given z_j) - p_{\theta_0}( (x,y) \given z'_j) \|_1.
    \]
    By applying~\cite[Lemma B.8.i]{ghosalFundamentalsNonparametricBayesian2017}, and using the scale invariance of the total variation distance, we get that
    \[
      \| {z'_j}^2 \theta e^{-z'_j t} - z_j^2 \theta e^{-z_j t} \|_1 \leq 2 \| z'_j e^{-z'_j t} - z_j e^{-z_j t} \|_1.
    \]
    By \Cref{lem:TVExponentials} and we find
    \[
        \| z'_j e^{-z'_j t} - z_j e^{-z_j t} \|_1 \leq 2 \frac{ |z'_j - z_j|}{z_j}.
    \]
    Combining all the results yields the claimed result.
\end{proof}

%% file: Frailty/LRbounds.tex
\subsubsection{Bounding the Kullback-Leibler divergence}
In this section we prove the bounds on the likelihood ratio that we will need.

This lemma is needed to bound the KL divergence in terms to the total variation distance.
\begin{lem}\label{lem:ExpectedLR}
    Let $F_0$ be a fixed probability distribution supported on $0 \leq a \leq b \leq \infty$. 
    If $b$ is infinite we consider the support $[a, \infty)$. 
    Suppose that $0 \leq \epsilon < \frac{1}{2}$, $0 < \delta < \frac{1}{3}$ and $\int z^{2 \delta} \dd F_0(z) < \infty$.
    
    Moreover, suppose that $p_{\theta, F} \geq \frac{\epsilon^2}{2} m_i^2 \theta e^{- M_i (x + \theta y)}$. 
    We assume that $ 0 < m_1 < M_1 < e M_1 < m_2 < M_2 < e M_2$.    
    Assume that $ a > 0$ and $a = m_1$ or that $a = 0$ and $m_1 < \frac{\delta}{6 } M_2$.
    Furthermore, assume that $\| \theta - \theta_0 \| \leq \epsilon \theta_0$.
    Then 
    \[
        P_{\theta_0, F_0} \left( \frac{p_{\theta_0, F_0}}{ p_{\theta, F}} \right)^\delta \lesssim \frac{1}{\epsilon^\delta m_1^{2 \delta}}. 
    \]

\end{lem}
Note that when $F_0$ is supported on $[a,b]$ with $0 < a < b$, $F_0$ has arbitrarily high moments.
Therefore, we can pick $\delta$ as large as needed to make $m_1 < \frac{\delta}{6} M_2$.
For example, $\delta = 6$ is a valid choice.
\begin{proof}
 
    \textbf{Step 1: Bounding using Jensen's inequality}
    Since the map $x \mapsto x^{1 + \delta}$ is convex, we get that
    \[
        \iint \frac{ 
            \left( 
                \int p_{\theta_0}(x,y | z) \dd F_0(z)
            \right)^{1 + \delta} }
        {
            \int p_{\theta}(x,y | z) \dd F(z)^\delta
        }
        \leq
        \iint
        \frac{
            \int p_{\theta_0}(x, y | z)^{1 + \delta} \dd F_0(z)
        }{
            \int p_{\theta}(x,y | z) \dd F(z)^\delta
        }
    \]
    \textbf{Step 2: Constructing suitable upper bounds}
    We use the bound from the assumption, this gives:
    
    \[
        \iint 
        \frac{ 
            \int z^{2 + 2 \delta} \theta_0^{1 + \delta} e^{-(1 + \delta)z(x + \theta_0 y)} \dd F_0(z) 
        }{
            \frac{\epsilon^{2\delta}}{2^\delta} m_i^{2 \delta} \theta^\delta e^{-\delta M_i(x + \theta y)}
         } \dd x \dd y.
    \]
    By Fubini, we find that this is equal to
    \[
        \int \frac{
            z^{2 + 2 \delta} \theta_0^{1 + \delta}
        }{
            \frac{\epsilon^{2\delta}{2^\delta}} m_i^{2 \delta} \theta^{\delta}
        }
        \iint e^{( \delta M_i - (1 + \delta)z) x + (\delta \theta M_i - (1 + \delta) \theta_0 z) y} \dd x \dd y \dd F_0(z).
    \]
    The inner integral is equal to
    \[
        \frac{1}{(1 + \delta) z - \delta M_i} \frac{1}{ (1 + \delta) \theta_0 z - \delta \theta M_i}.
    \]
    Plugging this back gives us
    \[
        \frac{
            \theta_0^{1 + \delta}
        }{
            \epsilon^{2\delta} 2^\delta m_i^{2 \delta} \theta^{\delta}
        }
        \int z^{2 + 2 \delta} \frac{1}{(1 + \delta) z - \delta M_i} \frac{1}{ (1 + \delta) \theta_0 z - \delta \theta M_i}\dd F_0(z).
    \]

    \textbf{Step 3: Removing lower tail of $F_0$}
   
    So that leaves us in the case that $a = 0$ and $\underline{z} < \frac{\delta}{9} \overline{z}$. 
    We will use $z' = z_k$ for the $k$ such that $z_k \geq z_i$ for all $i$, and set $w' = w_k$. 
    That means we want to bound
    \[
        \frac{
            \theta_0^{1 + \delta}
        }{
            \epsilon^{2\delta} 2^\delta m_2^{2 \delta} \theta^{\delta}
        }
        \int_{z < m_1} z^{2 + 2 \delta} \frac{1}{(1 + \delta) z - \delta M_2} \frac{1}{ (1 + \delta) \theta_0 z - \delta \theta M_2}\dd F_0(z).    
    \]

    Per assumption, $\delta M_2 > 6 m_1$.
    Thus, on the interval $z \in [0, m_1]$, we can bound $\frac{2 z^{2 + 2\delta} \theta_0^{1 + \delta}}{\theta^\delta \epsilon^{2\delta}}$ by $\frac{2 m_1^{2 + 2\delta} \theta_0^{1 + \delta}}{\theta^\delta \epsilon^{2\delta} m_2^{2\delta}}$.
    We can similarly bound the second term $ \frac{1}{\delta M_2 - (1 + \delta) z }$ by $ \frac{1}{6 m_1 - (1 + \delta) m_1} = \frac{1}{(5 - \delta) m_1}$ when $0 \leq z < m_1$.
    For the third term, observe that $\frac{\| \theta - \theta_0 \|}{\theta_0} < \epsilon < \frac{1}{2}$.
    Hence, $\frac{\theta}{\theta_0} \geq \frac{1}{2}$.
    Combined with $\delta M_2 > 6 m_1$, this leads to the bound
    \[
        \frac{1}{ \theta \delta M_2 - (1 + \delta) \theta_0 z} < \frac{1}{3 \theta_0 m_1 - (1 + \delta) \theta_0 z} < \frac{1}{\theta_0 m_1}.
    \]
    Combining all the terms and using that $\frac{\theta_0}{\theta} < 2 \epsilon$ yields
    \[
        \iint \frac{
            \int_{z < m_1} p_{\theta_0}(x,y|z)^{1+ \delta} \dd F_0(z)
        }{
            \int p_\theta(x,y|z) \dd F(z) ^\delta
        }
        \dd x \dd y \lesssim \frac{1}{\epsilon^{ \delta}}.
    \]
    \textbf{Step 4: Removing the upper tail}
    We now bound the integral over the $z > m_1$ region. 
    Note that since $M_1 < e m_1$, it follows that $\delta M_1 < \delta e m_1 < \delta 3 z$.
    We have taken $\delta < \frac{1}{6}$, thus $\delta M_1 < \frac{1}{2}z$.
    This shows that 
    \[
        \frac{1}{(1 + \delta)z - \delta M_1} < \frac{1}{ (\frac{1}{2} + \delta) z}.
    \]
    We use that $\frac{\theta}{\theta_0} < 1 + \epsilon$.
    Hence, $\delta \theta M_1 < (1+ \epsilon)^2 \delta \theta_0 m_1$.
    So similarly as before, we get $\delta \theta M_1 <  (1 + \epsilon)^2 \delta \theta_0 m_1 < (1 + \epsilon)^2 \delta \theta_0 z$. 
    Plugging in $\epsilon < \frac{1}{2}$ and $\delta < \frac{1}{3}$ gives $\delta \theta M_1 < \frac{3}{4} \theta z$.
    This leads to the bound
    \[
        \frac{1}{ (1 + \delta) \theta_0 z - \delta \theta M_1}< \frac{1}{(\frac{1}{4} + \delta) \theta_0 z}.
    \]
    Using these bounds in our integral, we get
    \[
        \int_{z > m_1} \frac{
            z^{2 + 2 \delta} \theta_0^{1 + \delta}
        }{
            \frac{\epsilon^{2\delta}}{2^\delta} m_1^{2 \delta} \theta^{\delta}
        }
        \frac{1}{ (\frac{1}{2} + \delta) z}
         \frac{1}{(\frac{1}{4} + \delta) \theta_0 z}
         \dd F_0(z)
        .
    \]
    Simplifying, $\frac{\theta_0}{\theta} < 2\epsilon$ yields the following upper bound
    \[
        16\frac{2^\delta}{\epsilon^{ \delta} m_1^{2 \delta}} \int z^\delta \dd F_0(z) \lesssim \frac{1}{\epsilon^\delta m_1^{2 \delta}}.
    \]
    Combining the two bounds, we find that
    \[
        P_{\theta_0, F_0} \left( \frac{p_{\theta_0, F_0}}{ p_{\theta, F}} \right)^\delta \lesssim \frac{1}{\epsilon^\delta m_1^{2 \delta}}. 
    \]

\end{proof}

%% file: Frailty/PriorMass.tex
\subsubsection{Prior mass bound}

\begin{lem}
    Suppose the prior $\pi$ on $\theta \in V$ where $V$ is an open subset of $\RR^d$ admits a density $\pi(\theta)$ with respect to the Lebesgue measure that is bounded from below by $c > 0$ in a neighbourhood of $\theta_0$. 
    Then there exists a constant $C > 0$ such that 
    \[
        \pi( \theta: \| \theta - \theta_0 \| < \epsilon) > C \epsilon^d.
    \]
\end{lem}

\begin{proof}
    Volume of a ball in $\RR^d$ times $c$. 
\end{proof}

\begin{lem}\label{lem:FrailtyPriorMassBoundDP}
    Let $F_0$ be a probability distribution on $[0, \infty)$ and $\theta_0 \in \RR_{\geq 0}$. 
    Suppose that there exists a $\gamma > 0$ such that $\int (z^{-\gamma} + z^\gamma) \dd F_0(z) < \infty$.

    Let $\pi$ be a probability distribution which admits a density  bounded from below in a neighbourhood around $\theta_0$.
    Moreover, let $G_t$ be distribution on $[0,\infty)$ which admits a density $g_t$ such that $g_t (z) \geq c (z + 1)^{-\alpha}$.
    In addition, let $\Pi_t = \textrm{DP} (M G_t) \otimes \pi$. 
    Finally, let $ (F, \theta) \sim \Pi$ and consider $ (X_i, Y_i) | \theta, F \overset{i.i.d.} \sim P_{\theta, F}$.
    Then there exists constants $c>0 $ such that for all $\|t \| < \theta_0$
    \[
      \Pi_t\left( \textrm{KL}^2 + V_2 \left( p_{\theta_0, F_0}; p_{\theta, F} \right) < 8 \epsilon \right) \geq e^{c \log( \frac{1}{\epsilon})^3}.
    \]
\end{lem}

\begin{proof}
    Let $\theta_0, F_0$ be as given.
    \textbf{Constructing a suitable approximation}
    For all $\theta$ with $\| \theta - \theta_0 \| < \epsilon \theta_0$, we find
    \[
        \| p_{\theta_0, F_0} - p_{\theta, F_0} \| < \epsilon.
    \]
    Next, we are going to construct a suitable $F_\epsilon$.
    By \Cref{lem:FrailtycompactApprox} we can find a probability distribution $F^*$ supported on $[\underline{m}_\epsilon, \overline{m}_\epsilon] =[ \frac{1}{\EE[Z^{-\gamma}]^{1/\gamma}} \epsilon^{-1/\gamma},  \EE[Z^\gamma]^{1/\gamma} \epsilon^{1/\gamma}]$ such that
    \[
        \| p_{\theta_0, F_0} - p_{\theta_0, F^*} \| < 4\epsilon.
    \]
    Then we can find $F_\epsilon$ using \Cref{lem:FrailtyfiniteApproxLog}, this gives a distribution with fewer than $N_\epsilon \sim C \log\left(\EE[Z^\gamma]^{1/\gamma} \epsilon^{1/\gamma}\EE[Z^{-\gamma}]^{1/\gamma} \epsilon^{1/\gamma}\right) \log \left(\frac{1}{\epsilon} \right) \sim C' \log (\frac{1}{\epsilon})^2$ support points such that
    \[
        \| p_{\theta_0, F^*} - p_{\theta_0, F_\epsilon} \|_1 < \epsilon.
    \]
    Without loss of generality, we can assume these support points to be $\epsilon$ separated. 
    Let $A_i$ be a partition of $[0, \infty)$ such that $\frac{\textrm{Diam} A_i}{\underline{m}_\epsilon} = \epsilon$ for $1 \leq i \leq N_\epsilon$.
    That is, $\textrm{Diam} (A_i) = \EE[Z^{-\gamma}]^{1/\gamma} \epsilon^{1 + 1/\gamma}$.

    \textbf{Step 2: Constructing a suitable collection of $\theta, F$}
    Define 
    \[
        \Theta_{\epsilon} \times \mathcal{F}_\epsilon := \{ \theta, F : \| \theta - \theta_0 \| < \epsilon \theta_0, \sum_{i = 1}^{N_\epsilon} | F(A_i) - F_\epsilon(A_i) | < \epsilon, \min_i F(A_i) > \frac{\epsilon^2}{2} \}
    \]
    We will show that the priors assign enough prior mass and that this satisfies our KL conditions.

    \textbf{Step 3: Prior mass}
    
    Let $F \sim \textrm{DP}\left(M \G \right)$.

    We can bound $G_t (A_i)$ by our assumption on $G_t$. 
    First by \Cref{lem:perturbationinleastfavdirectiondensitybound}
    \[
        g_t(z) \gtrsim \left(\frac{1}{1 + z} \right)^\alpha
    \]
    for all $t$ with $\|t \| < \theta_0$.
    Hence, 
    \[
        G_t(A_i) \gtrsim \frac{\EE[Z^{-\gamma}]^{1/\gamma} \epsilon^{1 + 1/\gamma} }{ \left( 1 + \EE[Z^\gamma]^{1/\gamma} \epsilon^{1/\gamma} \right)^\alpha} \gtrsim \epsilon^{1 + 1/\gamma}.
    \]
Hence, we can apply~\cite[Lemma G.13]{ghosalFundamentalsNonparametricBayesian2017} which gives that there exists a $c > 0$ such that for all centre measures $G_t$
    \[
        \textrm{DP} \left(\sum_{i = 1}^{N_n} | F(A_i) - F_{\epsilon}(A_i) | < \epsilon, \min (F (A_i)) > \frac{\epsilon^2}{2} \right) \geq e^{- c N_\epsilon \log {\epsilon}}.
    \]
    Moreover, $\pi_t\left(\theta:  \| \theta - \theta_0 \| < \epsilon \theta \right) \geq c' \epsilon$ for all sufficiently small $t, \epsilon$.
    Hence, there exists a $\tilde{c} > 0$ such that for all sufficiently small $\epsilon, t$, we have
    \[
        \Pi_t\left( \Theta_\epsilon \times \mathcal{F}_\epsilon \right) \geq e^{-\tilde{c} \log( \epsilon)^3}.
    \]
    
    \textbf{Step 4: Controlling the KL divergence}
    We will now bound the $KL$ divergence for these partitions.
    Without loss of generality, assume $A_i$ are partitions in increasing order.
    Denote $F_i = F (\cdot \cap A_i)$.
    Then $p_{\theta, F_i} > \frac{\epsilon^2}{2} \min(A_i)^2 \theta e^{- \max(A_i) (x + \theta y)}$.
    Let $\| \theta - \theta_0 \|< \epsilon \theta_0$ and $\sum_{i = 1}^{N_n} |F (A_i) - F_0 (A_0) | < \epsilon$. 
    Then by the triangle inequality, \Cref{lem:L1boundintheta}, \Cref{lem:L1boundinF} and our earlier computations we find
    \[
        p_{\theta_0, F_0} - p_{\theta, F} \leq 8 \epsilon.
    \]
    Moreover, if we also demand $\min F (A_i) > \frac{\epsilon^2}{2}$ we can apply \Cref{lem:ExpectedLR}.
    This gives
    \[
        \log  P_{\theta_0, F_0} \left( \frac{p_{\theta_0, F_0}}{ p_{\theta, F}} \right)^\delta \lesssim \log(\frac{1}{\epsilon}).
    \]
    By~\cite[Lemma B.2.i]{ghosalFundamentalsNonparametricBayesian2017}, we find that there exists a constant $C > 0$ such that for these $\theta, F$, we have
    \[
        \textrm{V}_2 + \textrm{KL} (P_{\theta_0, F_0}; P_{\theta, F} ) < C \epsilon \log(\frac{1}{\epsilon}).
    \]
    Setting $\eta = C \epsilon \log (\frac{1}{\epsilon})$, we find
    \[
        \Pi( (\theta, F) : \textrm{V}_2 + \textrm{KL} (P_{\theta_0, F_0}; P_{\theta, F} ) < \eta ) \geq e^{-c \log( \eta)^3}.
    \]
\end{proof}

For a mixture of finite mixtures, we can formulate the same type of lemma.
\begin{lem}\label{lem:FrailtyPriorMassBoundMFM}
    Let $F_0$ be a probability distribution on $[0, \infty)$ and $\theta_0 \in \RR_{\geq 0}$. 
    Suppose that there exists a $\gamma > 0$ such that $\int (z^{-\gamma} + z^\gamma) \dd F_0(z) < \infty$.

    Let $\pi$ be a probability distribution which admits a density  bounded from below in a neighbourhood around $\theta_0$.
    Moreover, let $G_t$ be distribution on $[0,\infty)$ which admits a density $g_t$ such that $g_t (z) \geq c (z + 1)^{-\alpha}$.
    In addition, let $\Pi_t = \textrm{MFM} (G_t) \otimes \pi$. 
    Finally, let $ (F, \theta) \sim \Pi$ and consider $ (X_i, Y_i) | \theta, F \overset{i.i.d.} \sim P_{\theta, F}$.
    Then there exists constants $c>0 $ such that for all $\|t \| < \theta_0$
    \[
      \Pi_t( \textrm{KL}^2 + V_2 \left( p_{\theta_0, F_0}; p_{\theta, F} \right) < 8 \epsilon) \geq e^{c \log( \frac{1}{\epsilon})^3}.
    \]
\end{lem}

\begin{proof}
    The proof goes along the same lines as the proof of \Cref{lem:FrailtyPriorMassBoundDP}.
    We start along the same lines, create an approximation $F_N = \sum_{j = 1}^N w_j \delta_{z_j}$, with $N \sim \log \frac{1}{\epsilon}^2$ such that $\| p_{\theta_0, F_n} - p_{\theta_0, F_0} \| \leq \epsilon$.
    The details that change is that we replace \Cref{lem:L1boundinF} by \Cref{lem:L1boundinFMFM}, define $N = N'$, set
    \begin{align*}
        \Theta_\epsilon \times \mathcal{F}_\epsilon &:= \{ \| \theta - \theta_0 \| \leq \epsilon, \sum_{j= 1}^N \| W_j - w_j \| \leq 2 \epsilon\\
        & \qquad , \min W_j \geq \frac{\epsilon^2}{2}, \| z'_j - z_j \| \leq \epsilon \}
    \end{align*}
    Then the prior mass of such an event is a product of the separate events.
    \[
        \PP(K = N') = p (1 - p)^{N'}
    \]
    The other bounds follow along the same lines as \Cref{lem:FrailtyPriorMassBoundDP}.
    Combining the bounds leads to
    \[
        \PP( \Theta_\epsilon \times \mathcal{F}_\epsilon ) \geq Ce^{ k (\log ( 1- p) + c \log (\frac{1}{\epsilon})} \geq C'e^{ - c' \log( \frac{1}{\epsilon})^3}.
    \]
    The rest of the proof remains the same.
\end{proof}

We need to give a uniform lower bound for the densities of the centre measures of the shifted Dirichlet processes,
The following lemma provides such a lower bound.
\begin{lem}\label{lem:perturbationinleastfavdirectiondensitybound}
    If $G_0$ is a distribution with a density $g$ such that $g (z) \geq C (1 + z)^{-\alpha}$ with $C > 0$. 
    Denote $\phi^{-1}_t $ the perturbation in the least favourable direction, i.e. $\phi^t_\theta(z) = \frac{2 \theta z}{2 \theta + t}$. 
    Then $G_t = G_0 \circ \phi^{-1}_t$ are distributions with densities $g_t$ that also satisfy $g_t (z) \geq C' (1 +z)^\alpha$ for some $C'> 0$, for all $t$ with $\|t \| < \theta_0$.
\end{lem}
\begin{proof}
    The densities $g_t (z)$ are given by $z \mapsto \frac{ 2 \theta_0}{2 \theta_0 - t} g_0(\frac{2 \theta_0 z}{2 \theta_0 - t})$.
    Using the bound on $g_0(z)$ gives
    \begin{align*}
        g_t(z) & \geq \frac{2 \theta_0}{2 \theta_0 - t} C (1 + \frac{ 2 \theta_0 z}{2 \theta_0 - t}) \\
            \intertext{When $\| t \| \leq \theta_0$ we have $\frac{2}{3} < \frac{2 \theta_0}{2 \theta_0 - t} < 2$, hence}
            & \geq \frac{2 }{3} C (1 + 2 z)^{-\alpha} \\
            & \geq 2^{-\alpha} \frac{2}{3} C (1 + z)^{-\alpha}
    \end{align*}
    Which gives the bound.
\end{proof}

%% file: Frailty/TechnicalLemmas.tex
\subsubsection{Technical lemmas}

\begin{lem}\label{lem:TVExponentials}
    Let $X, Y$ be two exponential distributions with rates $\theta, \eta$. 
    Then
    \[
        d_{TV}(X, Y) \leq 2 \frac{ \| \theta - \eta \|}{ \eta}.
    \]
    
\end{lem}

\begin{proof}
    The total variation distance is equal to the $L_1$ distance of the corresponding densities. 
    Hence, it follows that
    \[
        d_{TV}(X, Y) = \| \theta e^{-\theta x} - \eta_1 e^{-\eta x} \|_1.
    \]
    By adding and subtract $\theta e^{- \eta x}$ and using the triangle inequality we find
    \begin{align*}
        d_{TV}(X,Y) & \leq \| \theta \left( e^{- \theta y} - e^{- \eta y}\right) \|_1 
         + \|\left( \theta - \eta \right) e^{- \eta y} \|_1 \\
        &= \theta \| \int_0^\infty e^{-\theta x} - e^{-\eta x} \dd x\|
        + \|\theta - \eta \| \int_0^\infty e^{- \eta y} dd y \\
        &= \theta \| \frac{1}{\theta} - \frac{1}{\eta} \|+ \frac{\| \theta - \eta \|}{\eta} \\
        &= 2 \frac{ \| \theta - \eta \|}{\eta}.
    \end{align*}
    
\end{proof}

%% file: ErrorsInVars/appendix.tex
\section{Errors in Variables}\label{sec:AppendixErrorsInVars}
In this appendix we have collected all the results for the errors in variables model.
\input{ErrorsInVars/LAN.tex}
\input{ErrorsInVars/ChangeOfMeasure.tex}
\input{ErrorsInVars/L1Approx.tex}

\input{ErrorsInVars/LRBounds.tex}

\input{ErrorsInVars/PriorMass.tex}

\input{ErrorsInVars/TechnicalLemmas.tex}

%% file: ErrorsInVars/LAN.tex
\subsection{LAN expansion for the errors in variables model}

Here we prove the LAN expansion using \Cref{Lem:LANByDonsker} and the arguments in~\cite{murphyLikelihoodInferenceErrorsVariables1996}.

We want to be able to apply \Cref{Lem:LANByDonsker}.
For this we need to compute the gradient and the Hessian of the log likelihood of the perturbed model.
We use \Cref{lem:ExpressionsForGradiantHessian} to compute the gradient and the Hessian, which then gives all the expressions we need.
The mixing kernel is given by
\[
  p_{\theta}((X, Y) \given Z) = \frac{1}{\sigma} \phi\left( \frac{X - Z}{\sigma} \right) \frac{1}{\tau} \phi\left( \frac{Y - \alpha - \beta Z}{\tau} \right).
\]
This means that 
\[
    h_{\theta}(t,Z) = -\log (\sigma )-\log (\tau )-\frac{(X-Z)^2}{2 \sigma ^2}-\frac{(-\beta +Y-\alpha  Z)^2}{2 \tau ^2}-\log (2)-\log (\pi ).
\]
The least favourable submodel is given by $\phi_{t, \theta} (Z) = \frac{\beta  \text{t1}}{\beta ^2+1}+Z \left(1-\frac{\beta  \text{t2}}{\beta ^2+1}\right)$.
We can compute the gradient and the Hessian.
The Gradient is given by
\[
    \nabla_t h_{\theta}(t,z) =
    \begin{pmatrix}
        & a_1 + a_2 Z + a_3 X + a_4 Y + a_5 YZ + a_6 Z^2\\
        & a_7 + a_8 Z + a_9Y a_{10} XZ + a_{11} YZ + a_{12} Z^2
    \end{pmatrix}
\]
with the constants given by
\begin{align*}
    a_1 &= \frac{\beta  \left(-\beta  \tau ^2 t_1-\left(\sigma ^2 (\alpha +2 t_1) \left(\beta ^3+\beta +\beta  t_1 (\alpha +t_1)+\beta ^2 t_2+t_2\right)\right)\right)}{\left(\beta ^2+1\right)^2 \sigma ^2 \tau ^2},\\
    a_2 &= -\frac{(\beta  (t_2-\beta )-1) \left(\beta  \tau ^2+\sigma ^2 \left(\beta ^3+\beta( 1 +(\alpha +t_1) (\alpha +3 t_1)+\beta t_2)+t_2\right)\right)}{\left(\beta ^2+1\right)^2 \sigma ^2 \tau ^2},\\
    a_3 &=  \frac{\beta  }{\left(\beta ^2+1\right) \sigma ^2},\\
    a_4 &= \frac{\beta  (\alpha +2 t_1)}{\left(\beta ^2+1\right) \tau ^2}, \\
    a_5 &= \frac{\left(1-\frac{\beta  t_2}{\beta ^2+1}\right)}{\tau ^2},\\
    a_6 &= \frac{(\alpha +t_1) \left(\beta ^2-\beta  t_2+1\right)^2}{\left(\beta ^2+1\right)^2 \tau ^2},\\
    a_7 &= -\frac{\beta +\frac{\beta  t_1 (\alpha +t_1)}{\beta ^2+1}+t_2}{\tau ^2},\\
    a_8 &= \frac{ \left(\beta ^2 \tau ^2 t_1+\sigma ^2 (\alpha +t_1) \left(\beta ^2 (t_1 (\alpha +t_1)-1)+2 \left(\beta ^3+\beta \right) t_2-1\right)\right)}{\left(\beta ^2+1\right)^2 \sigma ^2 \tau ^2},\\
    a_9 &= \frac{1}{\tau ^2},\\
    a_{10}&=-\frac{\beta (\alpha +t_1)}{\left(\beta ^2+1\right) \tau ^2},\\
    a_{11}&=-\frac{\beta}{\left(\beta ^2+1\right) \sigma ^2},\\
    a_{12}&=\frac{\beta \left(\beta ^2-\beta  t_2+1\right) \left(\tau ^2+\sigma ^2 (\alpha +t_1)^2\right)}{\left(\beta ^2+1\right)^2 \sigma ^2 \tau ^2}.
\end{align*}

The hessian is given by \[ H_t h_{\theta}(t,Z) =
\left(
\begin{array}{cc}
c_1 + c_2 Z + c_3 Y + c_4 Z^ 2 & c_5 + c_6 Z + c_7 YZ + c_8 Z^2\\
d_1 + d_2 Z + d_3 YZ + d_4 Z^2 & d_5 + d_6 Z + d_7 Z^2
\end{array}
\right)          
\]

with the constants $c$ given by
\begin{align*}
  c_1 &=\frac{t_2 \left(-\left(\sigma ^2 \left(2 \beta +t_2 \left(6 \alpha ^2+t_1^2+6 \alpha  t_1+2 t_2 (\beta +t_2)+2\right)\right)\right)-\tau ^2 t_2\right)}{\sigma ^2 \tau ^2 \left(t_2^2+1\right)^2}\\
  c_2 &= -\frac{2 t_2 (3 \alpha +2 t_1) \left(t_2^2-\beta  t_2+1\right)}{\tau ^2 \left(t_2^2+1\right)^2}\\
  c_3 &= \frac{2 t_2 }{\tau ^2 \left(t_2^2+1\right)} \\
  c_4 &= -\frac{\left(t_2^2-\beta  t_2+1\right)^2}{\tau ^2 \left(t_2^2+1\right)^2}\\
  c_5 &= -\frac{t_2 (2 \alpha +t_1)}{\tau ^2 \left(t_2^2+1\right)} \\
  c_6 &= \frac{ \left(\sigma ^2 (t_2 (2 \beta +t_2 ((\alpha +t_1) (3 \alpha +t_1)+2 \beta  t_2-1))-1)+\tau ^2 t_2^2\right)}{\sigma ^2 \tau ^2 \left(t_2^2+1\right)^2} \\
  c_7 &= -\frac{t_2}{\tau ^2 \left(t_2^2+1\right)} \\
  c_8 &= \frac{2 t_2 (\alpha +t_1) \left(t_2^2-\beta  t_2+1\right)}{\tau ^2 \left(t_2^2+1\right)^2}
\end{align*}
and $d$ given by
\begin{align*}
  d_1 &= -\frac{t_2 (2 \alpha +t_1)}{\tau ^2 \left(t_2^2+1\right)} \\
  d_2 &= \frac{ \left(\sigma ^2 (t_2 (2 \beta +t_2 ((\alpha +t_1) (3 \alpha +t_1)+2 \beta  t_2-1))-1)+\tau ^2 t_2^2\right)}{\sigma ^2 \tau ^2 \left(t_2^2+1\right)^2}\\
  d_3 &= -\frac{t_2 }{\tau ^2 \left(t_2^2+1\right)} \\
  d_4 &= \frac{2 t_2 (\alpha +t_1) \left(t_2^2-\beta  t_2+1\right)}{\tau ^2 \left(t_2^2+1\right)^2}\\
  d_5 &= -\frac{1}{\tau ^2}\\
  d_6 &= \frac{2 t_2 (\alpha +t_1)}{\tau ^2 \left(t_2^2+1\right)}\\
  d_7 &= -\frac{t_2^2 \left(\tau ^2+\sigma ^2 (\alpha +t_1)^2\right)}{\sigma ^2 \tau ^2 \left(t_2^2+1\right)^2} \\
\end{align*}

\begin{lem}\label{lem:LanErrorsInVars}
  Assume that $\sigma^2 > 0$, $\tau = c \sigma$ for some known $c > 0$, $\alpha, \beta \in \RR$ and $F_0(z^{7 +\delta}) < \infty$ for some $\delta > 0$, then 
    \Cref{EqRemainders} holds for $P_{\theta_0, F_0}$.
\end{lem}

\begin{proof}
    By consistency, we can restrict to small neighbourhoods $\theta_0 \in \Theta, F_0 \in \mathcal{F}$.
    Next, by~\cite[Lemma 7.1 till 7.3]{murphyLikelihoodInferenceErrorsVariables1996}, and consistency we see that
    
    \[
        \{ \dot\ell(t/\sqrt{n}; \theta, F) : \theta \in \Theta, F \in \mathcal{F} \}
    \]
    is a Donsker class.
    By~\cite[Lemma 7.1]{murphyLikelihoodInferenceErrorsVariables1996} we also show that the first derivative has a square integrable envelope and the second derivative has an integrable envelope. 
    Furthermore, in the notation of \Cref{lem:LimitsOfExpectationsLogLik}, $h_{\theta} (x \given z)$ is twice differentiable in $\theta$ and $z$, and $\phi_{t, \theta} (z)$ is twice differentiable in $\theta, t$. 
    Moreover, the gradient and the hessian remain bounded.
    Thus, we can apply \Cref{lem:LimitsOfExpectationsLogLik}.
    Together with the Donsker condition and the discussion after \Cref{Lem:LANByDonsker} we can verify all the assumptions of \Cref{Lem:LANByDonsker} and conclude that \Cref{EqRemainders} holds.
\end{proof}

%% file: ErrorsInVars/ChangeOfMeasure.tex
\subsection{Verifying the change of measure condition}
\begin{lem}\label{lem:ErrorsInVarsChangeOfMeasure}
  Under the conditions of \Cref{lem:BvMErrorsInVars} \eqref{EqChangeOfMeasure} holds. 
\end{lem}

\begin{proof}
  We want to verify the assumptions of \Cref{lem:ChangeOfMeasure}.
  The least favourable submodel is given by \Cref{rem:LeastFavSubmodelErrorsInVars}.
  For our work, we will use the perturbation 
  \[
    \phi_{t, \theta}(Z) = Z\left( 1 - \frac{t_2 \beta}{1 + \beta^2} \right) + \frac{t_1 \beta}{1 + \beta^2}.
  \]
  Thus, \Cref{ass:ChangeOfVariables} is verified.
  We can verify posterior consistency in the perturbed model.
  By \Cref{lem:ErrorsInVarsPriorMassEpsDP} and \Cref{lem:ErrorsInVarsPriorMassEpsMFM} each neighbourhood of the truth has positive mass under the DP prior and the MFM prior respectively.
  Hence, by \Cref{lem:posteriorconsistency} and \Cref{remark:ConsistencyPerturbedPosterior} and the consistency proof in~\cite{murphyLikelihoodInferenceErrorsVariables1996,vandervaartEfficientMaximumLikelihood1996} we satisfy \Cref{ass:Consistency}.

  Next we verify \Cref{ass:Smoothness}.

  First we can verify the conditions on $h$. We have assumed that $h$ continuously differentiable in a neighbourhood of $\theta_0$. 
  By applying \Cref{lem:SmoothnessInH} we verify the conditions on $h$.
  
  We want to verify the assumptions of \Cref{lem:SmoothnessInGSufficient}.
  The least favourable submodel satisfies the assumption.
  Moreover, we assumed the density of the centre measure satisfy the requirements of the lemma.
  Hence, by \Cref{lem:SmoothnessInGSufficient} gives us the conditions to apply \Cref{lem:SmoothnessInG}.
  This means that \Cref{ass:smoothnessinG} is satisfied.

  Finally, we need to verify \Cref{ass:SmallNumberClusters}. For both models we will verify this assumption using the remaining mass theorem.
  \textbf{Dirichlet process case}
  
  Pick $\epsilon_n = n^{-\alpha}$ for $a = \frac{1}{2}$.
  By \Cref{lem:ErrorsInVarsPriorMassEpsDP} we know that
  \begin{align*}
    \pi_t \otimes \textrm{DP}_t(B_{n,2} (p_{\theta, F}; p_{\theta_0, F_0}, \epsilon_n)) &\leq= \pi \otimes \textrm{DP}_t( \textrm{KL} + V_2 \left( p_{\theta, F}; p_{\theta_0, F_0} \right) \leq \epsilon_n^2 )\\
      &\geq e^{-C \epsilon_n^{-\frac{1}{( 1 + \xi)}} \log( \epsilon_n^{-1})} \\
      &= e^{ - C' n^{\frac{\alpha}{(1 + \xi)}} \log n}
  \end{align*}
  Let $\frac{1}{2(1 + \xi)} < \beta < \frac{1}{2}$, for example, $\beta = \frac{ \frac{1}{2 (1 + \xi)} + \frac{1}{2}}{2} = \frac{2 + \xi}{4(1 + \xi)}$.
  By \Cref{lem:tailbound_distinct_obs_DP} we find that
  \[
    \PP( K > n^\beta) \lesssim e^{- C n^\beta \log n}.
  \]
     
  Hence, we want to show that
  \[
    e^{C n^{ \frac{\alpha}{2( 1 + \xi)}} \log n} e^{- c \sqrt{n} \log n} e^{2 n^{1 - 2 \alpha}} = o(1)
  \]
  It suffices to show that $n^\beta \log(n)$ dominates both other terms.
  $n \epsilon_n^2 = n^{1 - 2 \alpha} = 1$ is much smaller than $n^{\beta}\log(n)$.

  Now we need to show that $n^\beta \log n$ also dominates $n^{\frac{\alpha}{ (1 + \xi)}} \log(n)$.
  This holds true as soon as $\frac{1}{2(1 + \xi)} < \beta$, which is true per our choice of $\beta$.
  Thus, by~\cite[Theorem 8.20]{ghosalFundamentalsNonparametricBayesian2017} we have verified \Cref{ass:SmallNumberClusters} for the Dirichlet process case.
  
  \textbf{Finite mixture case}
  
  Pick $\epsilon_n = n^{-\alpha}$ for $a = \frac{1}{2}$.
  By \Cref{lem:ErrorsInVarsPriorMassEpsMFM} we know that
  \begin{align*}
        \pi_t \otimes \textrm{DP}_t(B_{n,2} (p_{\theta, F}; p_{\theta_0, F_0}, \epsilon_n)) &\leq= \pi \otimes \textrm{DP}_t( \textrm{KL} + V_2 \left( p_{\theta, F}; p_{\theta_0, F_0} \right) \leq \epsilon_n^2 )\\
            &\geq e^{-C \epsilon_n^{-\frac{1}{( 1 + \xi)}} \log( \epsilon_n^{-1})} \\
            &= e^{ - C' n^{\frac{\alpha}{(1 + \xi)}} \log n}
    \end{align*}
    Let $\frac{1}{2(1 + \xi)} < \beta < \frac{1}{2}$, for example, $\beta = \frac{ \frac{1}{2 (1 + \xi)} + \frac{1}{2}}{2} = \frac{2 + \xi}{4(1 + \xi)}$.
    By \Cref{lem:tailbound_distinct_obs_DP} we find that
    \[
      \PP( K > n^\beta) \lesssim e^{- C n^\beta \log n}.
    \]
     
    Hence, we want to show that
    \[
        e^{C n^{ \frac{\alpha}{2( 1 + \xi)}} \log n} e^{- c \sqrt{n} \log n} e^{2 n^{1 - 2 \alpha}} = o(1)
    \]
    It suffices to show that $n^\beta \log(n)$ dominates both other terms.
    $n \epsilon_n^2 = n^{1 - 2 \alpha} = 1$ is much smaller than $n^{\beta}\log(n)$.

    Now we need to show that $n^\beta \log n$ also dominates $n^{\frac{\alpha}{ (1 + \xi)}} \log (n)$.
    This holds true as soon as $\frac{1}{2(1 + \xi)} < \beta$, which is true per our choice of $\beta$.
    Thus, by~\cite[Theorem 8.20]{ghosalFundamentalsNonparametricBayesian2017} we have verified \Cref{ass:SmallNumberClusters} for the Finite mixture case.

\end{proof}

%% file: ErrorsInVars/L1Approx.tex
\subsubsection{Approximation in total variation}
The density of the observations is given by
\[
    p_{\theta, \eta} (x,y) = \int \frac{1}{\sigma} \phi\left( \frac{x - z}{\sigma}\right) \frac{1}{\tau} \phi \left( \frac{y - \alpha - \beta z}{\tau}\right) \dd F(z)
\]
where $\theta = (\alpha, \beta)$ and $\eta = (\sigma, F)$.
We presume $\frac{\tau}{\sigma}$ to be known for now.
Without this, the model is not identifiable.
We can maybe generalise to general $\Sigma$, instead of assuming some diagonal matrix. 
In that case, we need to assume that the matrix $\Sigma$ is known up to some scalar factor, again for identifiability reasons.

First we show that we can approximate well if we vary the parametric part of the parameter.
\begin{lem}\label{lem:ErrorsInVarsParametricApprox}
    Suppose that $F_0(|z|) < \infty$. Suppose that
    \begin{align*}
        \frac{3 \| \sigma^2 - \sigma_0^2 \|}{2\sigma_0^2} & \leq \epsilon \\
        \frac{3 \| \tau^2 - \tau_0^2 \|}{2 \tau_0^2 } & \leq \epsilon \\
        \frac{ \| \alpha - \alpha_0 \|}{ \tau_0} & \leq \epsilon \\
        \frac{ \| \beta - \beta_0 \|}{ \tau_0 } \int |z| \dd F_0(z) & \leq \epsilon
    \end{align*}
    then
    \[
        \| p_{\theta, F_0} - p_{\theta_0, F_0} \|_1 \leq 4 \epsilon.
    \]
\end{lem}

\begin{proof}
    By using the definition of $p_{\theta, F}$, Jensen's inequality,~\cite[Lemma B.4]{ghosalFundamentalsNonparametricBayesian2017}, and \Cref{lem:GaussianTVHellingerBounders} we find the claimed result of the lemma:
    \begin{align*}
        \| p_{\theta, F_0} - p_{\theta_0, F_0} \|_1 & = \| \int p_{\theta}((x,y) \given z) - p_{\theta_0}( (x,y) \given z) \dd F_0(z)\|_1  \\
        & \leq \int \| p_{\theta}( (x,y) \given z) - p_{\theta_0}( (x,y) \given z) \|_1 \dd F_0(z)\\
        & \leq \int \| \phi_\sigma(x - z) - \phi_{\sigma_0}( x - z) \|_1 \\
        & \qquad  + \| \phi_\tau(y - \alpha - \beta z) - \phi_{\tau_0}( y - \alpha_0 - \beta_0 z) \|_1 \dd F_0(z) \\
        & \leq \frac{3 \| \sigma^2 - \sigma_0^2 \|}{2\sigma_0^2} + \frac{3 \| \tau^2 - \tau_0^2 \|}{2 \tau_0^2 } \\
        & \qquad + \frac{ \| \alpha - \alpha_0 \|}{ \tau_0} +  \frac{ \| \beta - \beta_0 \|}{ \tau_0 } \int |z| \dd F_0(z) \\
        & \leq 4 \epsilon 
    \end{align*}
\end{proof}

Next we show that we approximate every mixture by looking at a set $A$ such that the complement of $A$ has small probability under the mixing measure.
\begin{lem}\label{lem:ErrorsInVarsRestrictingAndNormalisingErrorControl}
    Denote $p_{\theta, F} = \int p_\theta (x \given z) \dd F (Z)$, where $p_\theta(x|z)$ is a family of probability densities for every $\theta, Z$.
    Let $\epsilon > 0$ and pick $A$ a set such that $F (A^c) < \epsilon$.
    Denote $p_{\theta, F_{|A}} = \frac{1}{F (A)} \int_A p_\theta (x \given z) \dd F (z)$.
    Then
    \[
        \| p_{\theta, F} - p_{\theta, F_{|A}} \|_1 \leq 2 \epsilon.
    \]
\end{lem}

\begin{proof}
    By writing out the definitions and Jensen's inequality
    \begin{align*}
        \| p_{\theta, F} - p_{\theta, F_{|A}} \|_1 &= \| \int \left( 1 - \frac{1}{F(A)} \II_A \right) p_{\theta}(x \given z) \dd F(z) \|_1 \\
        & \leq \| \int_{A^c} p_{\theta}(x \given z) \dd F(z) \|_1 + \| \int_A \left(1 - \frac{1}{F(A)} \right) p_\theta( x \given z) \dd F(z) \|_1 \\
        &\leq \int_{A^c} \| p_\theta(x \given z) \|_1 \dd F(z) + \left( \frac{1}{F(A)} - 1 \right) \int_A \| p_\theta( x \given z) \|_1 \dd F(z) \\
        &= F(A^c) + \left( \frac{1}{F(A)} - 1\right) F(A) = 2 F(A^c) \leq 2 \epsilon
    \end{align*}
    since $\| p_\theta(x \given z) \|_1 = 1$.
\end{proof}

Finally, we show that every such mixture can be well approximated using a mixing distribution which has only a small number of support point.
\begin{lem}\label{lem:ErrorsInVarsFiniteApproximation}
    Let $\epsilon > 0$, suppose that
    \begin{align*}
        \frac{3 \| \sigma^2 - \sigma_0^2 \|}{2\sigma_0^2} & \leq 1 \\
        \frac{3 \| \tau^2 - \tau_0^2 \|}{2 \tau_0^2 } & \leq 1\\
        \frac{ \| \alpha - \alpha_0 \|}{ \tau_0} & \leq 1\\
        \frac{ \| \beta - \beta_0 \|}{ \tau_0 } \int |z| \dd F_0(z) & \leq 1
    \end{align*}
    Then, there exists a $C> 0$ such that for every probability measure $F$ on $[-A, A]$ there exists a discrete distribution $F^*$ with at most $C A \log \epsilon^{-1}$ support points, where $C$ is a universal constant, such that
    \[
        \| p_{\theta, F} - p_{\theta, F^*} \|_1 < \epsilon.
    \]
\end{lem}

\begin{proof}
    First we reduce the problem to the case where $F$ is supported on an interval of length of at most $1$.
    By partitioning $[-A, A]$ into intervals $I_i$ of length $1$, suppose we have the approximations $F^*_i$ on interval $I_i$ with at most $D \log(\epsilon^{-1})$ support points. 
    Then $F^* = F (I_i) p_{\theta, F^*_i}$ has at most $2 (A+2) D \log (\epsilon^{-1})$ support points and satisfies
    \[
        \| p_{\theta, F^*} - p_{\theta, F} \| \leq \epsilon.
    \]
    Hence, it suffices to show that the approximation holds for $F$ supported on an interval of size $1$, say $[\xi, \xi+1]$.

    A Taylor expansion of the exponential function gives,
    \[
        e^w = \sum_{l = 0}^{k-1} \frac{w^l}{l!} + R_k(w),
    \]
    with $|R_k (w)| \leq \frac{|w|^k}{k!}$ when $w < 0$.
    Furthermore, let us define $g (x,y,z) =  - \frac{1}{2} \left(\left(\frac{x - z}{\sigma} \right)^2 + \left(\frac{y - \alpha - \beta z}{\tau} \right)^2 \right)$.
    Note that
    \[
        p_{\theta}( (x,y) \given z) = \frac{1}{2 \pi \sigma \tau} e^{g(x,y,z)}.
    \]
    Pick $F^*$ a probability measure on $[\xi, \xi+1]$ such that
    \[
        \int z^l \dd F(z) = \int z^l \dd F^*(z).
    \]
    This can always be done, since by~\cite[Lemma L.1]{ghosalFundamentalsNonparametricBayesian2017} we can match $K$ moments using $K$ support points.
    If we apply the Taylor expansion to $p_{\theta} ((x,y) \given z)$ and integrate this with respect to $F - F^*$, this only leaves the integral over 
    \[
        R_k\left(g (x,y,z) \right).
    \]
    In the view of $k!> k^k e^{-k}$, we find that $R_k (w) \leq \left(\frac{e w}{k} \right)^k$.
    Let $T \geq \max (2, 2\beta)$.
    When $|x - \xi | \leq T$, we find $|x - z| \leq | x - \xi - (z - \xi)| \leq T + 1 \leq 2T$. 
    Similarly, when $|y - \alpha - \beta \xi | \leq T$, we find $y - \alpha - \beta z| \leq | y - \alpha - \beta \xi - \beta (z - \xi) | \leq (1 + \beta) T$.
    Therefore, when $|x - \xi | \leq T, |y - \alpha - \beta \xi| \leq T$, we can bound
    \[
        g(x,y,z) \leq - \frac{1}{2} \left( \frac{4} {\sigma^2} + \frac{ (1 + \beta)^2}{\tau^2} \right) T^2.
    \]
    Using the bound on $R_k$ we find that then also
    \[
        R_k (g (x,y,z)) \leq \left( \frac{2 e \left( \frac{4}{\sigma^2} + \frac{ (1 + \beta)^2}{\tau^2} \right)^2 }{k} \right)^k.
    \]
    Conversely, when $|x - \xi | \geq T$, we know that $|x - z| = |x - \xi - (z - \xi)| \geq \frac{ |x - \xi|}{2}$. This implies $\phi_\sigma\left(\frac{x - z}{\sigma} \right) \leq e^{-\frac{x^2}{8 \sigma^2}}$.
    We can use this to bound
    \begin{align*}
        \int_{|x- \xi| \geq T} &\int p_{\theta, F}(x,y) \dd x \dd y\\
        &= \int_{|x - \xi| \geq T}\int \int \phi_\sigma\left( x - z\right) \phi_\tau\left( y - \alpha - \beta z \right) \dd F(z) \dd y \dd x\\
        &= \int_{|x - \xi| > T} \int \phi_\sigma \left( x - z \right) \dd x \int \phi_\tau\left( y - \alpha - \beta z \right) \dd y \dd F(z) \dd x\\
        &=\int_{|x - \xi| > T} \phi_\sigma \left( x - z \right) \dd F(z) \dd x\\
        & \leq \int_{|x - \xi| > T} \frac{1}{\sqrt{ 2 \pi} \sigma} e^{- \frac{ \left( x- \xi\right)^2}{8 \sigma^2}} \\
        &\leq e^{- \frac{T^2}{8 \sigma^2}}.
    \end{align*}
    Similarly, when $|y - \alpha - \beta \xi| \geq T$, then $|y - \alpha - \beta z| \geq \frac{| y - \alpha - \beta \xi}{2}$.
    Hence, by the same argument as before
    \[
        \int_{| y - \alpha - \beta \xi| \geq T} \leq e^{- \frac{T^2}{8 \tau^2}}.
    \]
    Combining these bounds yield the upper bound
    \[
        \| p_{\theta, F} - p_{\theta, F^*} \|_1 \leq e^{- \frac{T^2}{8 \sigma^2}} e^{ - \frac{T^2}{8\tau^2}}+ T^2\left( \frac{ 2 e \left( \frac{4}{\sigma^2} + \frac{(1 + \beta)^2}{\tau^2} \right) T^2}{k} \right)^k.
    \]
    If we choose $T = 4 \max(\sigma, \tau) \sqrt{\log e^{-1}}$ we reduce the first two terms to be less than $\epsilon$.
    Next, we pick $k \geq 4 e \left(\frac{4}{\sigma^2} + \frac{(1 + \beta)^2}{\tau^2} \right) T^2$ so that
    \[
        T^2 \left( \frac{ 2 e \left( \frac{4}{\sigma^2} + \frac{(1 + \beta)^2}{\tau^2} \right) T^2}{k} \right)^k \leq T^2 2^{-k}.
    \]
    To make this eventually smaller than $\epsilon$ we must pick $k \geq 2\log_2 (T) \geq 4 \log \log(\epsilon)$.
    The choice $k \sim 64 e \max(\sigma, \tau)^2 \left(4 + (1 + \beta)^2 \right)\log \epsilon^{-1}$ suffices.
    Thus, there exists an $F^*$ with at most $128 (A + 2) \max \left(4 + (1 + \beta)^2 \right) \log \epsilon^{-1}$ support points such that
    \[
         \| p_{\theta, F} - p_{\theta, F^*} \|_1 \leq 3 \epsilon.
    \]
    By using the bounds on $\sigma, \tau, \beta$ and changing $\epsilon$, we can find the universal constant $C$.
\end{proof}

\begin{lem}\label{lem:ErrorsInVarsPartitionsCloseInTV}
    Let $\cup_{i = 0}^N A_i$ be a partition, $F_N = \sum_{j= 1}^N w_j \delta_{z_j}$ a probability measure with $z_j \in A_j$ for $j = 1,\dots,N$. Then
    \[
        \| p_{\theta, F} - P_{\theta, F_N} \|_1 \leq 2 \max_{1 \leq j \leq N} \frac{(1 + \frac{\beta}{\gamma})\text{diam } A_j}{\sigma} + \sum_{j = 1}^N | F(A_j) - w_j| + F(A_0).
    \]
\end{lem}

\begin{proof}
    \begin{align*}
        \int p_{\theta}(x \given z) \dd (F - F_N)(z) &= \sum_{j = 1}^N \int_{A_j} p_{\theta}(x \given z) - p_{\theta}(x \given z_j) \dd F(z) \\
        &+ \sum_{j = 1}^N p_\theta( x \given z_j) ( F(A_j) - w_j) \\
        &= \int_{A_0} p_\theta(x \given z) \dd F(z) 
    \end{align*}
    We begin with the last two terms. By Fubini, positivity and $p_\theta(x \given z)$ being a probability density, it follows that
    \[
        \| \int_{A_0} p_\theta(x \given z) \dd F(z) \|_1 \leq \int_{A_0} \| p_\theta(x \given z) \|_1 \dd F(z) \leq F(A_0).
    \]
    Secondly, it also follows that
    \[
        \| \sum_{j = 1}^N p_\theta(x \given z) (F(A_j) - w_j) \|_1 \leq \sum_{j = 1}^n | F(A_j) - w_j|.
    \]
    So we only need to conclude that
    \[
        \| \sum_{i = 1}^N \int A_i p_{\theta}(x \given z) - p_{\theta}(x \given z_i) \dd F(z)\|_1 \leq 2\max_{1 \leq j \leq N} \frac{ \text{diam } A_j}{\sigma}.
    \]
    First, we use the triangle inequality to get
    \[
        \| \sum_{j = 1}^N \int_{A_j} p_\theta(x \given z) - p_\theta( x \given z_j) \dd F(z) \|_1 \leq \sum_{j = 1}^N \int_{A_j} \| p_\theta(x \given z) - p_\theta(x \given z_j) \|_1 \dd F(z).
    \]
    By applying~\cite[Lemma B.8.i]{ghosalFundamentalsNonparametricBayesian2017}, the definition of $p_\theta(x \given z)$, and the fact that the distance between $\frac{1}{\sigma} \phi(\frac{x - \mu}{\sigma})$ and $\frac{1}{\sigma} \phi(\frac{x}{\sigma})$ is bounded by $\frac{|\mu|}{\sigma}$, we get that
    \[
        \|p_\theta(x \given z) - p_\theta(x \given z_j) \|_1 \leq \frac{z - z_j}{\sigma} + \frac{ \beta | z - z_j|}{\tau} = \frac{ 1 + \frac{\beta}{\gamma}}{\sigma} | z - z_j |.
    \]
    Hence,
    \[
        \sum_{j = 1}^N \int_{A_j} \| p_\theta(x \given z) - p_\theta(x \given z_j) \|_1 \dd F(z) \leq 2 \frac{1 + \frac{\beta}{\gamma}}{\sigma} \max_{1 \leq j \leq N} \text{diam } A_j.
    \]
\end{proof}

\begin{lem}\label{lem:ErrorsInVarsPairsCloseInTV}
  Let $F_N = \sum_{j= 1}^N w_j \delta_{z_j}$ and $F'_{N'} = \sum_{j = 1}^{N'} w_j' \delta_{z'_j}$ be probability measures. Denote $K = \min(N, N')$. Then
    \[
      \| p_{\theta, F_N} - P_{\theta, F'_{N'}} \|_1 \leq 2 \max_{1 \leq j \leq K} (1 + \frac{\beta}{\gamma})|z_j - z'_j| + \sum_{j = 1}^K | w_j - w'_j| + \sum_{j = K + 1}^N w_j + \sum_{j= K + 1}^{N'} w'_j.
    \]
\end{lem}

\begin{proof}
  
    \begin{align*}
      \int p_{\theta_0}(x \given z) \dd (F_N - F'_{N'}) &= \sum_{j = 1}^K w_j(p_{\theta_0}(x \given z_j) - p_{\theta_0}(x \given z'_j)) \\
                                                        &+ \sum_{j = 1}^K (w_j - w'_j)p_{\theta_0}(x \given z'_j) \\
                                                        &+ \sum_{j = K + 1}^N w_j p_{\theta_0}(x \given z_j) + \sum_{j= K + 1}^{N'} p_{\theta_0}(x \given z'_j)
    \end{align*}
    We again begin with the last three terms. By Fubini, positivity and $p_{\theta_0} (x \given z)$ being a probability density it follows that
    \[
        \| \sum_{j = K + 1}^N w_j p_{\theta_0}(x \given z_j) \|_1 = \sum_{j = K+ 1}^N w_j
    \]
    and similarly
    \[
      \| \sum_{j = K + 1}^{N'} w'_j p_{\theta_0}(x \given z'_j) \|_1 = \sum_{j = K + 1}^{N'} w'_j.
    \]
    It also follows that
    \[
      \| \sum_{j = 1}^K (w_j - w'_j) p_{\theta_0}(x \given z'_j) \|_1 = \sum_{j = 1}^K | w_j - w'_j|.
    \]
    So it suffices to show that 
    \[
      \| \sum_{j = 1}^K w_j(p_{\theta_0}(x \given z_j) - p_{\theta_0}(x \given z'_j)) \|_1 \leq 2( 1 = \frac{\beta}{\gamma}) \max_{ 1 \leq j \leq K} | z_j - z'_j | 
    \]
    By the triangle inequality
    \[
      \| \sum_{j = 1}^K w_j(p_{\theta_0}(x \given z_j) - p_{\theta_0}(x \given z'_j)) \|_1 \leq  \sum_{j = 1}^K w_j \| p_{\theta_0}(x \given z_j) - p_{\theta_0}(x \given z'_j) \|_1 
    \]
    By applying~\cite[Lemma B.8.i]{ghosalFundamentalsNonparametricBayesian2017}, the definition of $p_\theta(x \given z)$, and the fact that the distance between $\frac{1}{\sigma} \phi(\frac{x - \mu}{\sigma})$ and $\frac{1}{\sigma} \phi(\frac{x}{\sigma})$ is bounded by $\frac{|\mu|}{\sigma}$, we get that
    \[
        \|p_\theta(x \given z_j) - p_\theta(x \given z'_j) \|_1 \leq \frac{z_j - z'_j}{\sigma} + \frac{ \beta | z_j - z'_j|}{\tau} = \frac{ 1 + \frac{\beta}{\gamma}}{\sigma} | z_j - z'_j |.
    \]
    Hence,
    \[
      \sum_{j = 1}^K w_j \| p_{\theta_0}(x \given z_j) - p_{\theta_0}(x \given z'_j) \|_1 \leq \sum_{j = 1}^K w_j \frac{ 1 + \frac{\beta}{\gamma}}{\sigma} | z_j - z'_j |.
    \]
    Since $\sum_{j= 1}^K w_j \leq 1$, it follows that
    \[
      \sum_{j = 1}^K w_j \frac{ 1 + \frac{\beta}{\gamma}}{\sigma} | z_j - z'_j | \leq \frac{ 1 + \frac{\beta}{\gamma}}{\sigma}  \max_{1 \leq j \leq K} | z_j - z'_j |.
    \]
\end{proof}

%% file: ErrorsInVars/LRBounds.tex
\subsubsection{Likelihood ratio bounds}
\begin{lem}\label{lem:ErrorsInVarsCloseInKL}
    Let $\epsilon > 0$, $N_\epsilon > 0$. 
    Let $U_i$ be a partition of $\RR$.
    Suppose that $F$ is a probability measure such that $F (U_i) \geq \epsilon^2$.
    Furthermore, assume that $d_H (p_{\theta, F}, p_{\theta_0, F_0}) \leq \epsilon$.
    Then
    \[
        V_2 + \text{KL} (p_{\theta, F}; p_{\theta_0, F_0}) \leq \epsilon (\log \epsilon)^2.
    \]
\end{lem}

\begin{proof}
    Recall that
    \[
        p_{\theta, F} (x,y) = \int \phi_\sigma \left(\frac{ x - z}{\sigma} \right)^2 \phi_\tau \left(\frac{ y - \alpha - \beta z}{ \tau} \right)^2 \dd F(z)
    \]
    Denote $i (x)$ to be the $i$ such that $z \in U_{i (x)}$ implies $|\frac{ x - z}{\sigma}| \leq 2\epsilon$.
    Similarly, denote $j (y)$ to be the $j$ such that $z \in U_{j (y)}$ implies $|\frac{ y - \alpha - \beta}{\tau}| \leq 2 \epsilon$.
    Then we can construct 3 bounds for $\phi_\sigma (x - z)$ and $\phi_\tau (y - \alpha - \beta z)$.

    \begin{align*}
        \phi_\sigma(x - z) \gtrsim \begin{cases}
            1 & \text{ if } z \in U_{i(x)} \\
            e^{ - \frac{1}{2} \left( \frac{ x - z}{ \sigma} \right)^2 } & \text{ if } \max(|x|, |z|) < \sigma A_\epsilon \\
            \phi_\sigma(2x) & \text{ if } |x| \geq \sigma A_\epsilon, |z| < \sigma A_{\epsilon}
        \end{cases}
    \end{align*}
    Similarly for the $y$ component
    \begin{align*}
        \phi_\tau(y - \alpha - \beta z) \gtrsim 
        \begin{cases}
            1 & \text{ if } z \in U_{j(y)} \\
            e^{ - \frac{1}{2} \left( \frac{ y - \alpha - \beta z}{\tau} \right)^2} & \text{ if } \max( \frac{1}{\beta} | y - \alpha|, |z|) < \tau A_\epsilon \\
            \phi_{\tau}( 2y) & \text{ if } |z| < \sigma A_\epsilon, \frac{\tau}{\beta} | y - \alpha | \geq A_\epsilon
        \end{cases}
    \end{align*}
    Using these bounds we can give lower bounds for $p_{\theta, F}$, since $F (U_i) \geq \epsilon^2$ and $F (A) > 1 - \epsilon$
    \begin{align*}
        p_{\theta, F} \gtrsim \begin{cases}
           \epsilon^2 & \text{ if } |\sigma x| \leq A_{\epsilon}, \frac{\tau}{\beta} |y - \alpha| \leq A_{\epsilon}, i(x) = j(y)\\
           \epsilon^2 e^{ - \frac{1}{2} \left( \frac{y - \alpha - \beta \frac{x}{\sigma}}{\tau} \right)^2} & \text{ if } \max( |\sigma x|, \frac{\tau}{\beta} | y - \alpha|) < A_\epsilon\\
           \epsilon^2 \phi_\sigma(2x) & \text{ if } |\sigma x| > A_\epsilon, |y - \alpha| < \frac{\beta}{\tau} A_\epsilon \\
           \epsilon^2 \phi_\tau(2y) & \text{ if } |\sigma x| < A_\epsilon, |y - \alpha| > \frac{\beta}{\tau} A_\epsilon \\
           \phi_\sigma(2x) \phi_\tau (2y ) & \text{ otherwise.}
        \end{cases}
    \end{align*}
    Together with $\sup_{x,y} p_{\theta_0, F_0} (x,y) < \infty$, this implies that
    \[
        \log \frac{p_0}{p_{\theta, F}}(x,y) \lesssim \log( \epsilon) + x^2 + y^2
    \]
    Note that by Fubini and properties of the normal distribution $ \EE_{p_{\theta_0, F_0}}[ x^i, y^j] $ exists as soon as $\EE_{Z \sim F} [Z^{i + j}]$ exists.
    These thus our bound is square integrable under our assumptions.
    Hence, by~\cite[Lemma B.2.ii]{ghosalFundamentalsNonparametricBayesian2017} with (in that lemma) $\epsilon = 0.1$ and using that
    \[
        \EE\left[ \left( \log \frac{p_{\theta_0, F_0}}{p_{\theta, F}} (x,y) \right)^2 \right] \lesssim \log(\epsilon)^2
    \]
    Thus we get that
    \[
        V_2 + \text{KL} (P_{\theta, F}; P_{\theta_0, F_0}) \lesssim \epsilon ( \log \epsilon)^2.
    \]
\end{proof}

%% file: ErrorsInVars/PriorMass.tex
\subsubsection{Prior mass conditions}

Here we prove the prior mass conditions. We follow the arguments as in the lecture notes by Botond Szabo and Van der Vaart.

\begin{lem}\label{lem:ErrorsInVarsPriorMassEpsDP}
    Assume $F (|z|^{1 + \xi}) < \infty$ for some $\xi > 0$. 
    Assume that the centre measure $\alpha$ of the Dirichlet process admits a density $g$ such that $g (z) \gtrsim |z|^{-\delta}$ for some $\delta > 0$.
    Furthermore, assume that $\pi$ admits a density which is bounded from below on an interval around $\theta_0$.
    Then there exists a $C> 0$ such that for all $\epsilon > 0$ it holds that
    \[
        \Pi_t \left(\text{KL}+V_2\left( p_{\theta, F} ; p_{\theta_0, F_0} \right) \leq \epsilon \right) \geq e^{-C \epsilon^{-\frac{1}{(1 + \xi)}} \log(\epsilon^{-1})}.
    \]
\end{lem}

\begin{proof}
    Define $A_\eta = \frac{\EE[ |z|^{1 + \xi}]^{1/ (1 + \xi)}}{ \eta^{1/ (1 + \xi)}}$.
    By Markov's inequality,
    \[
        \PP( |z| \geq A_\eta) = \PP( |z|^{1 + \xi} \geq A_\eta^{1 + \xi}) \leq \frac{ \EE[ |Z|^{1+ \xi}]}{A_\eta^{1 + \xi}} = \eta
    \]
    Then, by the triangle inequality and \Cref{lem:ErrorsInVarsParametricApprox,lem:ErrorsInVarsRestrictingAndNormalisingErrorControl,lem:ErrorsInVarsFiniteApproximation} there exists a $C > 0$ such that for all $\eta > 0$, whenever
    \begin{align*}
        3 \| \sigma^2 - \sigma_0^2 \| &\leq 2 \sigma_0 \eta \\
        \tau &= \gamma \sigma\\
        \| \alpha - \alpha_0 \| &\leq \gamma \sigma_0 \eta \\
        \| \beta - \beta_0 \| &\leq \frac{\gamma \sigma_0}{ \int |z | \dd F_0(z) } \eta,
    \end{align*}
    then there exists a discrete probability distribution $F^*$ with at most $N_\eta \sim C A_\eta \log \eta^{-1}$ support points such that
    \[
        \| p_{\theta_0, F_0} - p_{\theta, F^*} \|_1 \lesssim \eta.
    \]
    Call the support points $z_1, \dots, z_{N_\eta}$.
    Without loss of generality, we can take the $z_i$ to be $2\frac{1 + \frac{\beta_0}{\gamma}}{\sigma_0} \eta$ separated.
    Let $\left(A_i\right)_{i = 0}^{N_\eta}$ be a partition of $\RR$ such that $z_i \in A_i$ and $\textrm{Diam} A_i \sim 2\frac{1 + \frac{\beta_0}{\gamma}}{\sigma_0} \eta$.
    Define
    \begin{align*}
        \Theta_\eta \times \mathcal{F}_\eta := \{ &\| \sigma^2 - \sigma_0^2  \| \leq \frac{2}{3} \sigma_0^2 \eta,  \quad \| \alpha - \alpha_0 \| \leq \gamma \sigma_0 \eta, \\
        & \| \beta - \beta_0 \| \leq \frac{ \gamma \sigma_0}{ \int |z| \dd F_0(z) } \eta,  \sum | F(A_i) - F_0(A_i)| \leq \eta, \\
        &\min F(A_i) \geq \eta^2 \}.
    \end{align*}
    By \Cref{lem:ErrorsInVarsPartitionsCloseInTV,lem:ErrorsInVarsCloseInKL}, whenever $ (\theta, F) \in \Theta_\eta \times \mathcal{F}_\eta$, it holds that
    \[
        \text{KL} + V_2 \left( p_{\theta, F}; p_{\theta_0, F_0}\right) \lesssim \eta^2 (\log \eta^{-1})^2.
    \]
    We can now bound the prior mass of $\theta_\eta \times \mathcal{F}_\eta$.
    By the independence of the priors, this is $\pi(\Theta_\eta) \text{DP} (\mathcal{F}_\eta)$.
    We assumed that the prior $\pi$ on $\theta$ has a continuous density which is bounded from below.
    If we take a small enough perturbation, there still exists a neighbourhood around $\theta_0$ on which the density is bounded from below.
    Hence, $\pi_t (\Theta_\eta)$ is greater than $c \eta^3$ for some $c > 0$ for small enough $t$ and $\eta$.
    To bound the prior mass of $\mathcal{F}_\epsilon$ under $\textrm{DP}_t$ we use~\cite[Lemma G.13]{ghosalFundamentalsNonparametricBayesian2017}.
    By \Cref{lem:ErrorsInVarsTailCenterMeasureRemains} the tail behaviour of the centre measure remains unchanged.
    \[
        \alpha_t(A_i) \geq \textrm{Diam}(A_i) \inf_{A_i}(g_t(z)) \geq \textrm{Diam}(A_i) \inf_{A}(g_t(z)) \geq \textrm{Diam}(A_i) A_\eta^{-\delta} \gtrsim \eta A_\eta^{-\delta}
    \]
    Hence we can apply~\cite[Lemma G.13]{ghosalFundamentalsNonparametricBayesian2017} as soon as $A_\eta$ is bounded by a polynomial in $\eta$.
    This gives
    \[
        \textrm{DP}(\mathcal{F}_\eta) \gtrsim e^{- C A_\eta \log(\eta^{-1})}.
    \]
    Combining this gives
    \[
        \pi\otimes\textrm{DP}( \Theta_\eta \times \mathcal{F}_\eta) \gtrsim e^{- C'A_\eta \log(\eta^{-1})}.
    \]
    Putting $\eta^2 = \epsilon \log \epsilon^{-1}$ gives the final result.
\end{proof}

\begin{lem}\label{lem:ErrorsInVarsPriorMassEpsMFM}
    Assume $F (|z|^{1 + \xi}) < \infty$ for some $\xi > 0$. 
    Assume that the centre measure $\alpha$ of the finite mixture admits a density $g$ such that $g (z) \gtrsim |z|^{-\delta}$ for some $\delta > 0$.
    Furthermore, assume that $\pi$ admits a density which is bounded from below on an interval around $\theta_0$.
    Then there exists a $C> 0$ such that for all $\epsilon > 0$ it holds that
    \[
        \Pi_t \left(\text{KL}+V_2\left( p_{\theta, F} ; p_{\theta_0, F_0} \right) \leq \epsilon \right) \geq e^{-C \epsilon^{-\frac{1}{(1 + \xi)}} \log(\epsilon^{-1})}.
    \]
\end{lem}
 This proof works along the same lines as the proof of \Cref{lem:ErrorsInVarsPriorMassEpsDP}.

\begin{proof}
  Define $A_\eta = \frac{\EE[ |z|^{1 + \xi}]^{1/ (1 + \xi)}}{ \eta^{1/ (1 + \xi)}}$.
    By Markov's inequality,
    \[
        \PP( |z| \geq A_\eta) = \PP( |z|^{1 + \xi} \geq A_\eta^{1 + \xi}) \leq \frac{ \EE[ |Z|^{1+ \xi}]}{A_\eta^{1 + \xi}} = \eta
    \]
    Then, by the triangle inequality and \Cref{lem:ErrorsInVarsParametricApprox,lem:ErrorsInVarsRestrictingAndNormalisingErrorControl,lem:ErrorsInVarsFiniteApproximation} there exists a $C > 0$ such that for all $\eta > 0$, whenever
    \begin{align*}
        3 \| \sigma^2 - \sigma_0^2 \| &\leq 2 \sigma_0 \eta \\
        \tau &= \gamma \sigma\\
        \| \alpha - \alpha_0 \| &\leq \gamma \sigma_0 \eta \\
        \| \beta - \beta_0 \| &\leq \frac{\gamma \sigma_0}{ \int |z | \dd F_0(z) } \eta,
    \end{align*}
    then there exists a discrete probability distribution $F^*$ with at most $N_\eta \sim C A_\eta \log \eta^{-1}$ support points such that
    \[
        \| p_{\theta_0, F_0} - p_{\theta, F^*} \|_1 \lesssim \eta.
    \]
    Call the support points $z_1, \dots, z_{N_\eta}$.
    Without loss of generality, we can take the $z_i$ to be $2\frac{1 + \frac{\beta_0}{\gamma}}{\sigma_0} \eta$ separated.
   
    Define
    \begin{align*}
      \Theta_\eta \times \mathcal{F}_\eta := \{ (\theta, \sum_{j = 1}^{N_\eta} w'_j\delta{z'_j})&: \| \sigma^2 - \sigma_0^2  \| \leq \frac{2}{3} \sigma_0^2 \eta,  \quad \| \alpha - \alpha_0 \| \leq \gamma \sigma_0 \eta, \\
          & \| \beta - \beta_0 \| \leq \frac{ \gamma \sigma_0}{ \int |z| \dd F_0(z) } \eta,  \| z_j - z'_j \| < 2\frac{1 + \frac{\beta_0}{\gamma}}{\sigma_0} \eta\\ 
          &\sum_{j = 1}^ {N_\eta} | w_j - w'_j| \leq \eta, \min_{1 \leq j \leq N_\eta} w'_j \geq \eta^2 \}.
    \end{align*}
    By \Cref{lem:ErrorsInVarsPairsCloseInTV,lem:ErrorsInVarsCloseInKL}, whenever $ (\theta, F) \in \Theta_\eta \times \mathcal{F}_\eta$, it holds that
    \[
        \text{KL} + V_2 \left( p_{\theta, F}; p_{\theta_0, F_0}\right) \lesssim \eta^2 (\log \eta^{-1})^2.
    \]
    We can now bound the prior mass of $\theta_\eta \times \mathcal{F}_\eta$.
    By the independence of the priors, this is $\pi(\Theta_\eta) \text{DP} (\mathcal{F}_\eta)$.
    We assumed that the prior $\pi$ on $\theta$ has a continuous density which is bounded from below.
    If we take a small enough perturbation, there still exists a neighbourhood around $\theta_0$ on which the density is bounded from below.
    Hence, $\pi_t (\Theta_\eta)$ is greater than $c \eta^3$ for some $c > 0$ for small enough $t$ and $\eta$.
    To bound the prior mass of $\mathcal{F}_\epsilon$ under $\textrm{DP}_t$ we use~\cite[Lemma G.13]{ghosalFundamentalsNonparametricBayesian2017}.

    This gives
    \[
      \textrm{Dir}\left( \sum_{j = 1}^ {N_\eta} | w_j - w'_j| \leq \eta, \min_{1 \leq j \leq N_\eta} w'_j \geq \eta^2 \right) \gtrsim e^{-C K_\eta \log \eta^{-1}}
    \]
    Since $K \sim \textrm{geom} (\lambda)$,
    \[
      \PP(K = N_\eta) = \lambda (1 - \lambda)^{N_\eta} \gtrsim e^{-C' N_\eta} 
    \]
    
    To bound 
    \[
      G( \| z_j - z'_j \| < 2\frac{1 + \frac{\beta_0}{\gamma}}{\sigma_0} \eta )
    \]
    We use \Cref{lem:ErrorsInVarsTailCenterMeasureRemains}. Define $A_i = \{ z: \| z_j - z \| < 2\frac{1 + \frac{\beta_0}{\gamma}}{\sigma_0} \eta \}$. Then 
    \[
        \alpha_t(A_i) \geq \textrm{Diam}(A_i) \inf_{A_i}(g_t(z)) \geq \textrm{Diam}(A_i) \inf_{A}(g_t(z)) \geq \textrm{Diam}(A_i) A_\eta^{-\delta} \gtrsim \eta A_\eta^{-\delta}
    \]
    Because $A_\eta$ is a polynomial in $\eta$, we get
    \[
        \textrm{MFM}(\mathcal{F}_\eta) \gtrsim e^{- C A_\eta \log(\eta^{-1})}.
    \]
    Combining this gives
    \[
        \pi\otimes\textrm{MFM}( \Theta_\eta \times \mathcal{F}_\eta) \gtrsim e^{- C'A_\eta \log(\eta^{-1})}.
    \]
    Putting $\eta^2 = \epsilon \log \epsilon^{-1}$ gives the final result.

\end{proof}

\begin{lem}\label{lem:ErrorsInVarsTailCenterMeasureRemains}
    Suppose that $g (z) \gtrsim |z|^{-\delta}$ for all $|z|$ large enough, and that $g$ is a continuous and positive density, then also
    \[
        g_t(z) = g( \phi_t(z)) \gtrsim |z|^{-\delta}
    \]
    for all $|z|$ large enough and $g_t$ is continuous and positive.
\end{lem}

\begin{proof}
    Recall that
    \[
        \phi_t(z) = z \left(1 + \frac{ t_2(\beta + t_2)}{1 + (\beta + t_2)^2} \right)^{-1} + \frac{t_1(\beta + t_2)}{1 + (\beta + t_2)^2}
    \]
    Hence $\phi_t (z)$ is continuous and thus $g (\phi_t (z))$ is continuous and positive.
    Moreover, $\phi_t (z) = a z + b$, and $g (az + b) \gtrsim |az + b|^{-\delta} \gtrsim |z|^{-\delta}$ when $|az| - |b| $ large enough.
\end{proof}

%% file: ErrorsInVars/TechnicalLemmas.tex
\subsubsection{Technical lemmas}

\begin{lem}\label{lem:GaussianTVHellingerBounders}
    For $\mu, \nu \in \RR$ and $\sigma, \tau > 0$ we have
    \[
        \textrm{TV}\left( N(\mu, \sigma^2), N(\nu, \tau^2) \right) \leq \frac{3 \| \sigma^2 - \tau^2 \|}{2 \sigma^2} + \frac{ \| \mu - \nu \|}{\sigma}
    \]
    and
    \[
        d_H(N(\mu, \sigma^2), N(\nu, \tau^2)) \leq \frac{ \| \mu - \nu \|}{\sigma + \tau} + 2 \frac{ \| \sigma - \tau \|}{\sigma + \tau }.
    \]
\end{lem}

\begin{proof}
    See~\cite{devroyeTotalVariationDistance2018} and~\cite[Lemma 6.9]{szaboBayesianStatistics2017}.
\end{proof}

%% file: bayeseiv.bbl
\begin{thebibliography}{TDIGM19}

\bibitem[BK12]{bickelSemiparametricBernsteinMises2012}
P.~J. Bickel and B.~J.~K. Kleijn.
\newblock The semiparametric {{Bernstein}}--von mises theorem.
\newblock {\em The Annals of Statistics}, 40(1):206--237, 2012.

\bibitem[BP20]{balanTutorialFrailtyModels2020}
Theodor~A Balan and Hein Putter.
\newblock A tutorial on frailty models.
\newblock {\em Statistical Methods in Medical Research}, 29(11):3424--3454,
  November 2020.
\newblock doi: 10.1177/0962280220921889.

\bibitem[BR87]{bickelEfficientEstimationErrors1987}
P.~J. Bickel and Y.~Ritov.
\newblock Efficient estimation in the errors in variables model.
\newblock {\em The Annals of Statistics}, 15(2):513--540, 1987.

\bibitem[Cas12]{castilloSemiparametricBernsteinMises2012}
Isma{\"e}l Castillo.
\newblock A semiparametric {{Bernstein}}--von {{Mises}} theorem for
  {{Gaussian}} process priors.
\newblock {\em Probability Theory and Related Fields}, 152(1-2):53--99, 2012.

\bibitem[Cha15]{chaeSemiparametricBernsteinMisesTheorem2015}
Minwoo Chae.
\newblock The semiparametric {{Bernstein-von Mises}} theorem for models with
  symmetric error, October 2015.

\bibitem[CKK19]{chaeSemiParametricBernsteinMisesTheorem2019}
Minwoo Chae, Yongdai Kim, and Bas J.~K. Kleijn.
\newblock The {{Semi-Parametric Bernstein-Von Mises Theorem}} for {{Regression
  Models}} with {{Symmetric Errors}}.
\newblock {\em Statistica Sinica}, 29(3):1465--1487, 2019.

\bibitem[CN14]{castilloBernsteinMisesPhenomenonNonparametric2014}
Isma{\"e}l Castillo and Richard Nickl.
\newblock On the {{Bernstein-von Mises}} phenomenon for nonparametric {{Bayes}}
  procedures.
\newblock {\em The Annals of Statistics}, 42(5), October 2014.

\bibitem[CR15]{castilloBernsteinMisesTheorem2015}
Isma{\"e}l Castillo and Judith Rousseau.
\newblock A {{Bernstein}}--von {{Mises}} theorem for smooth functionals in
  semiparametric models.
\newblock {\em The Annals of Statistics}, 43(6):2353--2383, 2015.

\bibitem[DMR18]{devroyeTotalVariationDistance2018}
Luc Devroye, Abbas Mehrabian, and Tommy Reddad.
\newblock The total variation distance between high-dimensional {{Gaussians}}
  with the same mean, October 2018.

\bibitem[DS95]{dellaportasBayesianAnalysisErrorsVariables1995}
Petros Dellaportas and David~A. Stephens.
\newblock Bayesian {{Analysis}} of {{Errors-in-Variables Regression Models}}.
\newblock {\em Biometrics}, 51(3):1085--1095, 1995.

\bibitem[Fra]{franssenWhenSubjectiveObjective}
Stefan Franssen.
\newblock {\em When Is Subjective Objective Enough? {{Frequentist}} Guarantees
  for {{Bayesian}} Methods}.

\bibitem[Gv17]{ghosalFundamentalsNonparametricBayesian2017}
Subhashis Ghosal and Aad {van der Vaart}.
\newblock {\em Fundamentals of Nonparametric {{Bayesian}} Inference}, volume~44
  of {\em Cambridge Series in Statistical and Probabilistic Mathematics}.
\newblock Cambridge University Press, Cambridge, 2017.

\bibitem[KW56]{kieferConsistencyMaximumLikelihood1956}
J.~Kiefer and J.~Wolfowitz.
\newblock Consistency of the {{Maximum Likelihood Estimator}} in the
  {{Presence}} of {{Infinitely Many Incidental Parameters}}.
\newblock {\em The Annals of Mathematical Statistics}, 27(4):887--906, 1956.

\bibitem[MG96]{mallickSemiparametricErrorsvariablesModels1996}
Bani~K. Mallick and Alan~E. Gelfand.
\newblock Semiparametric errors-in-variables models {{A Bayesian}} approach.
\newblock {\em Journal of Statistical Planning and Inference}, 52(3):307--321,
  July 1996.

\bibitem[MR97]{mullerBayesianSemiparametricModel1997}
Peter M{\"u}ller and Kathryn Roeder.
\newblock A {{Bayesian}} semiparametric model for case-control studies with
  errors in variables.
\newblock {\em Biometrika}, 84(3):523--537, September 1997.

\bibitem[Mv96]{murphyLikelihoodInferenceErrorsVariables1996}
S.A. Murphy and A.W. {van der Vaart}.
\newblock Likelihood {{Inference}} in the {{Errors-in-Variables Model}}.
\newblock {\em Journal of Multivariate Analysis}, 59(1):81--108, October 1996.

\bibitem[Mv00]{murphyProfileLikelihood2000}
S.~A. Murphy and A.~W. {van der Vaart}.
\newblock On profile likelihood.
\newblock {\em Journal of the American Statistical Association},
  95(450):449--485, 2000.

\bibitem[PD06]{pennellBayesianSemiparametricDynamic2006}
Michael~L. Pennell and David~B. Dunson.
\newblock Bayesian {{Semiparametric Dynamic Frailty Models}} for {{Multiple
  Event Time Data}}.
\newblock {\em Biometrics}, 62(4):1044--1052, December 2006.

\bibitem[Rei50]{reiersolIdentifiabilityLinearRelation1950}
Olav Reiers{\o}l.
\newblock Identifiability of a linear relation between variables which are
  subject to error.
\newblock {\em Econometrica : journal of the Econometric Society}, 18:375--389,
  1950.

\bibitem[Rv20]{raySemiparametricBayesianCausal2020}
Kolyan Ray and Aad {van der Vaart}.
\newblock Semiparametric {{Bayesian}} causal inference.
\newblock {\em The Annals of Statistics}, 48(5):2999--3020, 2020.

\bibitem[Sch65]{schwartzBayesProcedures1965}
Lorraine Schwartz.
\newblock On {{Bayes}} procedures.
\newblock {\em Zeitschrift f{\"u}r Wahrscheinlichkeitstheorie und Verwandte
  Gebiete}, 4(1):10--26, March 1965.

\bibitem[Set]{sethuramanConstructiveDefinitionDirichlet}
Jayaram Sethuraman.
\newblock A {{Constructive Definition}} of the {{Dirichlet Prior}}.

\bibitem[Sv17]{szaboBayesianStatistics2017}
Botond Szabo and A.~W. {van der Vaart}.
\newblock Bayesian {{Statistics}}.
\newblock Lecture notes, January 2017.

\bibitem[TDIGM19]{tallaritaBayesianAutoregressiveFrailty2019}
Marta Tallarita, Maria De~Iorio, Alessandra Guglielmi, and James {Malone-Lee}.
\newblock Bayesian {{Autoregressive Frailty Models}} for {{Inference}} in
  {{Recurrent Events}}.
\newblock {\em The International Journal of Biostatistics}, 16(1), 2019.

\bibitem[{van}88]{vandervaartEstimatingRealParameter1988}
A.~W. {van der Vaart}.
\newblock Estimating a {{Real Parameter}} in a {{Class}} of {{Semiparametric
  Models}}.
\newblock {\em The Annals of Statistics}, 16(4):1450--1474, December 1988.

\bibitem[{van}96]{vandervaartEfficientMaximumLikelihood1996}
A.~W. {van der Vaart}.
\newblock Efficient {{Maximum Likelihood Estimation}} in {{Semiparametric
  Mixture Models}}.
\newblock {\em The Annals of Statistics}, 24(2):862--878, April 1996.

\bibitem[{van}98]{vandervaartAsymptoticStatistics1998}
A.~W. {van der Vaart}.
\newblock {\em Asymptotic {{Statistics}}}, volume~3 of {\em Cambridge Series in
  Statistical and Probabilistic Mathematics}.
\newblock Cambridge University Press, Cambridge, 1998.

\bibitem[WM97]{walkerHierarchicalGeneralizedLinear1997}
Stephen~G. Walker and Bani~K. Mallick.
\newblock Hierarchical {{Generalized Linear Models}} and {{Frailty Models}}
  with {{Bayesian Nonparametric Mixing}}.
\newblock {\em Journal of the Royal Statistical Society: Series B
  (Methodological)}, 59(4):845--860, November 1997.

\bibitem[Wv23]{wellnerWeakConvergenceEmpirical2023}
Jon Wellner and Aad {van der Vaart}.
\newblock {\em Weak {{Convergence}} and {{Empirical Processes}}}.
\newblock Springer-Verlag, second edition edition, 2023.

\end{thebibliography}
